\documentclass[11 pt]{article}
\usepackage{amsmath}
\usepackage{amsthm}
\usepackage{graphicx}
\usepackage{bbm,amssymb}
\usepackage{bbold}
\usepackage[hyperfootnotes=false]{hyperref}

\newcommand{\floor}[1]{\lfloor {#1} \rfloor}
\newcommand{\ceil}[1]{\lceil {#1} \rceil}

\newtheorem{theorem}{Theorem}[section]
\newtheorem{prop}[theorem]{Proposition}
\newtheorem{lemma}[theorem]{Lemma}
\newtheorem{corollary}[theorem]{Corollary}

\theoremstyle{remark}
\newtheorem*{remark}{Remark}

\theoremstyle{definition}

\def\liminf{\mathop{\rm lim\,inf}\limits}
\def\Leb{\mathcal{L}}
\def\normal{{\bf n}}
\def\Points{{::}}
\def\div{{\text{div}\hspace{0.5ex}}}

\title{Scaling Limits for Internal Aggregation Models with Multiple Sources}
\author{Lionel Levine\footnote{supported by an NSF Graduate Research Fellowship, and NSF 
grant DMS-0605166}~~and Yuval Peres\footnote{partially supported by NSF grant DMS-0605166} \\ University of California, Berkeley and Microsoft Research}
\date{May 1, 2009}

\DeclareSymbolFont{AMSb}{U}{msb}{m}{n}
\DeclareMathSymbol{\C}{\mathbin}{AMSb}{"43} 
\DeclareMathSymbol{\EE}{\mathbin}{AMSb}{"45} 
\DeclareMathSymbol{\N}{\mathbin}{AMSb}{"4E} 
\DeclareMathSymbol{\PP}{\mathbin}{AMSb}{"50} 
\DeclareMathSymbol{\Q}{\mathbin}{AMSb}{"51} 
\DeclareMathSymbol{\R}{\mathbin}{AMSb}{"52} 
\DeclareMathSymbol{\Z}{\mathbin}{AMSb}{"5A}

\begin{document}

\maketitle
\renewcommand{\thefootnote}{}
\footnote{{\bf\noindent Key words:} asymptotic shape, divisible sandpile, Green's function, Hele-Shaw flow, internal diffusion limited aggregation, obstacle problem, quadrature domain, rotor-router model}
\footnote{{\bf\noindent 2000
Mathematics Subject Classifications:} Primary 60G50; Secondary 35R35, 31C20}
\renewcommand{\thefootnote}{\arabic{footnote}}

\begin{abstract}
We study the scaling limits of three different aggregation models on $\Z^d$: internal DLA, in which particles perform random walks until reaching an unoccupied site; the rotor-router model, in which particles perform deterministic analogues of random walks; and the divisible sandpile, in which each site distributes its excess mass equally among its neighbors.  As the lattice spacing tends to zero, all three models are found to have the same scaling limit, which we describe as the solution to a certain PDE free boundary problem in $\R^d$.  In particular, internal DLA has a deterministic scaling limit.  We find that the scaling limits are quadrature domains, which have arisen independently in many fields such as potential theory and fluid dynamics.  Our results apply both to the case of multiple point sources and to the Diaconis-Fulton smash sum of domains.  
\end{abstract}

\pagebreak
\tableofcontents

\section{Introduction}

Given finite sets $A, B \subset \Z^d$, Diaconis and Fulton \cite{DF} defined the \emph{smash sum} $A \oplus B$ as a certain random set whose cardinality is the sum of the cardinalities of $A$ and $B$.  Write $A \cap B = \{x_1, \ldots, x_k\}$.  To construct the smash sum, begin with the union $C_0 = A \cup B$ and for each $j=1,\ldots,k$ let
	\[ C_j = C_{j-1} \cup \{y_j\} \]
where $y_j$ is the endpoint of a simple random walk started at $x_j$ and stopped on exiting~$C_{j-1}$.  Then define $A\oplus B = C_k$.  The key observation of \cite{DF} is that the law of $A\oplus B$ does not depend on the ordering of the points $x_j$.  The sum of two squares in $\Z^2$ overlapping in a smaller square is pictured in Figure~\ref{smashsumfigure}.

\begin{figure}
\centering
\includegraphics{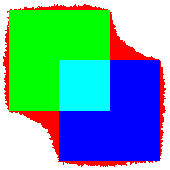}
\includegraphics{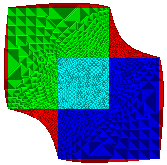}
\includegraphics{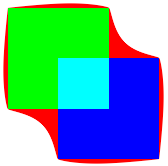}
\caption{Smash sum of two squares overlapping in a smaller square, for internal DLA (top left), the rotor-router model (top right), and the divisible sandpile.
}
\label{smashsumfigure}
\end{figure}

In Theorem~\ref{DFsum}, below, we prove that as the lattice spacing goes to zero, the smash sum $A \oplus B$ has a deterministic scaling limit in $\R^d$.  Before stating our main results, we describe some related models and discuss our technique for identifying their common scaling limit, which comes from the theory of free boundary problems in PDE.

The Diaconis-Fulton smash sum generalizes the model of \emph{internal diffusion limited aggregation} (``internal DLA'') studied in \cite{LBG}, and in fact was part of the original motivation for that paper.  In classical internal DLA, we start with $n$ particles at the origin $o \in \Z^d$
\index{$o$, origin in $\Z^d$}
and let each perform simple random walk until it reaches an unoccupied site.  The resulting random set of $n$ occupied sites in $\Z^d$ can be described as the $n$-fold smash sum of $\{o\}$ with itself.  We will use the term internal DLA to refer to the same process run from an arbitrary starting configuration of particles.
In this broader sense of the term, both the Diaconis-Fulton sum and the model studied in \cite{LBG} are particular cases of internal DLA. 

In defining the smash sum $A\oplus B$, various alternatives to random walk are possible.  Rotor-router walk is a deterministic analogue of random walk, first studied by Priezzhev et al.\ \cite{PDDK} under the name ``Eulerian walkers.''  At each site in $\Z^2$ is a {\it rotor} pointing north, south, east or west.  A particle performs a nearest-neighbor walk on the lattice according to the following rule: during each time step, the rotor at the particle's current location is rotated clockwise by $90$ degrees, and the particle takes a step in the direction of the newly rotated rotor.  In higher dimensions, the model can be defined analogously by repeatedly cycling the rotors through an ordering of the $2d$ cardinal directions in $\Z^d$.  The sum of two squares in $\Z^2$ using rotor-router walk is pictured in Figure~\ref{smashsumfigure}; all rotors began pointing west.  The shading in the figure indicates the final rotor directions, with four different shades corresponding to the four possible directions.

The {\it divisible sandpile} model uses continuous amounts of mass in place of discrete particles.  A lattice site is {\it full} if it has mass at least $1$.  Any full site can {\it topple} by keeping mass $1$ for itself and distributing the excess mass equally among its neighbors.  At each time step, we choose a full site and topple it.  As time goes to infinity, provided each full site is eventually toppled, the mass approaches a limiting distribution in which each site has mass $\leq 1$; this is proved in \cite{LP}.  Note that individual topplings do not commute.  However, the divisible sandpile is ``abelian'' in the sense that any sequence of topplings produces the same limiting mass distribution; this is proved in Lemma~\ref{abelianproperty}.  Figure~\ref{smashsumfigure} shows the limiting domain of occupied sites resulting from starting mass $1$ on each of two squares in $\Z^2$, and mass $2$ on the smaller square where they intersect.

Figure~\ref{smashsumfigure} raises a few natural questions: as the underlying lattice spacing becomes finer and finer, will the smash sum $A\oplus B$ tend to some limiting shape in $\R^d$, and if so, what is this shape?  Will it be the same limiting shape for all three models?  To see how we might identify the limiting shape, consider the divisible sandpile {\it odometer function}
	\[ u(x) = \text{total mass emitted from } x. \]
Since each neighbor $y\sim x$ emits an equal amount of mass to each of its $2d$ neighbors, the total  mass received by $x$ from its neighbors is $\frac{1}{2d} \sum_{y \sim x} u(y)$, hence
	\begin{equation} \label{netchangeinmass} \Delta u(x) = \nu(x) - \sigma(x) \end{equation}
where $\sigma(x)$ and $\nu(x)$ are the initial and final amounts of mass at $x$, respectively.  Here $\Delta$ is the discrete Laplacian in $\Z^d$, defined by $\Delta u(x) = \frac{1}{2d} \sum_{y \sim x} u(y) - u(x)$.

Equation (\ref{netchangeinmass}) suggests the following approach to finding the limiting shape.  We first construct a function on $\Z^d$ whose Laplacian is $\sigma-1$; an example is the function
	\begin{equation} \label{theobstacleintro} \gamma(x) = - |x|^2 - \sum_{y \in \Z^d} g_1(x,y) \sigma(y) \end{equation}
where in dimension $d\geq 3$ the Green's function $g_1(x,y)$ is the expected number of times a simple random walk started at $x$ visits $y$ (in dimension $d=2$ we use the recurrent potential kernel in place of the Green's function).  The sum $u+\gamma$ is then a superharmonic function on $\Z^d$; that is, $\Delta (u+\gamma) \leq 0$.  Moreover if $f \geq \gamma$ is any superharmonic function lying above $\gamma$, then $f - \gamma - u$ is superharmonic on the domain $D = \{x\in \Z^d | \nu(x)=1 \}$ of fully occupied sites, and nonnegative outside $D$, hence nonnegative everywhere.  Thus we have proved the following lemma of \cite{LP}.

\begin{lemma}
\label{discretemajorant}
Let $\sigma$ be a nonnegative function on $\Z^d$ with finite support.  Then the odometer function for the divisible sandpile started with mass $\sigma(x)$ at each site $x$ is given by
	\[ u = s - \gamma \]
where $\gamma$ is given by (\ref{theobstacleintro}), and
	\[ s(x) = \inf \{f(x) | f \text{ is superharmonic on $\Z^d$ and } f\geq \gamma \} \]
is the {\it least superharmonic majorant} of $\gamma$.
\end{lemma}

For a reformulation of this lemma as a ``least action principle,'' see Lemma~\ref{leastaction}.

\begin{figure}
\centering
\includegraphics[scale=.3]{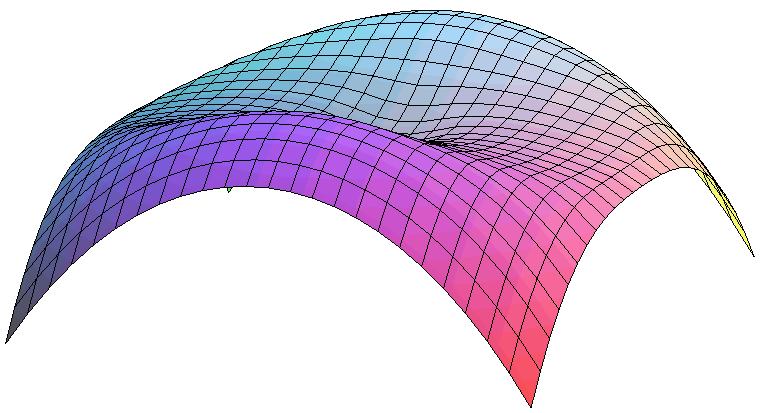} \\
\includegraphics[scale=.3]{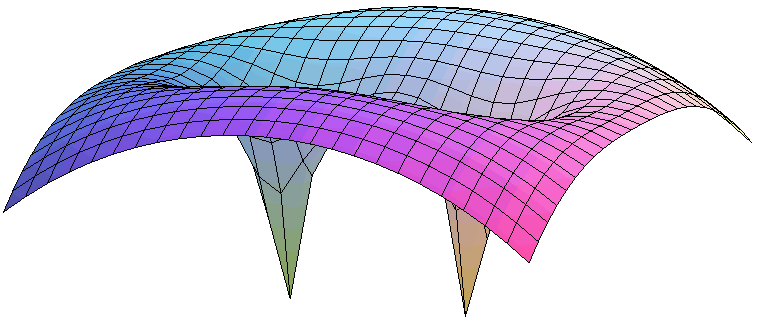}
\caption{The obstacles $\gamma$ corresponding to starting mass $1$ on each of two overlapping disks (top) and mass $100$ on each of two nonoverlapping disks.}
\end{figure}

Lemma~\ref{discretemajorant} allows us to formulate the problem in a way which translates naturally to the continuum.  Given a function $\sigma$ on $\R^d$ representing the initial mass density, by analogy with (\ref{theobstacleintro}) we define the {\it obstacle}
	\[ \gamma(x) = -|x|^2 - \int_{\R^d} g(x,y) \sigma(y) dy \]
where $g(x,y)$ is the harmonic potential on $\R^d$ proportional to $|x-y|^{2-d}$ in dimensions $d \geq 3$ and to $-\log |x-y|$ in dimension two.  We then let
	\[ s(x) = \inf \{f(x) | f \text{ is continuous, superharmonic and }  f \geq \gamma \}. \]
The {\it odometer function} for $\sigma$ is then given by $u = s - \gamma$, and the final domain of occupied sites is given by
	\begin{equation} \label{thenoncoincidencesetintro} D = \{x \in \R^d| s(x)>\gamma(x) \}. \end{equation}
This domain $D$ is called the {\it noncoincidence set for the obstacle problem} with obstacle $\gamma$; 
for an in-depth discussion of the obstacle problem, see \cite{Friedman}.

If $A,B$ are bounded open sets in $\R^d$, we define the \emph{smash sum} of $A$ and $B$ as
	\begin{equation} \label{smashsumdef} A\oplus B = A\cup B \cup D \end{equation}
where $D$ is given by (\ref{thenoncoincidencesetintro}) with $\sigma = 1_A + 1_B$.  In the two-dimensional setting, an alternative definition of the smash sum in terms of quadrature identities is mentioned in \cite{Gustafsson88}.

By analogy with the discrete case, we would expect that
	\[ \Leb(A \oplus B) = \Leb(A) + \Leb(B), \]
where $\Leb$ denotes Lebesgue measure in $\R^d$.  This is proved in Corollary~\ref{volumesadd}.  More generally, if $h$ is a superharmonic function on $\Z^d$, and $\sigma$ is a mass configuration for the divisible sandpile (so each site $x \in \Z^d$ has mass $\sigma(x)$), the sum $\sum_{x \in \Z^d} h(x) \sigma(x)$ can only decrease when we perform a toppling.  Thus
	\begin{equation} \label{discretequadrature} \sum_{x \in \Z^d} h(x) \nu(x) \leq \sum_{x \in \Z^d} h(x) \sigma(x), \end{equation}
where $\nu$ is the final mass configuration.
We therefore expect the domain $D$ given by (\ref{thenoncoincidencesetintro}) to satisfy the \emph{quadrature inequality}
	\begin{equation} \label{quadratureineq} \int_D h(x) dx \leq \int_D h(x) \sigma(x) dx \end{equation}
for all integrable superharmonic functions $h$ on $D$.  For a proof under suitable smoothness assumptions on $\sigma$ and $h$, see Proposition~\ref{boundaryregularitysmooth}; see also \cite{Sakai}.

A domain $D \subset \R^d$ satisfying an inequality of the form (\ref{quadratureineq}) is called a \emph{quadrature domain} for $\sigma$.  Such domains are widely studied in potential theory and have a variety of applications in fluid dynamics \cite{Crowdy,Richardson}.  For more on quadrature domains and their connection with the obstacle problem, see \cite{AS76,CKS,GuSh,KS,Sakai,Shahgholian}.  Equation (\ref{discretequadrature}) can be regarded as a discrete analogue of a quadrature inequality; in this sense, the three aggregation models studied in this paper produce discrete analogues of quadrature domains.  Indeed, these aggregation models can be interpreted as discrete analogues of Hele-Shaw flow \cite{Crowdy,VE}, which produces quadrature domains in the continuum.

The main goal of this paper is to prove that if any of our three aggregation models -- internal DLA, rotor-router, or divisible sandpile -- is run on finer and finer lattices with initial mass densities converging in an appropriate sense to $\sigma$, the resulting domains of occupied sites will converge in an appropriate sense to the domain $D$ given by (\ref{thenoncoincidencesetintro}).  

In order to state our main result, let us define the appropriate notion of convergence of domains, which amounts essentially to convergence in the Hausdorff metric.
Fix a sequence $\delta_n \downarrow 0$ representing the lattice spacing.  Given domains $A_n \subset \delta_n \Z^d$ and $D \subset \R^d$, write $A_n \to D$ if for any $\epsilon>0$
	\begin{equation} \label{convergenceinthehausdorffmetric} D_\epsilon \cap \delta_n \Z^d \subset A_n \subset D^{\epsilon} \end{equation}
for all sufficiently large $n$.  Here $D_\epsilon$ and $D^\epsilon$ denote respectively the inner and outer $\epsilon$-neighborhoods of $D$:
	\begin{equation} \label{neighborhoods} \begin{split} D_\epsilon &= \{x \in D \,|\, B(x,\epsilon) \subset D \} \\ 
	D^\epsilon &= \{x \in \R^d \,|\, B(x,\epsilon) \not \subset D^c \}
	\end{split} \end{equation}
where $B(x,\epsilon)$ is the ball of radius $\epsilon$ centered at $x$.

For $x \in \delta_n \Z^d$ we write $x^\Box = \left[x+\frac{\delta_n}{2}, x-\frac{\delta_n}{2}\right]^d$.  For $t \in \R$ write $\lfloor t \rceil$ for the closest integer to $t$, rounding up if $t \in \Z+\frac12$. 

The initial data for each of our lattice models consists of a function $\sigma_n : \delta_n \Z^d \to \Z_{\geq 0}$ representing the number of particles (or amount of mass) at each lattice site; we think of $\sigma_n$ as a density with respect to counting measure on $\delta_n \Z^d$, and refer to it as the ``initial density.''

Throughout this paper, to avoid trivialities we work in dimension $d \geq 2$.  Our main result is the following.
 
\begin{theorem}
\label{intromain}
Let $\Omega \subset \R^d$, $d\geq 2$ be a bounded open set, and let $\sigma : \R^d \rightarrow \Z_{\geq 0}$ be a bounded function which is continuous almost everywhere, satisfying $\{\sigma \geq 1\} = \bar{\Omega}$.  Let $D_n, R_n, I_n$ be the domains of occupied sites formed from the divisible sandpile, rotor-router model, and internal DLA, respectively, in the lattice $\delta_n \Z^d$ started from initial density 
	\[ \sigma_n(x) = \left\lfloor \delta_n^{-d} \int_{x^\Box} \sigma(y) dy \right\rceil. \]
Then as $n \uparrow \infty$
	\[ D_n, R_n \to D \cup \Omega; \]
and if $\delta_n \log n \downarrow 0$, then with probability one
	\[ I_n \to D \cup \Omega \]
where $D$ is given by (\ref{thenoncoincidencesetintro}), and the convergence is in the sense of~(\ref{convergenceinthehausdorffmetric}).
\end{theorem}

\begin{remark}
When forming the rotor-router domains $R_n$, the initial rotors in each lattice $\delta_n \Z^d$ may be chosen arbitrarily.
\end{remark}

As an immediate consequence, all three lattice models have a rotationally-invariant scaling limit.  This follows from the rotational symmetry of the obstacle problem: if $\rho$ is a rotation of $\R^d$, and $D$ is the noncoincidence set (\ref{thenoncoincidencesetintro}) for initial density $\sigma$, then the noncoincidence set for initial density $\sigma \circ \rho^{-1}$ is $\rho D$.

We prove a somewhat more general form of Theorem~\ref{intromain} which allows for some flexibility in how the discrete density $\sigma_n$ is constructed from $\sigma$.  In particular, taking $\sigma = 1_{\bar{A}} + 1_{\bar{B}}$ we obtain the following theorem, which explains the similarity of the three smash sums pictured in Figure~\ref{smashsumfigure}.

\begin{theorem}
\label{DFsum}
Let $A,B \subset \R^d$ be bounded open sets whose boundaries have measure zero.  Let $D_n, R_n, I_n$ be the smash sum of $A \cap \delta_n \Z^d$ and $B \cap \delta_n \Z^d$, formed using divisible sandpile, rotor-router and internal DLA dynamics, respectively.  Then as $n \uparrow \infty$
	\[ D_n, R_n \to A \oplus B; \]
and if $\delta_n \log n \downarrow 0$, then with probability one
	\[ I_n \to A \oplus B \]
where $A \oplus B$ is given by (\ref{smashsumdef}), and the convergence is in the sense of~(\ref{convergenceinthehausdorffmetric}).
\end{theorem}


For the divisible sandpile, Theorem~\ref{intromain} can be generalized by dropping the requirement that $\sigma$ be integer valued; see Theorem~\ref{domainconvergence} for the precise statement.  
Taking $\sigma$ real-valued is more problematic in the case of the rotor-router model and internal DLA, since these models work with discrete particles.  Still, one might wonder if, for example, given a domain $A \subset \R^d$, starting each even site in $A \cap \delta_n \Z^d$ with one particle and each odd site with two particles, the resulting domains $R_n, I_n$ would converge to the noncoincidence set $D$ for density $\sigma = \frac32 1_A$.  This is in fact the case: if $\sigma_n$ is a density on $\delta_n \Z^d$, as long as a certain ``smoothing'' of $\sigma_n$ converges to $\sigma$, the rotor-router and internal DLA domains started from initial density $\sigma_n$ will converge to $D$.  See Theorems~\ref{rotordomainconvergence} and~\ref{IDLAconvergence} for the precise statements.

\begin{figure}
\centering
\includegraphics[scale=.16]{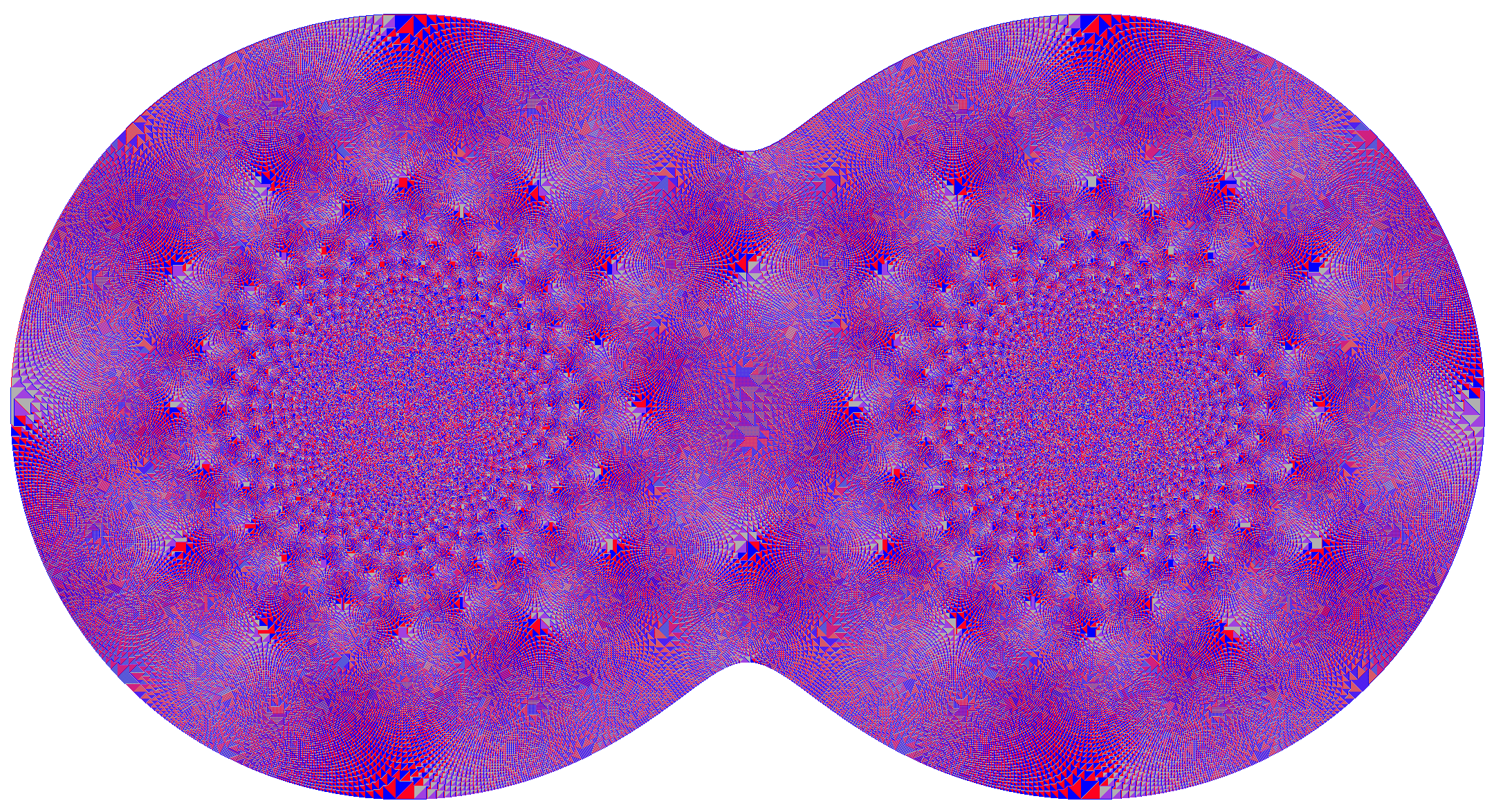}
\caption{The rotor-router model in $\Z^2$ started from two point sources on the $x$-axis.  The boundary of the limiting shape is an algebraic curve (\ref{twosourcequartic}) of degree~$4$.}
\label{rotortwosource}
\end{figure}

One interesting case not covered by Theorems~\ref{intromain} and~\ref{DFsum} is the case of multiple point sources.  Lawler, Bramson and Griffeath \cite{LBG} showed that the scaling limit of internal DLA in $\Z^d$ with a single point source of particles is a Euclidean ball.  In \cite{LP}, the present authors showed that for rotor-router aggregation and the divisible sandpile in $\Z^d$ with a single point source, the scaling limit in both cases is also a Euclidean ball.  

For $x \in \R^d$ write $x^\Points$ for the closest lattice point in $\delta_n \Z^d$, breaking ties to the right.  Our shape theorem for multiple point sources, which is deduced from Theorem~\ref{DFsum} using the main results of \cite{LBG} and \cite{LP}, is the following.

\begin{theorem}
\label{multiplepointsources}  
Fix $x_1, \ldots, x_k \in \R^d$ and $\lambda_1, \ldots, \lambda_k > 0$.  Let $B_i$ be the ball of volume $\lambda_i$ centered at $x_i$.  Fix a sequence $\delta_n \downarrow 0$, and for $x \in \delta_n \Z^d$ let
	\[ \sigma_n(x) = \left\lfloor \delta_n^{-d} \sum_{i=1}^k \lambda_i 1_{\{x = x_i^\Points\}} \right\rfloor. \]
Let $D_n,R_n,I_n$ be the domains of occupied sites in $\delta_n \Z^d$ formed from the divisible sandpile, rotor-router model, and internal DLA, respectively, started from initial density $\sigma_n$.  Then as $n \to \infty$
	\begin{equation} \label{smashsumofballs} D_n, R_n \to B_1 \oplus \ldots \oplus B_k; \end{equation}
and if $\delta_n \leq 1 / n$, then with probability one
	\[ I_n \to B_1 \oplus \ldots \oplus B_k \]
where $\oplus$ denotes the smash sum (\ref{smashsumdef}), and the convergence is in the sense of~(\ref{convergenceinthehausdorffmetric}).
\end{theorem}

Implicit in (\ref{smashsumofballs}) is the associativity of the smash sum operation, which is not readily apparent from the definition (\ref{smashsumdef}).  For a proof of associativity, see Lemma~\ref{associativity}.  Alternatively, the associativity follows from basic properties of ``partial balayage'' operators on measures on $\R^d$; see \cite[Theorem 2.2(iii)]{GuSa}.  For related results in dimension two, see \cite[Prop.~3.10]{Sakai82} and \cite[section\ 2.4]{VE}.

We remark that a similar model of internal DLA with multiple point sources was studied by Gravner and Quastel \cite{GQ}, who also obtained a variational solution.  In their model, instead of starting with a fixed number of particles, each source $x_i$ emits particles according to a Poisson process.  The shape theorems of \cite{GQ} concern convergence in the sense of volume, which is a weaker form of convergence than that studied in the present paper.

In Proposition~\ref{ballquadrature}, we show that the smash sum of balls $B_1 \oplus \ldots \oplus B_k$ arising in Theorem~\ref{multiplepointsources} obeys the classical quadrature inequality
	\begin{equation} \label{classicalquadrature} \int_{B_1 \oplus \ldots \oplus B_k} h(x) dx \leq \sum_{i=1}^k \lambda_i h(x_i) \end{equation}
for all integrable superharmonic functions $h$ on $B_1 \oplus \ldots \oplus B_k$, with equality if $h$ is harmonic.  This can be regarded as a generalization of the classical mean value property of harmonic functions, which corresponds to the case $k=1$.  

Let $B_1, \ldots, B_k$ be disks in $\R^2$ with distinct centers $x_1, \ldots, x_k$.  From (\ref{classicalquadrature}) and work of Aharonov and Shapiro \cite{AS76} and Gustafsson \cite{Gustafsson83,Gustafsson88} it follows that
 that the boundary $\partial(B_1 \oplus \ldots \oplus B_k)$ lies on an algebraic curve of degree $2k$.  More precisely, there is a polynomial $P \in \R[x,y]$ of the form
	\[ P(x,y) = \left(x^2 + y^2\right)^k + \mbox{lower order terms} \]
and there is a finite set of points $E \subset \R^2$, possibly empty, such that
	\[ \partial (B_1 \oplus \ldots \oplus B_k) = \{(x,y) \in \R^2 | P(x,y)=0\} - E. \]
	 
For example, if $B_1$ and $B_2$ are disks of equal radius $r>1$ centered at $(1,0)$ and $(-1,0)$, then $\partial (B_1 \oplus B_2)$ is given by the quartic curve \cite{Shapiro}
	\begin{equation} \label{twosourcequartic} \left(x^2+y^2\right)^2 - 2r^2 \left(x^2 + y^2\right) - 2(x^2 - y^2) = 0. \end{equation}
This curve describes the shape of the rotor-router model with two point sources pictured in Figure~\ref{rotortwosource}.

\section{Potential Theory Background}
\label{potentialtheorybackground}

In this section we review the basic properties of superharmonic potentials and of the least superharmonic majorant.  
The proofs are deferred to section~\ref{potentialtheoryproofs}.
In view of the diverse readership who might be interested in the models we study, we have elected to include somewhat more than the usual amount of background material.

\subsection{Least Superharmonic Majorant}

Since we will often be working with functions on $\R^d$ which may not be twice differentiable, it is desirable to define superharmonicity without making reference to the Laplacian.  Instead we use the mean value property.  A function $u$ on an open set $\Omega \subset \R^d$ is {\it superharmonic} if it is lower-semicontinuous and for any ball $B(x,r) \subset \Omega$
	\begin{equation} \label{meanvalueproperty} u(x) \geq A_r u (x) := \frac{1}{\omega_d r^d} \int_{B(x,r)} u(y) dy. \end{equation}
Here $\omega_d$ is the volume of the unit ball in $\R^d$.  We say that $u$ is subharmonic if $-u$ is superharmonic, and harmonic if it is both super- and subharmonic.

The following properties of superharmonic functions are well known; for proofs, see e.g.\ \cite{HFT}, \cite{Doob} or \cite{LL}.

\begin{lemma}
\label{basicproperties}
Let $u$ be a superharmonic function on an open set $\Omega \subset \R^d$ extending continuously to $\bar{\Omega}$.  Then
\begin{enumerate}
\item[(i)] $u$ attains its minimum in $\bar{\Omega}$ on the boundary.
\item[(ii)] If $h$ is continuous on $\bar{\Omega}$, harmonic on $\Omega$, and $h=u$ on $\partial \Omega$, then $u \geq h$.
\item[(iii)] If $B(x,r_0) \subset B(x,r_1) \subset \Omega$, then
	\[ A_{r_0} u (x) \geq A_{r_1} u (x). \]
\item[(iv)] If $u$ is twice differentiable on $\Omega$, then $\Delta u \leq 0$ on $\Omega$.
\item[(v)] If $B\subset \Omega$ is an open ball, and $v$ is a function on $\Omega$ which is harmonic on~$B$, continuous on $\bar{B}$, and agrees with $u$ on $B^c$, then $v$ is superharmonic.
\end{enumerate}
\end{lemma}

Given a function $\gamma$ on $\R^d$ which is bounded above, the {\it least superharmonic majorant} of $\gamma$ (also called the solution to the obstacle problem with obstacle $\gamma$) is the function
	\begin{equation} \label{majorantdef} s(x) = \inf \{f(x) | f \text{ is continuous, superharmonic and }  f \geq \gamma \}. \end{equation}
Note that since $\gamma$ is bounded above, the infimum is taken over a nonempty set.

\begin{lemma}
\label{majorantbasicprops}
Let $\gamma$ be a uniformly continuous function which is bounded above, and let $s$ be given by (\ref{majorantdef}).  Then
\begin{enumerate}
\item[(i)] $s$ is superharmonic.
\item[(ii)] $s$ is continuous.
\item[(iii)] $s$ is harmonic on the domain
	\[ D = \{x \in \R^d | s(x)>\gamma(x) \}. \]
\end{enumerate}
\end{lemma}

\subsection{Superharmonic Potentials}
Next we describe the particular class of obstacles which relate to the aggregation models we are studying.  For a bounded measurable function $\sigma$ on $\R^d$ with compact support, write
	\begin{equation} \label{thepotential} G\sigma(x) = \int_{\R^d} g(x,y) \sigma(y) dy, \end{equation}
where
\begin{equation} \label{greenskernel} g(x,y) = \begin{cases} -\frac{2}{\pi} \log |x-y|, & d=2; \\ a_d |x-y|^{2-d}, & d\geq 3. \end{cases} \end{equation}
Here $a_d = \frac{2}{(d-2)\omega_d}$, where $\omega_d$ is the volume of the unit ball in $\R^d$.  Note that (\ref{greenskernel}) differs by a factor of $2d$ from the standard harmonic potential in $\R^d$; however, the normalization we have chosen is most convenient when working with the discrete Laplacian and random walk.

The following result is standard; see \cite[Theorem 1.I.7.2]{Doob}.
						 
\begin{lemma} Let $\sigma$ be a measurable function on $\R^d$ with compact support.
\label{smoothnessofpotential}
\begin{enumerate}
\item[(i)] If $\sigma$ is bounded, then $G\sigma$ is continuous.
\item[(ii)] If $\sigma$ is $C^1$, then $G\sigma$ is $C^2$ and
	\begin{equation} \label{laplacianofpotential} \Delta G \sigma = -2d\sigma. \end{equation}
\end{enumerate}
\end{lemma}
 
Regarding (ii), if we remove the smoothness assumption on $\sigma$, equation (\ref{laplacianofpotential}) remains true in the sense of distributions.
For our applications, however, we will not need this, and the following lemma will suffice.

\begin{lemma}
\label{superharmonicpotential}
Let $\sigma$ be a bounded measurable function on $\R^d$ with compact support.
If $\sigma \geq 0$ on an open set $\Omega \subset \R^d$, then $G \sigma$ is superharmonic on $\Omega$.  
\end{lemma}

\noindent By applying Lemma~\ref{superharmonicpotential} both to $\sigma$ and to $-\sigma$, we obtain that $G\sigma$ is harmonic off the support of $\sigma$.

Let $B=B(o,r)$ be the ball of radius $r$ centered at the origin in $\R^d$.  We compute in dimensions $d\geq 3$
	\begin{equation} \label{ballpotential} G1_B(x) = \begin{cases} \frac{dr^2}{d-2} - |x|^2, & |x|<r \\
						  \frac{2 r^2}{d-2} \Big( \frac{r}{|x|} \Big)^{d-2}, & |x|\geq r. \end{cases} 		\end{equation}
Likewise in two dimensions
	\begin{equation} \label{ballpotentialdim2} G1_B(x) = \begin{cases} r^2(1-2\log r)-|x|^2, & |x|<r \\
						 -2r^2 \log |x|, & |x| \geq r. \end{cases}
						 \end{equation}
						 
Fix a bounded nonnegative function $\sigma$ on $\R^d$ with compact support, and let
	\begin{equation} \label{theobstacle} \gamma(x) = -|x|^2 - G\sigma(x). \end{equation}
Let
	\begin{equation} \label{themajorant} s(x) = \inf \{f(x) | f \text{ is continuous, superharmonic and }  f \geq \gamma \} \end{equation}
be the least superharmonic majorant of $\gamma$, and let
	\begin{equation} \label{thenoncoincidenceset} D = \{x\in \R^d| s(x)>\gamma(x)\} \end{equation}
be the noncoincidence set.

\begin{lemma}
\label{laplacianofobstacle}
\begin{itemize}
\item[(i)] $\gamma(x) + |x|^2$ is subharmonic on $\R^d$. 
\item[(ii)] If $\sigma \leq M$ on an open set $\Omega \subset \R^d$, then $\gamma(x) - (M-1)|x|^2$ is superharmonic on $\Omega$.
\end{itemize}
\end{lemma}

\begin{lemma}
\label{laplacianofodometer}
Let $u=s-\gamma$, where $\gamma$ and $s$ are given by (\ref{theobstacle}) and (\ref{themajorant}).  Then
\begin{itemize}
\item[(i)] $u(x) - |x|^2$ is superharmonic on $\R^d$.
\item[(ii)] If $\sigma \leq M$ on an open set $\Omega \subset \R^d$, then $u(x) + M|x|^2$ is subharmonic on $\Omega$.
\end{itemize}
\end{lemma}

\begin{lemma}
\label{startingdensitygreaterthan1}
Let $D$ be given by (\ref{thenoncoincidenceset}).  Then $\{\sigma>1\}^o \subset D$.
\end{lemma}

The next lemma concerns the monotonicity of our model: by starting with more mass, we obtain a larger odometer and a larger noncoincidence set; see also \cite[Lemma~7.1]{Sakai}.

\begin{lemma}
\label{monotonicity}
Let $\sigma_1 \leq \sigma_2$ be functions on $\R^d$ with compact support, and suppose that $\int_{\R^d} \sigma_2(x)dx < \infty$.  Let
	\[ \gamma_i(x) = -|x|^2 - G\sigma_i(x), \qquad i=1,2. \]
Let $s_i$ be the least superharmonic majorant of $\gamma_i$, let $u_i = s_i-\gamma_i$, and let
	\[ D_i = \{ x | s_i(x)>\gamma_i(x) \}. \]
Then $u_1 \leq u_2$ and $D_1 \subset D_2$.
\end{lemma}

Our next lemma shows that we can restrict to a domain $\Omega \subset \R^d$ when taking the least superharmonic majorant, provided that $\Omega$ contains the noncoincidence set.

\begin{lemma}
\label{majorantonacompactset}
Let $\gamma,s,D$ be given by (\ref{theobstacle})-(\ref{thenoncoincidenceset}).  Let $\Omega \subset \R^d$ be an open set with $D \subset \Omega$.  Then 
	\begin{equation} \label{majorantinadomain} s(x) = \inf \{f(x)|\text{$f$ is superharmonic on $\Omega$, continuous, and $f \geq \gamma$}\}. \end{equation}
\end{lemma}

\subsection{Boundary Regularity for the Obstacle Problem}

Next we turn to the regularity of the solution to the obstacle problem (\ref{themajorant}) and of the free boundary $\partial D$.  There is a substantial literature on boundary regularity for the obstacle problem. In our setting, however, extra care must be taken near points where $\sigma(x)=1$: at these points the obstacle (\ref{theobstacle}) is harmonic, and the free boundary can be badly behaved.  

One approach, adopted in \cite{KM}, is to reformulate the obstacle problem in~$\R^d$ as an obstacle problem on the set $\{\sigma > 1\}$.  The results of \cite{KM} apply in the situation that $\sigma \geq 1$ everywhere, but we will want to relax this condition.
We show only the minimal amount of regularity required for the proofs of our main theorems.  Much stronger regularity results are known in related settings; see, for example, \cite{Caffarelli, CKS}.

The following lemma shows that if the obstacle is sufficiently smooth, then the superharmonic majorant cannot separate too quickly from the obstacle near the boundary of the noncoincidence set.   As usual, we write 
	\begin{equation} \label{innerepsilonneighborhood} D_\epsilon = \{x \in D \,|\, B(x,\epsilon) \subset D \} \end{equation}
for the inner $\epsilon$-neighborhood of $D$.

\begin{lemma}
\label{smoothdensity}
Let $\sigma$ be a $C^1$ function on $\R^d$ with compact support.  Let $\gamma,s,D$ be given by (\ref{theobstacle})-(\ref{thenoncoincidenceset}), and write $u = s-\gamma$.  Then $u$ is $C^1$, and for $y \in \partial D_\epsilon$ we have $|\nabla u (y)| \leq C_0 \epsilon$, for a constant $C_0$ depending on $\sigma$.
\end{lemma}

Next we show that mass is conserved in our model: the amount of mass starting in $D$ is $\int_D \sigma(x) dx$, while the amount of mass ending in $D$ is $\Leb(D)$, the Lebesgue measure of $D$.  Since no mass moves outside of $D$, we expect these two quantities to be equal.  Although this seems intuitively obvious, the proof takes some work because we have no {\it a priori} control of the domain~$D$.  In particular, we first need to show that the boundary of $D$ cannot have positive $d$-dimensional Lebesgue measure.

\begin{prop}
\label{boundaryregularitysmooth}
Let $\sigma$ be a $C^1$ function on $\R^d$ with compact support, such that $\Leb(\sigma^{-1}(1))=0$. Let $D$ be given by (\ref{thenoncoincidenceset}).  Then 
\begin{itemize}
\item[(i)] $\Leb (\partial D) = 0$.
\item[(ii)] For any function $h \in C^1(\bar{D})$ which is superharmonic on $D$,
	\[  \int_D h(x) dx \leq \int_D h(x) \sigma(x) dx. \]
\end{itemize}
\end{prop}

Note that by applying (ii) both to $h$ and to $-h$, equality holds whenever $h$ is harmonic on $D$.  In particular, taking $h\equiv 1$ yields the conservation of mass: $\int_D \sigma(x)dx = \Leb(D)$.

We will also need a version of Proposition~\ref{boundaryregularitysmooth} which does not assume that $\sigma$ is $C^1$ or even continuous.  We can replace the $C^1$ assumption and the condition that $\Leb(\sigma^{-1}(1))=0$ by the following condition.
\begin{equation} \label{boundedawayfrom1again} 
	\text{For all }x\in \R^d \text{ either }\sigma(x) \leq \lambda \text{ or } \sigma(x)\geq 1
	\end{equation}
for a constant $\lambda<1$.  Then we have the following result.

\begin{prop}
\label{boundaryregularity}
Let $\sigma$ be a bounded function on $\R^d$ with compact support.  Let $D$ be given by (\ref{thenoncoincidenceset}), and let $\widetilde{D} = D \cup \{\sigma \geq 1\}^o$.  If $\sigma$ is continuous almost everywhere and satisfies (\ref{boundedawayfrom1again}), then
\begin{itemize}
\item[(i)] $\Leb \big(\partial \widetilde{D}\big) = 0$.
\item[(ii)] $\Leb \big(D\big) = \int_D \sigma(x) dx$.
\item[(iii)] $\Leb \big(\widetilde{D}\big) = \int_{\widetilde{D}} \sigma(x) dx$.
\end{itemize}
\end{prop}

In particular, taking $\sigma = 1_A + 1_B$, we have 
	\[ A \oplus B = A \cup B \cup D = \widetilde{D} \]
where $\oplus$ denotes the smash sum (\ref{smashsumdef}).
From (iii) we obtain the volume additivity of the smash sum.

\begin{corollary}
\label{volumesadd}
Let $A,B \subset \R^d$ be bounded open sets whose boundaries have measure zero.  Then
	\[ \Leb(A \oplus B) = \Leb(A) + \Leb(B). \]
\end{corollary}

Our next lemma describes the domain resulting from an initial mass density $m>1$ on a ball in $\R^d$.  Not surprisingly, the result is another ball, concentric with the original, of $m$ times the volume.  In particular, if $m$ is an integer, the $m$-fold smash sum of a ball with itself is again a ball.
						 
\begin{lemma}
\label{relaxingaball}
Fix $m>1$, and let $D$ be given by (\ref{thenoncoincidenceset}) with $\sigma = m1_{B(o,r)}$.  Then $D = B(o,m^{1/d}r)$.
\end{lemma}

\begin{proof}
Since $\gamma(x)=-|x|^2-G\sigma(x)$ is spherically symmetric and the least superharmonic majorant commutes with rotations, $D$ is a ball centered at the origin.  By Proposition~\ref{boundaryregularity}(ii) we have $\Leb(D) = m \Leb(B(o,r))$.
\end{proof}

Next we show that the noncoincidence set is bounded; see also \cite[Lemma~7.2]{Sakai}.

\begin{lemma}
\label{occupieddomainisbounded}
Let $\sigma$ be a function on $\R^d$ with compact support, satisfying $0 \leq \sigma \leq M$.  Let $D$ be given by (\ref{thenoncoincidenceset}).  Then $D$ is bounded.
\end{lemma}

\begin{proof}
Let $B=B(o,r)$ be a ball containing the support of $\sigma$.  By Lemmas~\ref{monotonicity} and~\ref{relaxingaball}, the difference $s-\gamma$ is supported in $B(o,M^{1/d}r)$.
\end{proof}

\subsection{Convergence of Obstacles, Majorants and Domains}

In broad outline, many of our arguments have the following basic structure:
	\begin{align*} \text{convergence of densities} &\implies \text{convergence of obstacles} \\
			&\implies \text{convergence of majorants} \\
			&\implies \text{convergence of domains.} \end{align*}
An appropriate convergence of starting densities is built into the hypotheses of the theorems.  From these hypotheses we use Green's function estimates to deduce the relevant convergence of obstacles.  Next, 
properties of the obstacle can often be parlayed into corresponding properties of the least superharmonic majorant.  Finally, deducing convergence of domains (i.e., noncoincidence sets) from the convergence of the majorants often requires additional regularity assumptions.  The following lemma illustrates this basic three-step approach.  

\begin{lemma}
\label{threesteps}
Let $\sigma$ and $\sigma_n$, $n=1,2,\dots$ be densities on $\R^d$ satisfying
	\[ 0 \leq \sigma, \sigma_n \leq M 1_B \]
for a constant $M$ and ball $B=B(o,R)$.  If $\sigma_n \to \sigma$ in $\L^1$, then

\begin{enumerate}
\item[(i)] $G \sigma_n \rightarrow G \sigma$ uniformly on compact subsets of $\R^d$.
\item[(ii)] $s_n \rightarrow s$ uniformly on compact subsets of $\R^d$, where $s,s_n$ are the least superharmonic majorants of the functions $\gamma = -|x|^2 - G\sigma$, and $\gamma_n = -|x|^2 - G\sigma_n$, respectively.
\item[(iii)] For any $\epsilon>0$ we have $D_\epsilon \subset D_{(n)}$ for all sufficiently large $n$, where $D,D_{(n)}$ are the noncoincidence sets $\{s>\gamma\}$ and $\{s_n>\gamma_n\}$, respectively.
\end{enumerate}
\end{lemma}

According to the following lemma, a nonnegative function on $\Z^d$ with bounded Laplacian can grow at most quadratically. 

\begin{lemma}
\label{atmostquadratic}
Fix $0<\beta<1$.  There is a constant $c_\beta$ such that any nonnegative function $f$ on $\Z^d$ with $f(o)=0$ and $|\Delta f| \leq \lambda$ in $B(o,R)$ satisfies
	\[ f(x) \leq c_\beta \lambda |x|^2, \qquad  x \in B(o,\beta R). \]
\end{lemma}

We will also use the following continuous version of Lemma~\ref{atmostquadratic}.
 
\begin{lemma}
\label{continuumatmostquadratic}
Fix $0<\beta<1$.  There is a constant $c'_\beta$ such that the following holds.  Let $f$ be a nonnegative function on $\R^d$ with $f(o)=0$, and let
	\[ f_\pm(x) = f(x) \pm \lambda \frac{|x|^2}{2d}. \]
If $f_+$ is subharmonic and $f_-$ is superharmonic in $B(o,R)$, then
	\[ f(x) \leq c'_\beta \lambda |x|^2, \qquad x \in B(o,\beta R). \]
\end{lemma}

In the following lemma, $A_\epsilon$ denotes the inner $\epsilon$-neighborhood of $A$, as defined by (\ref{innerepsilonneighborhood}).

\begin{lemma}
\label{pointset}
Let $A,B \subset \R^d$ be bounded open sets.  For any $\epsilon>0$ there exists $\eta>0$ with
	\[ (A \cup B)_\epsilon \subset A_\eta \cup B_\eta. \]
\end{lemma}

\subsection{Discrete Potential Theory}

\begin{table}
\label{boxandpoints}
\centering
\begin{tabular}{r | ll}
~ & $\delta_n\Z^d$ & $\R^d$ \\
\hline 
points & $x^\Points = \Big(x+\left( -\frac{\delta_n}{2}, \frac{\delta_n}{2} \right]^d \Big) \cap \delta_n \Z^d$ & $x^\Box = x + \left[ -\frac{\delta_n}{2}, \frac{\delta_n}{2} \right]^d$ \\ [0.5ex]
sets & $A^\Points = A \cap \delta_n\Z^d$ & $A^\Box = A + \left[ \frac{-\delta_n}{2}, \frac{\delta_n}{2} \right]^d$ \\ [0.5ex]
functions & $f^\Points = f|_{\delta_n \Z^d}$ & $f^\Box(x) = f(x^\Points)$ \\ [0.5ex]
\end{tabular}
\caption{Notation for transitioning between Euclidean space and the lattice.}
\end{table}

Fix a sequence $\delta_n \downarrow 0$ with $\delta_1 =1$.  In this section we relate discrete superharmonic potentials in the lattice $\delta_n \Z^d$ to their continuous counterparts in~$\R^d$.
If $A$ is a domain in $\delta_n \Z^d$, write	
	\[ A^\Box = A + \left[ -\frac{\delta_n}{2}, \frac{\delta_n}{2} \right]^d \]
for the corresponding union of cubes in $\R^d$.    If $A$ is a domain in $\R^d$, write $A^\Points = A \cap \delta_n \Z^d$.  Given $x \in \R^d$, write
	\[ x^\Points = \left(x - \frac{\delta_n}{2}, x+ \frac{\delta_n}{2} \right]^d \cap \delta_n \Z^d \]
for the closest lattice point to $x$, breaking ties to the right.  For a function $f$ on $\delta_n \Z^d$, write
	\[ f^\Box (x) = f(x^\Points) \]
for the corresponding step function on $\R^d$.  Likewise, for a function $f$ on $\R^d$, write $f^\Points = f|_{\delta_n \Z^d}$.  These notations are summarized in Table~1.

We define the discrete Laplacian of a function $f$ on $\delta_n \Z^d$ to be
	\[ \Delta f(x) = \delta_n^{-2} \left( \frac{1}{2d} \sum_{y\sim x} f(y)-f(x) \right). \]
According to the following lemma, if $f$ is defined on $\R^d$ and is sufficiently smooth, then its discrete Laplacian on $\delta_n \Z^d$ approximates its Laplacian on~$\R^d$.

\begin{lemma} \label{thirdderiv}
If $f : \R^d \to \R$ has continuous third derivative in a $\delta_n$-neighborhood of $x \in \delta_n \Z^d$, and $A$ is the maximum pure third partial of $f$ in this neighborhood, then
	\[ |\Delta f (x) - 2d \Delta f^\Points (x)| \leq \frac{d}{3} A \delta_n. \]
\end{lemma}

In three and higher dimensions, for $x,y \in \delta_n\Z^d$ write 
	\begin{equation} \label{definitionofg_n} g_n(x,y) = \delta_n^{2-d} g_1(\delta_n^{-1}x,\delta_n^{-1}y), \qquad d\geq 3, \end{equation}
where 
	\[ g_1(x,y) = \EE_x \# \{k|X_k=y\} \] 
is the Green's function for simple random walk on $\Z^d$.  The scaling in (\ref{definitionofg_n}) is chosen so that $\Delta_x g_n(x,y) = -\delta_n^{-d} 1_{\{x=y\}}$.  In two dimensions, write
	\begin{equation} \label{definitionofg_ndimension2} g_n(x,y) = -a(\delta_n^{-1}x,\delta_n^{-1}y) + \frac{2}{\pi} \log \delta_n, \qquad d=2, \end{equation}
where
	\[ a(x,y) = \lim_{m \to \infty} \big( \EE_x \# \{k\leq m | X_k=x\} - \EE_x \# \{k\leq m | X_k=y \} \big) \] 
is the recurrent potential kernel on $\Z^2$.

\begin{lemma}
\label{greensfunctionconvergence}
In all dimensions $d\geq 2$,
 	 \[ g_n(x,y) = g(x,y) + O(\delta_n^2 |x-y|^{-d}) \]
where $g$ is given by (\ref{greenskernel}). 
\end{lemma}

Our next result adapts Lemma~\ref{threesteps}(i) to the discrete setting.
We list here our standing assumptions on the starting densities.  Let $\sigma$ be a function on $\R^d$ with compact support, such that
 	\begin{equation} \label{absolutebound} 0 \leq \sigma \leq M \end{equation}
for some absolute constant $M$.  Suppose that $\sigma$ satisfies
	\begin{equation} \label{discontmeasurezero}  \Leb \left(DC(\sigma)\right)=0 \end{equation}
where $DC(\sigma)$ denotes the set of points in $\R^d$ where $\sigma$ is discontinuous.

For $n=1,2,\ldots$ let $\sigma_n$ be a function on $\delta_n \Z^d$ satisfying
	\begin{equation} \label{discreteabsolutebound} 0 \leq \sigma_n \leq M \end{equation}
and
\begin{equation} \label{convergingdensities} \sigma_n^\Box(x)\rightarrow \sigma(x), 
	\qquad x\notin DC(\sigma). \end{equation}
Finally, suppose that
	\begin{equation} \label{uniformcompactsupport} \text{There is a ball $B\subset \R^d$ containing the supports of $\sigma$ and $\sigma_n^\Box$ for all $n$.} \end{equation}
Although for maximum generality we have chosen to state our hypotheses on $\sigma$ and $\sigma_n$ separately, we remark that the above hypotheses on $\sigma_n$ are satisfied in the particular case when $\sigma_n$ is given by averaging $\sigma$ over a small box:
	\begin{equation} \label{averageinabox} \sigma_n(x) = \delta_n^{-d} \int_{x^\Box} \sigma(y)dy. \end{equation}

In parallel to (\ref{thepotential}), for $x \in \delta_n \Z^d$ write
	\begin{equation} \label{thediscretepotential} G_n \sigma_n (x) = \delta_n^{d} \sum_{y \in \delta_n \Z^d} g_n(x,y) \sigma_n(y), \end{equation}
where $g_n$ is given by (\ref{definitionofg_n}) and (\ref{definitionofg_ndimension2}).  Since $g_n(x,\cdot)$ is discrete harmonic except at~$x$, where it has discrete Laplacian $-\delta_n^{-d}$, we have for any function $f$ on~$\delta_n \Z^d$ with bounded support
	\begin{equation} \label{inverseoflaplacian} G_n \Delta f = \Delta G_n f = -f. \end{equation}
The first equality can be seen directly from the symmetry $g_n(x,y)=g_n(y,x)$ of the Green's function, or else by observing that the operators $\Delta$ and $G_n$ are both defined by convolutions, so they commute.

\begin{lemma}
\label{greensintegralconvergencecont}
If $\sigma$, $\sigma_n$ satisfy (\ref{absolutebound})-(\ref{uniformcompactsupport}), then
	\[ (G_n \sigma_n)^\Box \rightarrow G \sigma \]
uniformly on compact subsets of $\R^d$. 
\end{lemma}

Our next result adapts Lemma~\ref{majorantonacompactset} to the discrete setting.  Let $\sigma_n$ be a function on $\delta_n \Z^d$ with finite support, and let $\gamma_n$ be the function on $\delta_n \Z^d$ defined by
	\begin{equation} \label{thediscreteobstacle} \gamma_n(x) = -|x|^2 - G_n \sigma_n(x). \end{equation}
Let 
	\begin{equation} \label{thediscretemajorant} s_n = \inf \{f(x)|\text{$f$ is superharmonic on $\delta_n \Z^d$ and $f \geq \gamma_n$}\} \end{equation}
be the discrete least superharmonic majorant of $\gamma_n$, and let
	\begin{equation} \label{thediscretenoncoincidenceset} D_n = \{x \in \delta_n \Z^d| s_n(x)>\gamma_n(x)\}. \end{equation}

\begin{lemma}
\label{discretemajorantonacompactset}
Let $\gamma_n,s_n,D_n$ be given by (\ref{thediscreteobstacle})-(\ref{thediscretenoncoincidenceset}).  If $\Omega \subset \delta_n \Z^d$ satisfies $D_n \subset \Omega$, then 
	\begin{equation} \label{discretemajorantinadomain} s_n(x) = \inf \{f(x)|\text{\em $f$ is superharmonic on $\Omega$ and $f \geq \gamma_n$}\}. \end{equation}
\end{lemma}

\section{Divisible Sandpile}

\subsection{Abelian Property and Least Action Principle}
It will be useful here to prove the abelian property for the divisible sandpile in somewhat greater generality than that of \cite{LP}.  We work in continuous time.  Fix $\tau>0$, and let $T : [0,\tau] \rightarrow \Z^d$ be a function having only finitely many discontinuities in the interval $[0,t]$ for every $t<\tau$.  The value $T(t)$ represents the site being toppled at time $t$.  The odometer function at time $t$ is given by
	\[ u_t(x) = \Leb \left(T^{-1}(x) \cap [0,t]\right), \]
where $\Leb$ denotes Lebesgue measure.  We will say that $T$ is a {\it legal toppling function} for an initial configuration $\eta_0$ if for every $0 \leq t \leq \tau$
	\[ \eta_t(T(t)) \geq 1 \]
where
	\[ \eta_t(x) = \eta_0(x) + \Delta u_t(x) \]
is the amount of mass present at $x$ at time $t$.  If in addition $\eta_\tau \leq 1$, we say that $T$ is {\it complete}.  

\begin{lemma}
\label{abelianproperty}
(Abelian Property)
If $T_1 : [0,\tau_1] \rightarrow \Z^d$ and  $T_2 : [0,\tau_2] \rightarrow \Z^d$ are complete legal toppling functions for an initial configuration $\eta_0$, then for any $x \in \Z^d$
	\[ \Leb \left(T_1^{-1}(x)\right) = \Leb\left(T_2^{-1}(x)\right). \]
In particular, $\tau_1 = \tau_2$ and the final configurations $\eta_{\tau_1}$, $\eta_{\tau_2}$ are identical.
\end{lemma}

\begin{proof}
For $i=1,2$ write
	\begin{align*} 
	&u^i_t(x) = \Leb \left( T_i^{-1}(x) \cap [0,t] \right); \\
	&\eta^i_t(x) = \eta_0(x) + \Delta u^i_t(x).
	\end{align*}
Write $u^i = u^i_{\tau_i}$ and $\eta^i = \eta^i_{\tau_i}$.
Let $t(0)=0$ and let $t(1) < t(2) < \ldots$ be the points of discontinuity for $T_1$.  Let $x_k$ be the value of $T_1$ on the interval $(t(k-1),t(k))$.  We will show by induction on $k$ that 
	\begin{equation} \label{inductivehypothesis} u^2(x_k) \geq u^1_{t(k)}(x_k). \end{equation}
Note that for any $x \neq x_k$, if $u^1_{t(k)}(x)>0$, then letting $j<k$ be maximal such that $x_j=x$, since $T_1 \neq x$ on the interval $(t(j),t(k))$, it follows from the inductive hypothesis (\ref{inductivehypothesis}) that
	\begin{equation} \label{strongerformofindhyp} u^2(x) = u^2(x_j) \geq u^1_{t(j)}(x_j) = u^1_{t(k)}(x). \end{equation}
Since $T_1$ is legal and $T_2$ is complete, we have
	\[ \eta^2(x_k) \leq 1 \leq \eta^1_{t(k)}(x_k) \]
hence
	\[ \Delta u^2(x_k) \leq \Delta u^1_{t(k)}(x_k). \]
Since $T_1$ is constant on the interval $(t(k-1),t(k))$ we obtain
	\[ u^2(x_k) \geq u^1_{t(k)}(x_k) + \frac{1}{2d} \sum_{x \sim x_k} (u^2(x)-u^1_{t(k-1)}(x)). \]
By (\ref{strongerformofindhyp}), each term in the sum on the right side is nonnegative, completing the inductive step.

Since $t(k) \uparrow \tau_1$ as $k \uparrow \infty$, the right side of (\ref{strongerformofindhyp}) converges to $u^1(x)$ as $k \uparrow \infty$, hence $u^2 \geq u^1$.  After interchanging the roles of $T_1$ and $T_2$, the result follows.  
\end{proof}

The next lemma is a simple reformulation of Lemma~\ref{discretemajorant}, but is sometimes more convenient to use.

\begin{lemma} (Least Action Principle)
\label{leastaction}
Let $\sigma$ be a nonnegative function on~$\Z^d$ with finite support, and let $u$ be its divisible sandpile odometer function.  If $u_1 : \Z^d \to \R_{\geq 0}$ satisfies
	\begin{equation} \label{stabiliizingodom} \sigma + \Delta u_1 \leq 1 \end{equation}
then $u_1 \geq u$.
\end{lemma}

\begin{proof}
From Lemma~\ref{discretemajorant} we have $u=s-\gamma$.  Since
	\[ \Delta u_1 \leq 1 - \sigma = -\Delta \gamma \]
the sum $u_1 + \gamma$ is superharmonic on $\Z^d$.  Since $u_1$ is nonnegative, it follows that $u_1 + \gamma \geq s$, hence $u_1 \geq u$.
\end{proof}

Note that $u_1$ need not have finite support.  The moniker ``least action principle'' comes from the following interpretation.  Starting from $\sigma$, and toppling each site $x \in \Z^d$ for time $u_1(x)$ without regard to whether these topplings are legal, equation (\ref{stabiliizingodom}) says that the resulting configuration is stable: no site has mass greater than~$1$.  Clearly the odometer~$u$ is such a stabilizing function.  The Least Action Principle says that $u$ is (pointwise) minimal among all such functions.

\subsection{Convergence of Odometers}
	
By the odometer function for the divisible sandpile on $\delta_n \Z^d$ with initial density $\sigma_n$, we will mean the function
	\[ u_n(x) = \delta_n^2 \cdot \text{total mass emitted from $x$ if each site $y$ starts with mass } \sigma_n(y). \]

\begin{theorem}
\label{odomconvergence}
Let $u_n$ be the odometer function for the divisible sandpile on $\delta_n \Z^d$ with initial density $\sigma_n$.  If $\sigma$, $\sigma_n$ satisfy (\ref{absolutebound})-(\ref{uniformcompactsupport}), then 		\[ u_n^\Box \to s-\gamma \qquad \text{\em uniformly,} \]
where
	\[ \gamma(x) = -|x|^2 - G \sigma(x), \]
$G\sigma$ is given by (\ref{thepotential}), and $s$ is the least superharmonic majorant of $\gamma$.
\end{theorem}

\begin{lemma}
\label{bigballs}
Let $D_n$ be the set of fully occupied sites for the divisible sandpile in $\delta_n \Z^d$ started from initial density $\sigma_n$.  There is a ball $\Omega \subset \R^d$ with 
	\[ \bigcup_{n\geq 1} D_n^\Box \cup D \subset \Omega. \]
\end{lemma}

\begin{proof}
Let $A_n$ be the set of fully occupied sites for the divisible sandpile in $\delta_n \Z^d$ started from initial density $\tau(x) = M 1_{x \in B}$, where $B$ is given by (\ref{uniformcompactsupport}).  From the abelian property, Lemma~\ref{abelianproperty}, we have $D_n \subset A_n$.  By the inner bound of \cite[Theorem 1.3]{LP} if we start with mass $m=2\delta_n^{-d} \Leb(B)$ at the origin in $\delta_n \Z^d$, the resulting set of fully occupied sites contains $B^\Points$; by the abelian property it follows that if we start with mass $Mm$ at the origin in $\delta_n \Z^d$, the resulting set $\Omega_n$ of fully occupied sites contains $A_n$.  By the outer bound of \cite[Theorem 1.3]{LP}, $\Omega_n$ is contained in a ball $\Omega$ of volume $3M \Leb(B)$.  By Lemma~\ref{occupieddomainisbounded} we can enlarge $\Omega$ if necessary to contain $D$.
\end{proof}

For $x \in \delta_n \Z^d$ write
	\[ \gamma_n(x) = -|x|^2-G_n\sigma_n(x), \]
where $G_n$ is defined by (\ref{thediscretepotential}).  Denote by $s_n$ the least superharmonic majorant of $\gamma_n$ in the lattice $\delta_n \Z^d$.

\begin{lemma}
\label{gammaupperbound}
Let $\Omega$ be as in Lemma~\ref{bigballs}.  There is a constant $M'$ independent of $n$, such that $|\gamma_n| \leq M'$ in $\Omega^\Points$.
\end{lemma}

\begin{proof}
By Lemma~\ref{greensfunctionconvergence} we have for $x \in \Omega^\Points$
	\begin{align*} |G_n\sigma_n(x)| &\leq \delta_n^d \sum_{y \in B^\Points} |g_n(x,y)| \sigma_n(y) \\
			&\leq 2M \delta_n^d \sum_{y \in B^\Points} |g(x,y)| \\
			&\leq CM R^2 \log R \end{align*}
for a constant $C$ depending only on $d$, where $R$ is the radius of $\Omega$.  It follows that $|\gamma_n| \leq (CM+1)R^2 \log R$ in $\Omega^\Points$.  
\end{proof}

\begin{lemma}
\label{alphabounds}
Fix $x\in \R^d$, and for $y \in \delta_n\Z^d$ let 
	\begin{equation} \label{alphadef} \alpha(y) = \delta_n^d g_n(x^\Points,y) - \int_{y^\Box} g(x,z) dz. \end{equation}
There are constants $C_1, C_2$ depending only on $d$, such that
	\begin{enumerate}
	\item[{\em (i)}] $|\alpha(y)| \leq C_1 \delta_n^{1+d} |x-y|^{1-d}$.
	\item[{\em (ii)}] If $y_1 \sim y_2$, then $|\alpha(y_1)-\alpha(y_2)| \leq C_2 \delta_n^{2+d} |x-y_1|^{-d}$.
	\end{enumerate}
\end{lemma}

\begin{proof}
(i) By Lemma~\ref{greensfunctionconvergence}, if $z\in y^\Box$ then
	\begin{align*} | g_n(x^\Points,y) - g(x,z) | &\leq 
|g_n(x^\Points,y)-g(x^\Points,y)| + |g(x^\Points,y)-g(x,z)| \\
	&\leq C \delta_n |x-y|^{1-d} \end{align*}
for a constant $C$ depending only on $d$.  Integrating over $y^\Box$ gives the result.
	
(ii) Let $p=y_2-y_1$.  By Lemma~\ref{greensfunctionconvergence}, we have
	\begin{align*} |\alpha(y_1) - \alpha(y_2)| &\leq \int_{y_1^\Box} |g_n(x^\Points,y_1)-g_n(x^\Points,y_2)-g(x,z)+g(x,z+p)| \,dz \\
	&\leq \int_{y_1^\Box} |g(x^\Points,y_1)-g(x^\Points,y_2)-g(x,z)+g(x,z+p)| \,dz \;+ \\ 
		&\qquad\qquad\qquad\qquad\qquad + O(\delta_n^{d+2} |x-y_1|^{-d}).  \end{align*}
Writing $q=x^\Points - x - y_1+z$, the integrand can be expressed as
	\[ |f(o)-f(p)-f(q)+f(p+q)|, \]
where
	\[ f(w) = g(w,x^\Points-y_1). \]
Since $|p|,|q|,|p+q| \leq \delta_n (\sqrt{d}+1)$, we have by Taylor's theorem with remainder
	\begin{align*} |f(o)-f(p)-f(q)+f(p+q)| &\leq 3d\delta_n^2 \sum_{i,j=1}^{d} \left| \frac{\partial^2 f}{\partial x_i \partial x_j} (o) \right| \\
		& \leq \begin{cases}	(24/\pi) \delta_n^2 |x-y_1|^{-2}, &d=2 \\
					6d^2 (d-1)(d-2) a_d \delta_n^2 |x-y_1|^{-d}, &d\geq 3. \qed
		\end{cases}
		\end{align*}
\renewcommand{\qedsymbol}{}
\end{proof}

\begin{lemma}
\label{truncatedlaplacian}
Let $\Omega$ be an open ball as in Lemma~\ref{bigballs}, and let $\Omega_1$ be a ball with $\bar{\Omega} \subset \Omega_1$.  Let
	\[ \phi_n = -\Delta(s_n1_{\Omega_1^\Points}). \]
Then
	\[ \left| s_n^\Box - G\phi_n^\Box \right| \rightarrow 0 \]
uniformly on $\Omega$.
\end{lemma}

\begin{proof}
Let $\nu_n(x)$ be the amount of mass present at $x$ in the final state of the divisible sandpile on $\delta_n \Z^d$.  By Lemma~\ref{discretemajorant} we have $s_n = u_n + \gamma_n$, hence
	\[ \Delta s_n = (\nu_n - \sigma_n) + (\sigma_n - 1) = \nu_n -1. \]
In particular, $|\Delta s_n| \leq 1$.

Note that
	\[ G\phi_n^\Box(x) = \sum_{y \in \Omega_1^\Points \cup \partial \Omega_1^\Points} \phi_n(y) \int_{y^\Box} g(x,z) dz \]
whereas $s_n = G_n \phi_n$ in $\Omega^\Points$, by (\ref{inverseoflaplacian}).  Hence for $x\in \Omega$
	\begin{align} G\phi_n^\Box(x) - s_n^\Box(x) &= G\phi_n^\Box(x) -  G_n\phi_n(x^\Points) \nonumber \\
			&= - \sum_{y \in \Omega_1^\Points \cup \partial \Omega_1^\Points} \phi_n(y) \alpha(y) \label{herecomethealphas}
			 \end{align}
where $\alpha(y)$ is given by (\ref{alphadef}).  Now write
	\[ \phi_n = -(\Delta s_n) 1_{\Omega_1^\Points} + f_1 - f_2 \]
where $f_1$ is supported on the inner boundary of $\Omega_1^\Points$ and given there by
	\[ f_1(y) = \frac{1}{2d \delta_n^2} \sum_{z\sim y, \,z \notin \Omega_1^\Points} s_n(z) \]
while $f_2$ is supported on the outer boundary of $\Omega_1^\Points$ and given there by
	\[ f_2(z) = \frac{1}{2d \delta_n^2} \sum_{y \sim z, \,y \in \Omega_1^\Points} s_n(y). \]
The sum in (\ref{herecomethealphas}) now takes the form
	\begin{equation} \label{interiorandboundaryterms} G\phi_n^\Box(x) - s_n^\Box(x) = \sum_{y \in \Omega_1^\Points} \Delta s_n(y) \alpha(y) + \sum_{\substack{y \in \Omega_1^\Points, z \notin \Omega_1^\Points \\ y\sim z}} \frac{s_n(y)\alpha(z) - s_n(z)\alpha(y)}{2d\delta_n^2}. \end{equation}

   Let $R,R_1$ be the radii of $\Omega,\Omega_1$.  By Lemma~\ref{alphabounds}(i), summing in spherical shells about $x$, the first sum in (\ref{interiorandboundaryterms}) is bounded in absolute value by 
	\[ \sum_{y \in \Omega_1^\Points} |\alpha(y)| \leq \int_0^{2R_1} \frac{d \omega_d r^{d-1}}{\delta_n^d} (C_1 \delta_n^{1+d} r^{1-d}) dr = C_1 d \omega_d R_1\delta_n. \]
To bound the second sum in (\ref{interiorandboundaryterms}), note that $s_n=\gamma_n$ outside $\Omega$, so
	\begin{equation} \label{twogradients} |s_n(y)\alpha(z) - s_n(z)\alpha(y)| \leq |\gamma_n(y)| |\alpha(y)-\alpha(z)| + |\alpha(y)| |\gamma_n(y)-\gamma_n(z)|. \end{equation}
By Lemmas~\ref{gammaupperbound} and~\ref{alphabounds}(ii), the first term is bounded by
	\[ |\gamma_n(y)| |\alpha(y)-\alpha(z)| \leq C_2 M'\delta_n^{2+d} |x-y|^{-d}. \]
Fix $\epsilon>0$, and let $\Omega_2$ be a ball with $\bar{\Omega}_1 \subset \Omega_2$.  Since $\gamma$ is uniformly continuous on $\Omega_2$, and $\gamma_n \rightarrow \gamma$ uniformly on $\Omega_2$ by Lemma~\ref{greensintegralconvergencecont}, for sufficiently large $n$ we have
	\[ |\gamma_n(y)-\gamma_n(z)| \leq |\gamma_n(y)-\gamma(y)| + |\gamma(y)-\gamma(z)| + |\gamma(z)-\gamma_n(z)| \leq \epsilon. \]
Thus by Lemma~\ref{alphabounds}(i) the second term in (\ref{twogradients}) is bounded by
	\[ |\alpha(y)| |\gamma_n(y)-\gamma_n(z)| \leq C_1 \delta_n^{1+d} |x-y|^{1-d} \epsilon. \]
Since $x\in \Omega$ and $y$ is adjacent to $\partial \Omega_1$, we have $|x-y| \geq R_1-R-\delta_n$, so the second sum in (\ref{interiorandboundaryterms}) is bounded in absolute value by 
	\[ \delta_n^{-2} \# \partial \Omega_1^\Points \left(C_2 M' \delta_n^{2+d} (R_1-R)^{-d} + C_1 \delta_n^{1+d} (R_1-R)^{1-d} \epsilon \right) \leq C_3 \epsilon \]
for sufficiently large $n$.
\end{proof}

\begin{lemma}
\label{majorantconvergence}
$s_n^\Box \rightarrow s$ uniformly on compact subsets of $\R^d$.
\end{lemma}

\begin{proof}
By Lemma~\ref{bigballs} there is a ball $\Omega$ containing $D$ and $D_n^\Box$ for all $n$.  Outside $\Omega$ we have 
 	\[ s_n^\Box = \gamma_n^\Box \rightarrow \gamma = s \]
uniformly on compact sets by Lemma~\ref{greensintegralconvergencecont}.  

To show convergence in $\Omega$, write
	\begin{equation} \label{mollified} \tilde{s}(x) = \int_{\R^d} s(y) \lambda^{-d} \eta \left(\frac{x-y}{\lambda}\right) dy, \end{equation}
where $\eta$ is the standard smooth mollifier
	\[ \eta(x) = \begin{cases} Ce^{1/(|x|^2-1)}, &|x|<1 \\
					0, &|x|\geq 1  \end{cases} \]
normalized so that $\int_{\R^d} \eta \, dx = 1$ (see \cite[Appendix C.4]{Evans}).  Then $\tilde{s}$ is smooth and superharmonic.  Fix $\epsilon>0$.  By Lemma~\ref{majorantbasicprops}(ii) and compactness, $s$ is uniformly continuous on $\bar{\Omega}$, so taking $\lambda$ sufficiently small in (\ref{mollified}) we have $|s-\tilde{s}|<\epsilon$ in $\Omega$.   Let $A_\epsilon$ be the maximum third partial of $\tilde{s}$ in $\Omega$.  By Lemma~\ref{thirdderiv} the function
	\[ q_n(x) = \tilde{s}^\Points(x) - \frac16 A_\epsilon \delta_n |x|^2 \]
is superharmonic in $\Omega^\Points$.  By Lemma~\ref{greensintegralconvergencecont} we have $\gamma_n^\Box \rightarrow \gamma$ uniformly in $\Omega$.  Taking $N$ large enough so that $\frac16 A_\epsilon \delta_n |x|^2 < \epsilon$ in $\Omega$ and $|\gamma_n-\gamma^\Points|<\epsilon$ in $\Omega^\Points$ for all $n>N$, we obtain
	\[ q_n > \tilde{s}^\Points - \epsilon > s^\Points - 2\epsilon \geq \gamma^\Points - 2\epsilon > \gamma_n - 3\epsilon \]
in $\Omega^\Points$.  In particular, the function $f_n = \text{max}(q_n+3\epsilon,\gamma_n)$ is superharmonic in $\Omega^\Points$.  By Lemma~\ref{discretemajorantonacompactset} it follows that $f_n \geq s_n$, hence
	\[ s_n \leq q_n + 3\epsilon < \tilde{s}^\Points + 3\epsilon < s^\Points + 4\epsilon \]
in $\Omega^\Points$.  By the uniform continuity of $s$ on $\bar{\Omega}$, taking $N$ larger if necessary we have $|s-s^{\Points\Box}|<\epsilon$ in $\Omega$, and hence $s_n^\Box < s + 5\epsilon$ in $\Omega$ for all $n>N$.

For the reverse inequality, let
	\[ \phi_n = - \Delta(s_n 1_{\Omega_1^\Points}) \] 
where $\Omega_1$ is an open ball containing $\bar{\Omega}$.  By Lemma~\ref{truncatedlaplacian} we have
	\[ \left| s_n^\Box - G\phi_n^\Box \right| < \epsilon \]
and hence
	\begin{equation} \label{abovegamma} G\phi_n^\Box > \gamma_n^\Box - \epsilon > \gamma - 2\epsilon. \end{equation}
on $\Omega$ for sufficiently large $n$.  Since $\phi_n^\Box$ is nonnegative on $\Omega$, by Lemma~\ref{superharmonicpotential} the function $G\phi_n^\Box$ is superharmonic on $\Omega$, so by (\ref{abovegamma}) the function 
	\[ \psi_n = \text{max}(G\phi_n^\Box+2\epsilon,\gamma) \]
is superharmonic on $\Omega$ for sufficiently large $n$.  By Lemma~\ref{majorantonacompactset} it follows that $\psi_n \geq s$, hence
	\[ s_n^\Box > G\phi_n^\Box - \epsilon \geq s-3\epsilon \]
on $\Omega$ for sufficiently large $n$.
\end{proof}

\begin{proof}[Proof of Theorem~\ref{odomconvergence}]
Let $\Omega$ be as in Lemma~\ref{bigballs}.  By Lemmas~\ref{greensintegralconvergencecont} and~\ref{majorantconvergence} we have $\gamma_n^\Box \to \gamma$ and $s_n^\Box \to s$ uniformly on $\Omega$.  By Lemma~\ref{discretemajorant},
we have $u_n = s_n-\gamma_n$.   Since $u_n^\Box = 0 = s-\gamma$ off $\Omega$, we conclude that $u_n^\Box \rightarrow s-\gamma$ uniformly.
\end{proof}

\subsection{Convergence of Domains}

In addition to the conditions (\ref{absolutebound})-(\ref{uniformcompactsupport}) assumed in the previous section, we assume in this section that the initial density $\sigma$ satisfies 
	\begin{equation} \label{boundedawayfrom1} 
	\text{For all }x\in \R^d \text{ either }\sigma(x) \leq \lambda \text{ or } \sigma(x)\geq 1
	\end{equation}
for a constant $\lambda<1$.  We also assume that
	\begin{equation} \label{closureofinterior}
	\{\sigma \geq 1\} = \overline{\{\sigma \geq 1\}^o}. 
	\end{equation}
Moreover, we assume that for any $\epsilon>0$ there exists $N(\epsilon)$ such that
	\begin{equation} \label{sillytechnicality}
	\text{If }x \in \{\sigma \geq 1\}_\epsilon, \text{ then } \sigma_n(x) \geq 1 \text{ for all }n \geq N(\epsilon);
	\end{equation}
and
	\begin{equation} \label{discretetechnicality}
	\text{If }x \notin \{\sigma \geq 1\}^\epsilon, \text{ then } \sigma_n(x) \leq \lambda \text{ for all }n \geq N(\epsilon).
	\end{equation}
As before, we have chosen to state the hypotheses on $\sigma$ and $\sigma_n$ separately for maximum generality, but all hypotheses on $\sigma_n$ are satisfied in the particular case when $\sigma_n$ is given by averaging $\sigma$ in a small box (\ref{averageinabox}).

We set
	\[ \gamma(x) = -|x|^2-G\sigma(x) \]
with $s$ the least superharmonic majorant of $\gamma$ and 
	\[ D = \{x \in\R^d | s(x)>\gamma(x)\}. \]
We also write
	\[ \widetilde{D} = D \cup \{x \in \R^d | \sigma(x)\geq 1 \}^o. \]
For a domain $A \subset \R^d$, denote by $A_\epsilon$ and $A^\epsilon$ its inner and outer open $\epsilon$-neighborhoods, respectively.  

\begin{theorem}
\label{domainconvergence}
Let $\sigma$ and $\sigma_n$ satisfy (\ref{absolutebound})-(\ref{uniformcompactsupport}) and (\ref{boundedawayfrom1})-(\ref{discretetechnicality}).  For $n\geq 1$ let $D_n$ be the domain of fully occupied sites for the divisible sandpile in $\delta_n \Z^d$ started from initial density $\sigma_n$.  For any $\epsilon>0$ we have for large enough $n$
	\begin{equation} \label{domainbounds} \widetilde{D}_\epsilon^\Points \subset D_n \subset \widetilde{D}^{\epsilon\Points}. \end{equation}
\end{theorem}
							
According to the following lemma, near any occupied site $x \in D_n$ lying outside $\widetilde{D}^\epsilon$, we can find a site $y$ where the odometer $u_n$ is relatively large.  This is a discrete version of a standard argument for the obstacle problem; see for example Friedman \cite{Friedman}, Ch.\ 2 Lemma 3.1. 

\begin{lemma}
\label{quadraticgrowth}
Fix $\epsilon>0$ and $x \in D_n$ with $x\notin \widetilde{D}^\epsilon$.  If $n$ is sufficiently large, there is a point $y\in \delta_n\Z^d$ with $|x-y|\leq \frac{\epsilon}{2}+\delta_n$ and 
	\[ u_n(y) \geq u_n(x) + \frac{1-\lambda}{4}\epsilon^2. \]
\end{lemma}

\begin{proof}
%
%
By (\ref{closureofinterior}), we have $\{ \sigma \geq 1\} \subset \overline{\widetilde{D}}$.  Thus if $x \notin \widetilde{D}^\epsilon$, the ball $B=B(x,\frac{\epsilon}{2})^\Points$ is disjoint from $\{\sigma \geq 1\}^{\epsilon/2}$.  In particular, if $n \geq N(\frac{\epsilon}{2})$, then by (\ref{discretetechnicality}) we have $\sigma_n \leq \lambda$ on $B$.  Thus the function
	\[ w(y) = u_n(y) - (1-\lambda)|x-y|^2 \]
is subharmonic on $B \cap D_n$, so it attains its maximum on the boundary.  Since $w(x)\geq 0$, the maximum cannot be attained on $\partial D_n$, where $u_n$ vanishes; so it is attained at some point $y \in \partial B$, and
	\begin{align*} u_n(y) &\geq w(y) + (1-\lambda)\left(\frac{\epsilon}{2}\right)^2.
\end{align*}
Since $w(y)\geq w(x)=u_n(x)$, the proof is complete.
\end{proof}

\begin{proof}[Proof of Theorem~\ref{domainconvergence}]
Fix $\epsilon>0$.  By Lemma~\ref{pointset} we have 
	\[ \widetilde{D}_\epsilon \subset D_{\eta} \cup \{\sigma \geq 1\}_{\eta} \] 
for some $\eta>0$.  Let $u=s-\gamma$.  Since the closure of $D_{\eta}$ is compact and contained in $D$, we have $u \geq m_{\eta}$ on $D_{\eta}$ for some $m_{\eta}>0$.  By Theorem~\ref{odomconvergence}, for sufficiently large $n$ we have $u_n > u^\Points - \frac12 m_{\eta} > 0$ on $D_{\eta}^\Points$, hence $D_{\eta}^\Points \subset D_n$.  Likewise, by (\ref{sillytechnicality}) we have $\{\sigma \geq 1\}_{\eta}^\Points \subset D_n$ for large enough $n$.  Thus $\widetilde{D}_\epsilon \subset D_n$.

For the other inclusion, fix $x \in \delta_n\Z^d$ with $x \notin \widetilde{D}^\epsilon$.  Since $u$ vanishes on the ball $B=B(x,\frac{\epsilon}{2})$, by Theorem~\ref{odomconvergence} we have $u_n < \frac{1-\lambda}{4} \epsilon^2$ on $B^\Points \cup \partial B^\Points$ for all sufficiently large $n$.  By Lemma~\ref{quadraticgrowth} we conclude that $x \notin D_n$, and hence $D_n \subset \widetilde{D}^{\epsilon\Points}$.
\end{proof}

\section{Rotor-Router Model}

In trying to adapt the proofs of Theorems~\ref{odomconvergence} and~\ref{domainconvergence} to the rotor-router model, we are faced with two main problems.  The first is to define an appropriate notion of convergence of integer-valued densities $\sigma_n$ on $\delta_n \Z^d$ to a real-valued density $\sigma$ on $\R^d$.  The requirement that $\sigma_n$ take only integer values is of course imposed on us by the nature of the rotor-router model itself, since unlike the divisible sandpile, the rotor-router works with discrete, indivisible particles.  The second problem is to find an appropriate analogue of Lemma~\ref{discretemajorant} for the rotor-router model.

Although these two problems may seem unrelated, the single technique of \emph{smoothing} neatly takes care of them both.  To illustrate the basic idea, suppose we are given a domain $A \subset \R^d$, and let $\sigma_n$ be the function on $\delta_n \Z^d$ taking the value $1$ on odd lattice points in $A^\Points$, and the value $2$ on even lattice points in $A^\Points$, while vanishing outside $A^\Points$.  We would like to find a sense in which $\sigma_n$ converges to the real-valued density $\sigma = \frac32 1_A$.  One approach is to average $\sigma_n$ in a box whose side length $L_n$ goes to zero more slowly than the lattice spacing: $L_n \downarrow 0$ while $L_n / \delta_n \uparrow \infty$ as $n \uparrow \infty$.  The resulting ``smoothed'' version of $\sigma_n$ converges to $\sigma$ pointwise away from the boundary of $A$.  

By smoothing the odometer function in the same way, we can obtain an approximate analogue of Lemma~\ref{discretemajorant} for the rotor-router model.  Rather than average in a box as described above, however, it is technically more convenient to average according to the distribution of a lazy random walk run for a fixed number $\alpha(n)$ of steps.  Denote by $(X_k)_{k\geq 0}$ the lazy random walk in $\delta_n \Z^d$ which stays in place with 
probability $\frac12$ and moves to each of the $2d$ neighbors with probability 
$\frac{1}{4d}$.  Given a function $f$ on $\delta_n \Z^d$, define its $k$-{\it smoothing}
	\begin{equation} \label{smoothingdef} S_k f (x) = \EE (f(X_k) | X_0=x). \end{equation}
From the Markov property we have $S_k S_\ell = S_{k+\ell}$.  Also, the discrete Laplacian can be written as
	\begin{equation} \label{deltaintermsofsmoothing} \Delta = 2\delta_n^{-2}(S_1 - S_0). \end{equation}
In particular, $\Delta S_k = S_k \Delta$.

\subsection{Convergence of Odometers}

For $n=1,2,\ldots$ let $\sigma_n$ be an integer-valued function on $\delta_n\Z^d$ satisfying $0\leq \sigma_n \leq M$.  We assume as usual that there is a ball $B \subset \R^d$ containing the support of $\sigma_n$ for all $n$.  Let $\sigma$ be a function on $\R^d$ supported in $B$ satisfying (\ref{absolutebound}) and (\ref{discontmeasurezero}).  In place of condition (\ref{convergingdensities}) we assume that there exist integers $\alpha(n) \uparrow \infty$ with $\delta_n \alpha(n) \downarrow 0$ such that
	\begin{equation} \label{smootheddensityconvergence} (S_{\alpha(n)} \sigma_n)^\Box(x) \rightarrow \sigma(x), \qquad x \notin DC(\sigma). \end{equation}
	
By the odometer function for rotor-router aggregation starting from initial density $\sigma_n$, we will mean the function
	\[ u_n(x) = \delta_n^2 \cdot \text{number of particles emitted from $x$} \]
if $\sigma_n(y)$ particles start at each site $y$.

\begin{theorem}
\label{rotorodomconvergence}
Let $u_n$ be the odometer function for rotor-router aggregation on $\delta_n\Z^d$ starting from initial density $\sigma_n$.  If $\sigma, \sigma_n$ satisfy (\ref{absolutebound})-(\ref{discreteabsolutebound}), (\ref{uniformcompactsupport}) and (\ref{smootheddensityconvergence}),    then $u_n^\Box \rightarrow s-\gamma$ uniformly, where
	\[ \gamma(x) = -|x|^2 - \int_{\R^d} g(x,y) \sigma(y) dy \]
and $s$ is the least superharmonic majorant of $\gamma$.
\end{theorem}	

Given a function $f$ on $\delta_n \Z^d$, for a directed edge $(x,y)$ write
	\[ \nabla f (x,y) = \frac{f(y)-f(x)}{\delta_n}. \]
Given a function $\kappa$ on directed edges in $\delta_n \Z^d$, write
	\[ \div \kappa (x) = \frac{1}{2d\delta_n} \sum_{y \sim x} \kappa(x,y). \]
The discrete Laplacian on $\delta_n \Z^d$ is then given by
	\[ \Delta f(x) = \div \nabla f = \delta_n^{-2} \left( \frac{1}{2d} \sum_{y \sim x} f(y)-f(x) \right). \]
The following ``rescaled" version of \cite[Lemma 5.1]{LP} is proved in the same way.

\begin{lemma}
\label{scaledodomflow}
For a directed edge $(x,y)$ in $\delta_n \Z^d$, denote by $\kappa(x,y)$ the net number of crossings from $x$ to $y$ performed by particles during a sequence of rotor-router moves.  Let 
	\[u(x) = \delta_n^2 \cdot \text{\em number of particles emitted from $x$ during this sequence}.\] 
Then
	\begin{equation} \label{scaledgradodom} \nabla u(x,y) = \delta_n (-2d\kappa(x,y) + \rho(x,y))
\end{equation}
for some function $\rho$ on directed edges in $\delta_n \Z^d$ satisfying
	\[ |\rho(x,y)| \leq 4d-2 \]
for all directed edges $(x,y)$.
\end{lemma}

\begin{proof}
Writing $N(x,y)$ for the number of particles routed from $x$ to $y$, for any $y,z \sim x$ we have
	\[ |N(x,y)-N(x,z)| \leq 1. \]
Since $u(x) = \delta_n^2 \sum_{y \sim x} N(x,y)$, we obtain
	\[ \delta_n^{-2}u(x)-2d+1 \leq 2d N(x,y) \leq \delta_n^{-2}u(x)+2d-1 \]
hence
	\begin{align*}  | \nabla u(x,y) + 2d\delta_n \kappa(x,y) | &= \delta_n|\delta_n^{-2} u(y) - \delta_n^{-2}u(x) + 2d N(x,y) - 2d N(y,x) | \\
		&\leq (4d - 2)\delta_n. \qed \end{align*}
\renewcommand{\qedsymbol}{} 
\end{proof}

\vspace{-3ex}
Let $R_n$ be the final set of occupied sites for rotor-router aggregation on~$\delta_n \Z^d$ starting from initial density $\sigma_n$.
Since each site $x$ starts with $\sigma_n(x)$ particles and ends with either one particle or none accordingly as $x\in R_n$ or $x \notin R_n$, we have 
	\[ 2d\delta_n\, \div \kappa = \sigma_n - 1_{R_n}. \]
Taking the divergence in (\ref{scaledgradodom}) we obtain
	\begin{equation} \label{laplacianofrotorodom} \Delta u_n = 1_{R_n} - \sigma_n + \delta_n\, \div \rho.   \end{equation}
In particular, since $0 \leq \sigma_n \leq M$, and $|\rho| \leq 4d-2$, we have
	\begin{equation} \label{rotorlaplacianbound} |\Delta u_n| \leq M+4d \end{equation}
on all of $\delta_n \Z^d$.

Recall that the divisible sandpile odometer function has Laplacian $1-\sigma_n$ inside the set $D_n$ of fully occupied sites.  The next lemma shows that the same is approximately true of the smoothed rotor-router odometer function~$S_k u_n$.  Denote by $\PP_x$ and $\EE_x$ the probability and expectation operators for the lazy random walk $(X_k)_{k \geq 0}$ defining the smoothing (\ref{smoothingdef}), started from $X_0=x$.

\begin{lemma}
\label{laplacianofsmoothing}
There is a constant $C_0$ depending only on $d$, such that
	\[ |\Delta S_k u_n(x) - \PP_x(X_k \in R_n) + S_k\sigma_n(x)| \leq \frac{C_0}{\sqrt{k+1}}. \]
\end{lemma}

\begin{proof}
Let $\kappa$ and $\rho$ be defined as in Lemma~\ref{scaledodomflow}.  Using the fact that $S_k$ and~$\Delta$ commute, we obtain from (\ref{laplacianofrotorodom})
	\begin{equation} \label{divergenceerrorterm} \Delta S_k u_n(x) = \PP_x(X_k \in R_n) - \EE_x \sigma_n(X_k) + \delta_n \EE_x \div \rho(X_k). \end{equation}
Since $\nabla u$ and $\kappa$ are antisymmetric, $\rho$ is antisymmetric by (\ref{scaledgradodom}).
Thus the last term in (\ref{divergenceerrorterm}) can be written
	\begin{align} \delta_n \EE_x \div \rho(X_k) &= \frac{1}{2d} \sum_y \sum_{z \sim y} \PP_x(X_k=y) \rho(y,z) \nonumber \\
	&= - \frac{1}{2d} \sum_y \sum_{z \sim y} \PP_x(X_k=z)\rho(y,z) \nonumber \\
	&= \frac{1}{4d} \sum_y \sum_{z \sim y} (\PP_x(X_k= y) - \PP_x(X_k=z)) \rho(y,z) \label{readyfortirangleineq} \end{align}
where the sums are taken over all pairs of neighboring sites $y,z \in \delta_n\Z^d$.

We can couple lazy random walk $X_k$ started at $x$ with lazy random walk $X'_k$ started at a neighbor $x'$ of $x$ so that the probability of not coupling in $k$ steps is at most $C/\sqrt{k+1}$, where $C$ is a constant depending only on~$d$~\cite{Lindvall}.  Since the total variation distance between the distributions of $X_k$ and $X'_k$ is at most the probability of not coupling, we obtain
	\begin{align*} \sum_y \sum_{z\sim y} |\PP_x(X_k= y) - \PP_x(X_k=z)| &= \sum_y \sum_{x' \sim x} | \PP_x(X_k=y) - \PP_{x'}(X'_k=y) | \\
	&\leq \frac{4dC}{\sqrt{k+1}}. \end{align*}
Using the fact that $|\rho(y,z)| \leq 4d$ in (\ref{readyfortirangleineq}), taking $C_0 = 4dC$ completes the proof.
\end{proof}



The next lemma shows that smoothing the rotor-router odometer function does not introduce much extra error.

\begin{lemma}
\label{smoothingdist}
$|S_k u_n-u_n| \leq \delta_n^2(\frac12 Mk + C_0\sqrt{k+1})$.
\end{lemma}
 
\begin{proof}
From (\ref{deltaintermsofsmoothing}) we have
	\begin{align*} |S_k u_n-u_n| &\leq \sum_{j=0}^{k-1} |S_{j+1}u_n-S_j u_n| \\
			&= \frac{\delta_n^2}{2} \sum_{j=0}^{k-1} |\Delta S_j u_n|. \end{align*}
But by Lemma~\ref{laplacianofsmoothing}
	\[ |\Delta S_j u_n| \leq M + \frac{C_0}{\sqrt{j+1}}. \]
Summing over $j$ yields the result.
\end{proof}

Let
	\[ \gamma_n(x) = -|x|^2 - G_n S_{\alpha(n)} \sigma_n(x) \]
and let $s_n$ be the least superharmonic majorant of $\gamma_n$.  By Lemma~\ref{discretemajorant}, the difference $s_n-\gamma_n$ is the odometer function for the divisible sandpile  on $\delta_n \Z^d$ starting from the smoothed initial density $\widetilde{\sigma}_n = S_{\alpha(n)} \sigma_n$.  Note that by Lemma~\ref{bigballs}, there is a ball $\Omega \subset \R^d$ containing the supports of $s-\gamma$ and of $s_n-\gamma_n$ for all $n$.  The next lemma compares the smoothed rotor-router odometer for the initial density $\sigma_n$ with the divisible sandpile odometer for the smoothed density $\widetilde{\sigma}_n$.

\begin{lemma}
\label{smoothedodomlowerbound}
Let $\Omega \subset \R^d$ be a ball centered at the origin containing the supports of $s-\gamma$ and of $s_n-\gamma_n$ for all $n$.  Then 
 	\begin{equation} \label{onesidedodomcomparison} S_{\alpha(n)} u_n \geq s_n - \gamma_n - C_0 r^2 \alpha(n)^{-1/2} \end{equation}
on all of $\delta_n \Z^d$, where $r$ is the radius of $\Omega$.
\end{lemma}

\begin{proof}
 Since
	 \[ \Delta \gamma_n = -1 + S_{\alpha(n)} \sigma_n, \]
by Lemma~\ref{laplacianofsmoothing} the function
	\[ f(x) = S_{\alpha(n)}u_n(x) + \gamma_n(x) + C_0 \alpha(n)^{-1/2}(r^2-|x|^2) \]
is superharmonic on $\delta_n \Z^d$.  Since $f\geq \gamma_n$ on $\Omega^\Points$, the function $\psi_n = \text{max}(f,\gamma_n)$ is superharmonic on $\Omega^\Points$, hence $\psi_n \geq s_n$ by Lemma~\ref{discretemajorantonacompactset}.  Thus $f \geq s_n$ on $\Omega^\Points$, so (\ref{onesidedodomcomparison}) holds on $\Omega^\Points$ and hence everywhere.  
\end{proof}

\begin{lemma}
\label{rotorbigballs}
Let $R_n$ be the set of occupied sites for rotor-router aggregation in $\delta_n \Z^d$ started from initial density $\sigma_n$.  There is a ball $\Omega \subset \R^d$ with $\bigcup R_n^\Box \subset \Omega$.
\end{lemma}

\begin{proof}
By assumption there is a ball $B \subset \R^d$ containing the support of $\sigma_n$ for all $n$.
Let $A_n$ be the set of occupied sites for rotor-router aggregation in $\delta_n \Z^d$ started from initial density $\tau(x) = M 1_{x \in B}$.  From the abelian property we have $R_n \subset A_n$.  By the inner bound of \cite[Theorem~1.1]{LP}, if $\left\lfloor 2\delta_n^{-d} \Leb(B) \right\rfloor$ particles perform rotor-router aggregation starting at the origin in $\delta_n \Z^d$, the resulting set of occupied sites contains $B^\Points$; by the abelian property it follows that if we start with $M \left\lfloor 2\delta_n^{-d} \Leb(B) \right\rfloor$ particles at the origin, the resulting set $\Omega_n$ of fully occupied sites contains $A_n$.  By the outer bound of \cite[Theorem~1.1]{LP}, $\Omega_n$ is contained in a ball $\Omega$ of volume $3M \Leb(B)$.
\end{proof}

Let
	\[ R_n^{(k)} = \{x \in R_n \,:\, y \in R_n \text{ whenever } ||x-y||_1 \leq k\delta_n \}. \]
The following lemma shows that the rotor-router odometer grows at most quadratically as we move away from the boundary of $R_n$.

\begin{lemma}
\label{rotoratmostquadratic}
For any $k\geq 0$, if $x \notin R_n^{(k)}$, then for any $\beta>0$
	\[ u_n(x) \leq c(M+4d)k^2 \delta_n^2 \]
where $c = c_\beta$ is the constant in Lemma~\ref{atmostquadratic}.
\end{lemma}

\begin{proof}
If $x \notin R_n^{(k)}$, then there is a point $y \not\in R_n$ with $|x-y| \leq k \delta_n$.  Since $u_n(y)=0$, we have by (\ref{rotorlaplacianbound}) and Lemma~\ref{atmostquadratic}
	\[ u_n(x) \leq c(M+4d) |x-y|^2. \qed \]
\renewcommand{\qedsymbol}{}
\end{proof}

\begin{proof}[Proof of Theorem~\ref{rotorodomconvergence}]
By Lemma~\ref{bigballs} there is a ball $\Omega$ containing the support of $s-\gamma$ and of $s_n-\gamma_n$ for all $n$.  By Lemma~\ref{rotorbigballs} we can enlarge $\Omega$ if necessary to contain the support of $S_{\alpha(n)} u_n$ for all $n$.
By Lemma~\ref{laplacianofsmoothing}, the function
	\[ \phi(x) = S_{\alpha(n)} u_n(x) - s_n(x) + \gamma_n(x) + C_0 \alpha(n)^{-1/2} |x|^2. \]
is subharmonic on the set
	\[ \widetilde{R}_n = \{x \in R_n \,:\, y \in R_n \text{ whenever } ||x-y||_1 \leq \alpha(n)\delta_n \} \]
since $\PP_x\left(X_{\alpha(n)} \in R_n \right)=1$ for $x \in \widetilde{R}_n$.  

Given a point $y \notin \widetilde{R}_n$, we have by Lemma~\ref{rotoratmostquadratic} 
	\[ u_n(y) \leq c(M+4d)\alpha(n)^2 \delta_n^2. \]
Now by Lemma~\ref{smoothingdist},
	\[ S_{\alpha(n)} u_n(y) \leq \delta_n^2 \left( c(M+4d)\alpha(n)^2 + \frac12M\alpha(n) +  C_0 \sqrt{\alpha(n)+1} \right). \]
The right side is at most $C_1 \alpha(n)^2 \delta_n^2$, where $C_1 = c(M+4d)+\frac12 M + C_0 \sqrt{2}$.  Since $s_n \geq \gamma_n$, we obtain for all $y \in \Omega - \widetilde{R}_n$
	\begin{equation} \label{smalloutsideR_ntilde} \phi(y) \leq C_1 \alpha(n)^2 \delta_n^2 + C_0 \alpha(n)^{-1/2} r^2 \end{equation}
where $r$ is the radius of $\Omega$.  By the maximum principle, (\ref{smalloutsideR_ntilde}) holds for all $y \in \Omega$.  Now from Lemma~\ref{smoothedodomlowerbound} we obtain
	\begin{align} \label{twosidedbound} -C_0 r^2\alpha(n)^{-1/2} &\leq S_{\alpha(n)} u_n - s_n + \gamma_n \nonumber \\ &\leq C_1 \alpha(n)^2\delta_n^2 + C_0 r^2\alpha(n)^{-1/2} \end{align}
on all of $\delta_n \Z^d$.

By Lemmas~\ref{greensintegralconvergencecont} and~\ref{majorantconvergence} we have $\gamma_n^\Box \to \gamma$ and $s_n^\Box \to s$ uniformly on $\Omega$.  Since $\alpha(n) \uparrow \infty$ and $\delta_n \alpha(n) \downarrow 0$, we conclude from (\ref{twosidedbound}) that $\left(S_{\alpha(n)} u_n\right)^\Box \to s-\gamma$ uniformly on $\Omega$.  Since both $S_{\alpha(n)} u_n$ and $s-\gamma$ vanish outside $\Omega$, this convergence is uniform on $\R^d$.  By Lemma~\ref{smoothingdist} we have 
	\[ \left|S_{\alpha(n)} u_n - u_n \right| \to 0 \]
uniformly on $\delta_n \Z^d$, and hence $u_n^\Box \to s-\gamma$ uniformly on $\R^d$.
\end{proof}

\subsection{Convergence of Domains}
In addition to the assumptions of the previous section, in this section we require that
	\begin{equation} \label{rotorboundedawayfrom1} 
	\text{For all }x\in \R^d \text{ either }\sigma(x) \geq 1 \text{ or } \sigma(x)=0.
	\end{equation}
We also assume that
	\begin{equation} \label{rotorclosureofinterior}
	\{\sigma \geq 1\} = \overline{\{\sigma \geq 1\}^o}. 
	\end{equation}
Moreover, we assume that for any $\epsilon>0$ there exists $N(\epsilon)$ such that
	\begin{equation} \label{rotorsillytechnicality}
	\text{If }x \in \{\sigma \geq 1\}_\epsilon, \text{ then } \sigma_n(x) \geq 1 \text{ for all }n \geq N(\epsilon);
	\end{equation}
and
	\begin{equation} \label{rotordiscretetechnicality}
	\text{If }x \notin \{\sigma \geq 1\}^\epsilon, \text{ then } \sigma_n(x) =0 \text{ for all }n \geq N(\epsilon).
	\end{equation}
	
\begin{theorem}
\label{rotordomainconvergence}
Let $\sigma$ and $\sigma_n$ satisfy (\ref{absolutebound})-(\ref{discreteabsolutebound}), (\ref{uniformcompactsupport}), (\ref{smootheddensityconvergence}) and (\ref{rotorboundedawayfrom1})-(\ref{rotordiscretetechnicality}).  For $n\geq 1$, let $R_n$ be the final domain of occupied sites for rotor-router aggregation in $\delta_n \Z^d$ started from initial density $\sigma_n$.  For any $\epsilon>0$ we have for all sufficiently large $n$
	\begin{equation} \label{rotordomainbounds} \widetilde{D}_\epsilon^\Points \subset R_n \subset \widetilde{D}^{\epsilon\Points}, \end{equation}
where
	\[ \widetilde{D} = \{s>\gamma\} \cup \{\sigma\geq 1\}^o. \]
\end{theorem}

\begin{lemma}
\label{bootstrappinggrowth}
Fix $\epsilon>0$ and $n \geq N(\epsilon/4)$.  Given $x \not\in \widetilde{D}^{\epsilon/2}$ and $\delta_n \leq \rho \leq \epsilon/4$, let
	\[ \mathcal{N}_\rho(x) = \# B(x,\rho) \cap R_n. \]
If $u_n \leq \delta_n^2 m$ on $B(x,\rho)$, then
	\[ \mathcal{N}_\rho(x) \geq \left(1 + \frac1m \right) \mathcal{N}_{\rho-\delta_n}(x). \]
\end{lemma}

\begin{proof}
By (\ref{rotorclosureofinterior}), we have $\{\sigma \geq 1\}^{\epsilon/4} \subset \widetilde{D}^{\epsilon/4}$.  The ball $B(x,\rho)$ is disjoint from $\widetilde{D}^{\epsilon/4}$, so by (\ref{rotordiscretetechnicality}) we have $\sigma_n(y)=0$ for all $y \in B(x,\rho)$.
Since at least $\mathcal{N}_{\rho-\delta_n}(x)$ particles must enter the ball $B(x,\rho-\delta_n)$, we have
	\[ \sum_{y \in \partial B(x,\rho-\delta_n)} u_n(y) \geq \delta_n^2 \mathcal{N}_{\rho-\delta_n}(x). \]
There are at most $\mathcal{N}_{\rho}(x)-\mathcal{N}_{\rho-\delta_n}(x)$ nonzero terms in the sum on the left side, and each term is at most $\delta_n^2 m$, hence
	\[ m\big( \mathcal{N}_{\rho}(x) - \mathcal{N}_{\rho-\delta_n}(x) \big) \geq \mathcal{N}_{\rho-\delta_n}(x). \qed \]
\renewcommand{\qedsymbol}{}
\end{proof}

The following lemma can be seen as a weak analogue of Lemma~\ref{quadraticgrowth} for the divisible sandpile.  In Lemma~\ref{rotorquadraticgrowth}, below, we obtain a more exact analogue under slightly stronger hypotheses.

\begin{lemma}
\label{slowgrowth}
For any $\epsilon>0$, if $n$ is sufficiently large and $x \in R_n$ with $x \not\in \widetilde{D}^{\epsilon/2}$, then there is a point $y \in R_n$ with $|x-y| \leq \epsilon/4$ and 
	\[ u_n(y) \geq \frac{\epsilon \delta_n}{16 \log \big(2\omega_d (\epsilon/4\delta_n)^d \big)}. \]
\end{lemma}

\begin{proof}
Take $n \geq N(\epsilon/4)$ large enough so that $8 \delta_n <\epsilon$.  Let $m$ be the maximum value of $\delta_n^{-2} u_n$ on $B(x,\epsilon/4)^\Points$.  Since $x \in R_n$ we have $\mathcal{N}_0(x) = 1$.  Iteratively applying Lemma~\ref{bootstrappinggrowth}, we obtain
	\[ \mathcal{N}_{\epsilon/4}(x) \geq \left( 1+\frac{1}{m} \right)^{\floor{\epsilon/4\delta_n}} \mathcal{N}_0(x) \geq \exp \left( \frac{\epsilon}{16m \delta_n} \right). \]
For large enough $n$ we have
	\[ \mathcal{N}_{\epsilon/4}(x) \leq \# B(x,\epsilon/4)^\Points \leq 2\omega_d \left(\frac{\epsilon}{4\delta_n} \right)^d \]
hence
	\[ \frac{\epsilon}{16m\delta_n} \leq \log \big( 2\omega_d (\epsilon/4\delta_n)^d \big). \]
Solving for $m$ yields the result.
\end{proof}


\begin{lemma}
\label{rotorquadraticgrowth}
Fix $\epsilon>0$ and $k \geq 4C_0^2$, where $C_0$ is the constant in Lemma~\ref{laplacianofsmoothing}.  There is a constant $C_1$ such that if $x \in R_n$, $x \not\in \widetilde{D}^{3\epsilon/4}$ satisfies
	\begin{equation} \label{decentsized} S_k u_n(x) > C_1 k^2 \delta_n^2 \end{equation}
and $n$ is sufficiently large, then there exists $y \in \delta_n \Z^d$ with $|x-y| \leq \frac{\epsilon}{2} + \delta_n$ and
	\[ S_k u_n(y) \geq S_k u_n(x) + \frac{\epsilon^2}{8}. \]
\end{lemma}

\begin{proof}
Let $C_1 = c(M+4d)+\frac12 M + C_0 \sqrt{2}$.  Note first that by Lemma~\ref{smoothingdist} and (\ref{decentsized})
	\begin{align*} u_n(x) &\geq S_k u_n(x) -  \delta_n^2 \left(\frac12 kM + C_0 \sqrt{k+1} \right) \\
					&> c(M+4d) k^2 \delta_n^2 \end{align*}
so $x \in R_n^{(k)}$ by Lemma~\ref{rotoratmostquadratic}.

By Lemma~\ref{laplacianofsmoothing}, we have
	\[ \Delta S_k u_n (y) \geq 1 - S_k \sigma_n (y) - C_0/\sqrt{k+1}, \qquad y \in R_n^{(k)}. \]
Note that $C_0/\sqrt{k+1}<\frac12$.  Take $n \geq N(\epsilon/8)$ large enough so that $k \delta_n < \epsilon/8$; then $S_k \sigma_n$ vanishes on the ball $B=B(x,\epsilon/2)^\Points$ by (\ref{rotordiscretetechnicality}).  Thus the function
	\[ f(z) = S_k u_n(z) - \frac12 |x-z|^2 \]
is subharmonic on $\widetilde{R}_n := R_n^{(k)} \cap B$, so it attains its maximum on the boundary.  

If $z \in \partial R_n^{(k)}$, then by Lemma~\ref{rotoratmostquadratic}
	\[ u_n(z) \leq c(M+4d) k^2 \delta_n^2. \]
From Lemma~\ref{smoothingdist} we obtain
	\[ S_k u_n(z) \leq C_1 k^2 \delta_n^2 < S_k u_n(x). \]
Thus $f(z)<f(x)$.  Since $x \in \widetilde{R}_n$, it follows that $f$ cannot attain its maximum in $\widetilde{R}_n \cup \partial \widetilde{R}_n$ at a point $z\in \partial R_n^{(k)}$, and hence must attain its maximum at a point $y \in \partial B$.  Since $f(y) \geq f(x)$ we conclude that
	\begin{align*} S_k u_n(y) &= f(y) + \frac12 |x-y|^2 \\ 
		&\geq S_k u_n(x) + \frac12 \left( \frac{\epsilon}{2} \right)^2. \qed \end{align*}
\renewcommand{\qedsymbol}{}
\end{proof}

\begin{proof}[Proof of Theorem~\ref{rotordomainconvergence}]
Fix $\epsilon>0$.  Let $u=s-\gamma$ and $D=\{u>0\}$.  By Lemma~\ref{pointset} we have 
	\[ \widetilde{D}_\epsilon \subset D_{\eta} \cup \{\sigma \geq 1\}_{\eta} \] 
for some $\eta>0$.  Since the closure of $D_{\eta}$ is compact and contained in $D$, we have $u \geq m_{\eta}$ on $D_{\eta}$ for some $m_{\eta}>0$.  By Theorem~\ref{rotorodomconvergence}, for sufficiently large $n$ we have $u_n > u - \frac12 m_{\eta} > 0$ on $D_{\eta}^\Points$, hence $D_{\eta}^\Points \subset R_n$.  Likewise, by (\ref{rotorsillytechnicality}) we have $\{\sigma \geq 1\}_{\eta}^\Points \subset R_n$ for large enough $n$.  Thus $\widetilde{D}_\epsilon^\Points \subset R_n$.

For the other inclusion, fix $x \in \delta_n \Z^d$ with $x \notin \widetilde{D}^\epsilon$.  Since $u$ vanishes on the ball $B=B(x,\epsilon)$, by Theorem~\ref{rotorodomconvergence} we have $u_n < \frac{\epsilon^2}{8}$ on $B^\Points$ for all sufficiently large $n$.  Let $k=4C_0^2$, and take $n$ large enough so that $k\delta_n < \epsilon/4$.  Then $S_k u_n < \frac{\epsilon^2}{8}$ on the ball $B(x,3\epsilon/4)^\Points$.  By Lemma~\ref{rotorquadraticgrowth} it follows that $S_k u_n \leq C_1 k^2 \delta_n^2$ on the smaller ball $B' = B(x,\epsilon/4)^\Points$.  In particular, for $y \in B'$ we have
	\[ u_n(y) \leq \frac{S_k u_n(y)}{\PP_y(X_k=y)} < C_2 k^{d/2+2} \delta_n^2 \]
where~$C_2$ is a constant depending only on~$d$.
For sufficiently large $n$ the right side is at most $\epsilon \delta_n / 16\log \big(2\omega_d (\epsilon/4\delta_n)^d \big)$, so we conclude from Lemma~\ref{slowgrowth} that $x \not\in R_n$.
\end{proof}

\section{Internal DLA}
 
Our hypotheses for convergence of internal DLA domains are the same as those for the rotor-router model: $\sigma$ is a bounded, nonnegative, compactly supported function on $\R^d$ that is continuous almost everywhere and $\geq 1$ on its support; and $\{\sigma_n\}_{n\geq 1}$ is a sequence of uniformly bounded functions on $\delta_n \Z^d$ with uniformly bounded supports, whose ``smoothings'' converge pointwise to $\sigma$ at all continuity points of $\sigma$. 

\begin{theorem}
\label{IDLAconvergence}
Let $\sigma$ and $\sigma_n$ satisfy (\ref{absolutebound})-(\ref{discreteabsolutebound}), (\ref{uniformcompactsupport}), (\ref{smootheddensityconvergence}) and (\ref{rotorboundedawayfrom1})-(\ref{rotordiscretetechnicality}).  For $n\geq 1$ let $I_n$ be the random domain of occupied sites for internal DLA in $\delta_n \Z^d$ started from initial density $\sigma_n$.  If $\delta_n \log n \downarrow 0$, then for all $\epsilon>0$ we have with probability one
	\begin{equation} \label{IDLAdomainbounds} \widetilde{D}_\epsilon^\Points \subset I_n \subset \widetilde{D}^{\epsilon\Points} \qquad \text{\em for all sufficiently large $n$,}  \end{equation}
where
	\[ \widetilde{D} = \{s>\gamma\} \cup \{\sigma\geq 1\}^o. \]
\end{theorem}
            
\subsection{Inner Estimate}

Fix $n \geq 1$, and label the particles in $\delta_n \Z^d$ by the integers $1, \ldots, m_n$, where $m_n = \sum_{x \in \delta_n \Z^d} \sigma_n(x)$.  Let $x_i$ be the starting location of the particle labeled $i$, so that
	\[ \# \{i|x_i=x\} = \sigma_n(x). \]
For each $i=1,\ldots,m_n$ let $(X^i_t)_{t\geq 0}$ be a simple random walk in $\delta_n \Z^d$ such that $X_0^i = x_i$, with $X^i$ and $X^j$ independent for $i\neq j$.
	
For $z \in \delta_n \Z^d$ and $\epsilon>0$, consider the stopping times
	\begin{align} &\tau_z^i = \inf \big\{t \geq 0 \,|\, X_t^i = z\big\}; \nonumber \\
			        &\tau_\epsilon^i = \inf \big\{t \geq 0 \,|\, X_t^i \not \in D_\epsilon^\Points\big\}; \nonumber \\
			        &\nu^i = \inf \big\{t \geq 0 \,|\, X_t^i \not \in \{X_{\nu^j}^j\}_{j=1}^{i-1} \big\}.\label{stoppingtimefortheithparticle}
	\end{align}
The stopping time $\nu^i$ is defined inductively in $i$ with $\nu^1 = 0$.  We think of building up the internal DLA cluster one site at a time, by letting the particle labeled $i$ walk until it exits the set of sites already occupied by particles with smaller labels.  Thus $\nu^i$ is the number of steps taken by the particle labeled $i$, and $X_{\nu^i}^i$ is the location where it stops.

Fix $z \in D_\epsilon^\Points$ and consider the random variables 
	\begin{align*} &\mathcal{M}_\epsilon = \sum_{i=1}^{m_n} 1_{\{\tau_z^i < \tau_\epsilon^i\}}; \\
			&L_\epsilon = \sum_{i=1}^{m_n} 1_{\{\nu^i \leq \tau_z^i < \tau_\epsilon^i\}}.
	\end{align*}
These sums can be interpreted in terms of the following two-stage procedure.  During the first stage, we let each particle walk until it either reaches an unoccupied site or exits $D_\epsilon^\Points$.  In the second stage, we let each particle that did not already exit $D_\epsilon^\Points$ continue walking until it exits $D_\epsilon^\Points$.  Then $\mathcal{M}_\epsilon$ counts the number of particles that visit $z$ during both stages, while $L_\epsilon$ counts the number of particles that visit $z$ during the second stage.  In particular, if $L_\epsilon < \mathcal{M}_\epsilon$, then $z$ was visited during the first stage and hence belongs to the occupied cluster $I_n$.

The sum $L_\epsilon$ is difficult to estimate directly because the indicator random variables in the sum are not independent.  Following \cite{LBG}, we can bound $L_\epsilon$ by a sum of independent indicators as follows.  For each site $y \in D_\epsilon^\Points$, let $(Y_t^y)_{t \geq 0}$ be a simple random walk in $\delta_n \Z^d$ such that $Y_0^y=y$, with $Y^x$ and~$Y^{y}$ independent for $x \neq y$.  Let
	\[ \widetilde{L}_\epsilon = \sum_{y \in D_\epsilon^\Points} 1_{\{\tau_z^y < \tau_\epsilon^y\}} \]
where
	\begin{align*}
	&\tau_z^y = \inf \big\{t \geq 0 \,|\, Y_t^y = z\big\}; \\
	&\tau_\epsilon^y = \inf \big\{t \geq 0 \,|\, Y_t^y \not \in D_\epsilon^\Points\big\}.
	\end{align*}
Thus $\widetilde{L}_\epsilon$ counts the number of walks $Y^y$ that hit $z$ before exiting $D_\epsilon^\Points$.  Since the sites $X_{\nu^i}^i$ are distinct, we can couple the walks $\{Y^y\}$ and $\{X^i\}$ so that $L_\epsilon \leq \widetilde{L}_\epsilon$. 

Define
	\[ f_{n,\epsilon}(z) = g_{n,\epsilon}(z,z) \EE(\mathcal{M}_\epsilon-\widetilde{L}_\epsilon), \]
where 
	\[ g_{n,\epsilon}(y,z) = \EE \#\{t < \tau_\epsilon^y \,|\, Y_t^y=z\} \]
is the Green's function for simple random walk in $\delta_n \Z^d$ stopped on exiting $D_\epsilon^\Points$.  Then
	\begin{align} \label{fintermsofgreens} f_{n,\epsilon}(z) &= g_{n,\epsilon}(z,z) \left(\sum_{i=1}^{m_n} \PP\big(\tau_z^i<\tau_\epsilon^i\big) - \sum_{y \in D_\epsilon^\Points} \PP\big(\tau_z^y<\tau_\epsilon^y\big) \right)
	    \nonumber \\
	    &= g_{n,\epsilon}(z,z) \sum_{y \in D_\epsilon^\Points} (\sigma_n(y)-1) \PP\big(\tau_z^y<\tau_\epsilon^y\big) \nonumber \\
	&= \sum_{y \in D_\epsilon^\Points} (\sigma_n(y)-1)g_{n,\epsilon}(y,z) \end{align}
where in the last step we have used the identity
	\begin{equation} \label{greensquotient} \PP\big(\tau_z^y < \tau_\epsilon^y\big) = \frac{g_{n,\epsilon}(y,z)}{g_{n,\epsilon}(z,z)}. \end{equation}
Thus $f_{n,\epsilon}$ solves the Dirichlet problem
	\begin{align}
	\Delta f_{n,\epsilon} &= \delta_n^{-2}(1-\sigma_n), 
	\qquad &\text{on }D_\epsilon^\Points; \label{laplacianoff} \\
	f_{n,\epsilon}&=0, \qquad &\text{on }\partial D_\epsilon^\Points. \label{boundarycondf}
	\end{align}

Note that the divisible sandpile odometer function $u_n$ for the initial density $\sigma_n$ solves exactly the same Dirichlet problem, with the domain $D_\epsilon^\Points$ in (\ref{laplacianoff}) and (\ref{boundarycondf}) replaced by the domain $D_n=\{u_n>0\}$ of fully occupied sites.  Our strategy will be first 
to argue that since the domains $D_n$ and $D$ are close, the solutions to the Dirichlet problems in $D_n$ and $D_\epsilon^\Points$ should be close; next, by Theorem~\ref{odomconvergence}, since $u_n \to u := s-\gamma$ it follows that the functions $f_{n,\epsilon}$ and $u$ are close.  Since $u$ is strictly positive in $D_\epsilon$, we obtain in this way a lower bound on $f_{n,\epsilon}$, and hence a lower bound on $\EE (\mathcal{M}_\epsilon-\widetilde{L}_\epsilon)$.  Finally, using large deviations for sums of independent indicators, we conclude that with high probability $\widetilde{L}_\epsilon<\mathcal{M}_\epsilon$, and hence with high probability every point $z \in D_\epsilon^\Points$ belongs to the occupied cluster $I_n$.  The core of the argument is Lemma~\ref{coreoftheargument}, below, which gives the desired lower bound on $f_{n,\epsilon}$.

In order to apply Theorem~\ref{odomconvergence}, we must have discrete densities which converge pointwise to $\sigma$ at all continuity points of $\sigma$.  Recall, however, that in order to allow $\sigma$ to assume non-integer values, we have chosen not to assume that $\sigma_n$ converges to $\sigma$, but rather only that the \emph{smoothed density} $S_{\alpha(n)} \sigma_n$ converges to $\sigma$; see (\ref{smootheddensityconvergence}).  The next lemma shows that this initial smoothing step does not change the divisible sandpile odometer by very much.

\begin{lemma}
\label{odometerofsmoothing}
Let $k \geq 0$ be an integer, and let $u_n$ (resp.\ $\tilde{u}_n$) be the odometer function for the divisible sandpile on $\delta_n \Z^d$ started from initial density $\sigma_n$ (resp.\ $S_k \sigma_n$).  If $\sigma_n$ has finite support and $0 \leq \sigma_n \leq M$, then
	\[ \left| u_n - \tilde{u}_n \right| \leq \frac12 k M \delta_n^2. \]
\end{lemma}

\begin{proof}
From (\ref{deltaintermsofsmoothing}) we have the identity
	$ S_1 f = f +\frac{\delta_n^2}{2} \Delta f $.
By induction on $k$, it follows that
	\[ S_k \sigma_n = \sigma_n + \Delta w_k \]
where
	\[ w_k = \frac{\delta_n^2}{2} \sum_{j=0}^{k-1} S_j \sigma_n. \]
Since $0 \leq S_j \sigma_n \leq M$ for all $j$, we have $0 \leq w_k \leq \frac12 kM \delta_n^2$.

Let $\nu_n$ (resp.\ $\tilde{\nu}_n$) be the final mass density on $\delta_n \Z^d$ for the divisible sandpile started from initial density $\sigma_n$ (resp.\ $S_k \sigma_n$).  Then
	\begin{align*} \tilde{\nu}_n &= S_k \sigma_n + \Delta \tilde{u}_n  \\
 					     &= \sigma_n + \Delta \left(\tilde{u}_n + w_k\right). \end{align*}
Since $\tilde{\nu}_n \leq 1$ and $\tilde{u}_n + w_k \geq 0$, Lemma~\ref{leastaction} yields $u_n \leq \tilde{u}_n+w_k$.  Likewise
	\begin{align*} \nu_n &= \sigma_n + \Delta u_n  \\
 				    &= S_k \sigma_n + \Delta \left(u_n - w_k + \frac12 kM\delta_n^2 \right). \end{align*}
Since $\nu_n \leq 1$ and $u_n - w_k + \frac12 kM\delta_n^2 \geq 0$, Lemma~\ref{leastaction} yields $\tilde{u}_n \leq u_n - w_k + \frac12 kM\delta_n^2$.
\end{proof}

The next lemma adapts Theorem~\ref{odomconvergence} to our current setting.  The hypotheses are identical to those of Theorem~\ref{odomconvergence}, except that the convergence of smoothed densities (\ref{smootheddensityconvergence}) replaces (\ref{convergingdensities}).

\begin{lemma}
\label{unsmoothedconvergence}
Let $u_n$ be the odometer function for the divisible sandpile on $\delta_n \Z^d$ with initial density $\sigma_n$, and let $D_n =\{u_n>0\}$.
If $\sigma$ and $\sigma_n$ satisfy (\ref{absolutebound})-(\ref{discreteabsolutebound}), (\ref{uniformcompactsupport}) and (\ref{smootheddensityconvergence}), then
$u_n^\Box \to u := s-\gamma$ uniformly.  Moreover, for any $\epsilon>0$ we have $D_n \subset D_\epsilon^\Points$ for all sufficiently large $n$.
\end{lemma}

\begin{proof}
Let $\tilde{u}_n$ be the odometer function for the divisible sandpile on $\delta_n \Z^d$ started from initial density $S_{\alpha(n)} \sigma_n$.  Then $\tilde{u}^\Box_n \to u$ uniformly by Theorem~\ref{odomconvergence}.
By Lemma~\ref{odometerofsmoothing}, we have for all $z\in \R^d$
	\[ \left|u_n^\Box(z) - \tilde{u}^\Box_n(z)\right| \leq \frac12 M \alpha(n) \delta_n^2. \]
Since the right side tends to zero as $n \uparrow \infty$, we obtain $u_n^\Box \to u$ uniformly.

Let $\beta>0$ be the minimum value of $u$ on $\overline{D_\epsilon}$.  Taking $n$ large enough so that $|u_n - u^\Points| < \beta/2$, we have $u_n>\beta/2$ on $D_\epsilon^\Points$, hence $D_n \subset D_\epsilon^\Points$.
\end{proof}
        
\begin{lemma}
\label{coreoftheargument}
Fix $\epsilon>0$, and let $\beta>0$ be the minimum value of $u=s-\gamma$ on $\overline{D_\epsilon}$.  There exists $0<\eta < \epsilon$ such that for all sufficiently large $n$
	\[ f_{n,\eta}(z) \geq \frac12 \beta \delta_n^{-2}, \qquad z \in D_\epsilon^\Points. \]
\end{lemma}

\begin{proof}
Since $u$ is uniformly continuous on $\overline{D}$ and vanishes outside $D$, we can choose $\eta>0$ small enough so that $u \leq \beta/6$ outside $D_{\eta}$.  Since $u\geq \beta$ on $\overline{D_\epsilon}$, 
we have $\eta<\epsilon$.  Let $u_n$ be the odometer function for the divisible sandpile on $\delta_n \Z^d$ started from initial density $\sigma_n$, and let $D_n = \{u_n>0\}$ be the resulting domain of fully occupied sites in $\delta_n \Z^d$.  We have $u_n^\Box \rightarrow u$ uniformly by Lemma~\ref{unsmoothedconvergence}, so for $n$ sufficiently large we have $|u_n-u^\Points| \leq \beta/6$, hence $u_n \leq \beta/3$ on $\partial D_{\eta}^\Points$.  


By Lemma~\ref{unsmoothedconvergence} we have $D_\eta^\Points \subset D_n$ for all sufficiently large $n$.  Now (\ref{laplacianoff}) implies that the difference $\delta_n^2 f_{n,\eta}-u_n$ is harmonic on $D_{\eta}^\Points$, so it attains its minimum on the boundary.  Since $f_{n,\eta}$ vanishes on $\partial D_\eta^\Points$, we obtain for $z \in D_\eta^\Points$
	\begin{align*} \delta_n^2 f_{n,\eta}(z) &\geq u_n(z) - \frac{\beta}{3} \\
							       &\geq u(z) - \frac{\beta}{2}.  \end{align*}
Since $u \geq \beta$ in $D_\epsilon^\Points$, we conclude that $\delta_n^2 f_{n,\eta} \geq \beta/2$ in $D_\epsilon^\Points$.
\end{proof}

\begin{lemma}
\label{exittime}
We have
\[ \EE \widetilde{L}_\eta = \frac{\EE \tau_\eta^z}{g_{n,\eta}(z,z)}. \]
Moreover
\[ \EE \mathcal{M}_\eta \leq M \frac{\EE \tau_\eta^z}{g_{n,\eta}(z,z)}. \]
\end{lemma}

\begin{proof}
By (\ref{greensquotient}) and the symmetry of $g_{n,\eta}$, we have
	\begin{align} \label{exactforL}
	\EE \widetilde{L}_\eta &= \sum_{y \in D_\eta^\Points} \PP \big(\tau_z^y < \tau_\eta^y \big) \\
			&= \sum_{y \in D_\eta^\Points} \frac{g_{n,\eta}(y,z)}{g_{n,\eta}(z,z)} \nonumber \\
			&= \frac{1}{g_{n,\eta}(z,z)} \sum_{y \in D_\eta^\Points} g_{n,\eta}(z,y). \nonumber
	\end{align}
The sum in the last line is $\EE_z T_{\partial D_\eta^\Points}$.  

To prove the inequality for $\EE \mathcal{M}_\eta$, note that $\EE \mathcal{M}_\eta$ is bounded above by $M$ times the sum on the right side of (\ref{exactforL}).
\end{proof}

The next lemma, using a martingale argument to compute the expected time for simple random walk to exit a ball, is well known.

\begin{lemma}
\label{ballexittime}
Fix $r>0$ and let $B=B(o,r)^\Points \subset \delta_n \Z^d$, and let $T$ be the first hitting time of $\partial B$.  Then 
	\[ \EE_o T = \left( \frac{r}{\delta_n} \right)^2 + O \left( \frac{r}{\delta_n} \right). \]
\end{lemma}

\begin{proof}
Since $\delta_n^{-2} |X_t|^2 - t$ is a martingale with bounded increments, and $\EE_o T < \infty$, by optional stopping we have
	\[ \EE_o T = \delta_n^{-2} \EE_o |X_T|^2 = \delta_n^{-2} (r + O(\delta_n))^2. \qed \]
\renewcommand{\qedsymbol}{}
\end{proof}

The next lemma, which bounds the expected number of times simple random walk returns to the origin before reaching distance $r$, is also well known.

\begin{lemma}
\label{ballgreens}
Fix $r>0$ and let $B=B(o,r)^\Points \subset \delta_n \Z^d$, and let $G_B$ be the Green's function for simple random walk stopped on exiting $B$.  If $n$ is sufficiently large, then
	\[ G_B(o,o) \leq \log \frac{r}{\delta_n}. \]
\end{lemma}

\begin{proof}
In dimension $d\geq 3$ the result is trivial since simple random walk on $\delta_n \Z^d$ is transient.  In dimension two, consider the function
	\[ f(x) = G_B(o,x) - g_n(o,x) \]
where $g_n$ is the rescaled potential kernel defined in (\ref{definitionofg_ndimension2}).  
Since $f$ is harmonic in $B$ it attains its maximum on the boundary, hence by Lemma~\ref{greensfunctionconvergence} we have for $x \in B$
	\[ f(x) \leq \frac{2}{\pi} \log r + O \left ( \frac{\delta_n^2}{r^2} \right). \]
Since $g_n(o,o) = \frac{2}{\pi} \log \delta_n$, the result follows on setting $x=o$.
\end{proof}

We will use the following large deviation bound; see \cite[Cor.\ A.14]{AS}. 

\begin{lemma}
\label{alonspencer}
If $N$ is a sum of finitely many independent indicator random variables, then for all $\lambda>0$
	\[ \PP(|N-\EE N|>\lambda \EE N) < 2e^{-c_\lambda \EE N} \]
where $c_\lambda>0$ is a constant depending only on $\lambda$.
\end{lemma}

Let
	\[ \tilde{I}_n = \big\{X_{\nu^i}^i| \nu^i<\tilde{\tau}^i \big\} \subset I_n \]
where $\nu^i$ is given by (\ref{stoppingtimefortheithparticle}), and
	\[ \tilde{\tau}^i = \inf \big\{t \geq 0| X_t^i \notin \widetilde{D}^\Points \big\}. \]
The inner estimate of Theorem~\ref{IDLAconvergence} follows immediately from the lemma below.  Although for the inner estimate it suffices to prove Lemma~\ref{stronginnerestimate} with $I_n$ in place of $\tilde{I}_n$, we will make use of the stronger statement with $\tilde{I}_n$ in the proof of the outer estimate in the next section.

\begin{lemma}
\label{stronginnerestimate}
For any $\epsilon>0$,
	\[ \PP \big(\widetilde{D}_\epsilon^\Points \subset \tilde{I}_n \text{\em~for all but finitely many $n$}\big) = 1. \]
\end{lemma}

\begin{proof}
For $z \in \widetilde{D}_\epsilon^\Points$, let $\mathcal{E}_z(n)$ be the event that $z \notin \tilde{I}_n$.  By Borel-Cantelli it suffices to show that
	\begin{equation} \label{borelcantelliwiths} \sum_{n \geq 1} \sum_{z \in \widetilde{D}_\epsilon^\Points} \PP(\mathcal{E}_z(n)) < \infty. \end{equation}
By Lemma~\ref{pointset}, since $\widetilde{D} = D \cup \{\sigma \geq 1\}^o$ we have
	\[ \widetilde{D}_\epsilon \subset D_{\epsilon'} \cup \{\sigma \geq 1\}_{\epsilon'} \]
for some $\epsilon'>0$.  By (\ref{rotorsillytechnicality}), for $n \geq N(\epsilon')$ the terms in (\ref{borelcantelliwiths}) with $z \in \{\sigma \geq 1\}_{\epsilon'}$ vanish, so it suffices to show
	\begin{equation} \label{borelcantelliwithd} \sum_{n \geq 1} \sum_{z \in D_{\epsilon'}^\Points} \PP(\mathcal{E}_z(n)) < \infty. \end{equation}

By Lemma~\ref{coreoftheargument} there exists $0<\eta < \epsilon'$ such that 
	\begin{equation} \label{fromthecore} f_{n,\eta}(z) \geq \frac12 \beta \delta_n^{-2}, \qquad z \in D_{\epsilon'}^\Points \end{equation}
for all sufficiently large $n$, where $\beta>0$ is the minimum value of $u$ on $\overline{D_{\epsilon'}}$.  Fixing $z \in D_{\epsilon'}^\Points$, since $L_\eta \leq \widetilde{L}_\eta$ we have
	\begin{align} \PP(\mathcal{E}_z(n)) &\leq \PP(\mathcal{M}_\eta = L_\eta) \nonumber \\
						 &\leq \PP(\mathcal{M}_\eta \leq \widetilde{L}_\eta) \nonumber \\
						 &\leq \PP(\mathcal{M}_\eta \leq a) + \PP(\widetilde{L}_\eta \geq a) \label{choiceofa} \end{align}
for a real number $a$ to be chosen below.  By Lemma~\ref{alonspencer}, since $\widetilde{L}_\eta$ and $\mathcal{M}_\eta$ are sums of independent indicators, we have
\begin{align} & \PP( \widetilde{L}_\eta \geq (1+\lambda) \EE \widetilde{L}_\eta) < 2e^{-c_\lambda \EE \widetilde{L}_\eta} \label{largedevbounds} \\
			&\PP(\mathcal{M}_\eta \leq (1-\lambda) \EE \mathcal{M}_\eta) < 2e^{-c_\lambda \EE \mathcal{M}_\eta} \nonumber
			 \end{align}
where $c_\lambda$ depends only on $\lambda$, chosen below.  Now since $z\in D_{\epsilon'}$ and $D$ is bounded by Lemma~\ref{occupieddomainisbounded}, we have
	\[ B(z,\epsilon'-\eta) \subset D_\eta \subset B(z,R) \]
hence by Lemma~\ref{ballexittime}
	\begin{equation} \label{quadraticexittimebound} \frac12 \left( \frac{\epsilon'-\eta}{\delta_n} \right)^2 \leq \EE \tau_\eta^z \leq \left( \frac{R}{\delta_n} \right)^2. \end{equation}
By Lemma~\ref{exittime} it follows that $\EE \mathcal{M}_\eta \leq MR^2/\delta_n^2 g_{n,\eta}(z,z)$.  Taking 
	\[ a = \EE \widetilde{L}_\eta + \frac{\beta}{4\delta_n^2 g_{n,\eta}(z,z)} \] 
in (\ref{choiceofa}), and letting $\lambda = \beta/4MR^2$ in (\ref{largedevbounds}), we have
	\begin{equation} \label{thisfollowsfromthedefinitions} \lambda \EE \widetilde{L}_\eta \leq \lambda \EE \mathcal{M}_\eta \leq \frac{\beta}{4 \delta_n^2 g_{n,\eta}(z,z)}, \end{equation}
hence
		\[ a \geq (1+\lambda) \EE \widetilde{L}_\eta. \]
Moreover, from (\ref{fromthecore}) we have
	\[ \EE \mathcal{M}_\eta - E\widetilde{L}_\eta = \frac{f_{n,\eta}(z)}{g_{n,\eta}(z,z)} 
						\geq \frac{\beta}{2\delta_n^2 g_{n,\eta}(z,z)}, \]
hence by (\ref{thisfollowsfromthedefinitions})
	\[ a \leq \EE \mathcal{M}_\eta - \frac{\beta}{4\delta_n^2 g_{n,\eta}(z,z)} \leq (1-\lambda) \EE \mathcal{M}_\eta. \]

Thus we obtain from (\ref{choiceofa}) and (\ref{largedevbounds})
	\begin{equation} \label{largedev} \PP(\mathcal{E}_z(n)) \leq 4 e^{-c_\lambda \EE \widetilde{L}_\eta}. \end{equation}
By Lemmas~\ref{exittime} and~\ref{ballgreens} along with (\ref{quadraticexittimebound})
	\begin{align*} \EE \widetilde{L}_\eta &= \frac{\EE_z T_{\partial D_\eta^\Points}}{g_{n,\eta}(z,z)} \\
		&\geq \frac12 \left( \frac{\epsilon'-\eta}{\delta_n} \right)^2 \frac{1}{\log(R/\delta_n)}. \end{align*}
Using (\ref{largedev}), the sum in (\ref{borelcantelliwithd}) is thus bounded by
	\[ \sum_{n \geq 1} \sum_{z \in D_{\epsilon'}^\Points} \PP (\mathcal{E}_z(n)) \leq
	   \sum_{n \geq 1} \delta_n^{-d} \omega_d R^d \cdot 4 \exp \left(-\frac{c_\lambda(\epsilon'-\eta)^2}{2\delta_n^2 \log(R/\delta_n)} \right) < \infty. \qed \]
\renewcommand{\qedsymbol}{}	
\end{proof}

\subsection{Outer Estimate}

For $x \in \Z^d$ write
	\[ Q(x,h) = \{y \in \Z^d \,:\, ||x-y||_{\infty} \leq h \} \]
for the cube of side length $2h+1$ centered at $x$.  According to the next lemma, if we start a simple random walk at distance $h$ from a hyperplane $H \subset \Z^d$, then the walk is fairly ``spread out'' by the time it hits $H$, in the sense that its chance of first hitting $H$ at any particular point $z$ has the same order of magnitude as $z$ ranges over a $(d-1)$-dimensional cube of side length order $h$.

\begin{lemma}
\label{boxshells}
Fix $y \in \Z^d$ with $y_1 = h$, and let $T$ be the first hitting time of the hyperplane $H=\{x\in \Z^d|x_1=0\}$.  Let $F = H \cap Q(y,h)$.   For any $z \in F$ we have  
	\[ \PP_y(X_T = z) \geq ah^{1-d} \]
for a constant $a$ depending only on $d$.
\end{lemma}

\begin{proof}
For fixed $w \in H$, the function
	\[ f(x) = \PP_x (X_T=w) \]
is harmonic in the ball $B=B(y,h)$.  By the Harnack inequality \cite[Theorem 1.7.2]{Lawler} we have
	\[ f(x) \geq c f(y), \qquad x \in B(y,h/2) \]
for a constant $c$ depending only on $d$.  By translation invariance, it follows that for $w' \in H \cap B(w,h/2)$, letting $x = y+w-w'$ we have
	\[ \PP_y (X_T = w') = f(x) \geq c f(y) = c \PP_y (X_T = w). \]
Iterating, we obtain for any $w' \in H \cap Q(w,h)$
	\begin{equation} \label{harnackincube} \PP_y (X_T = w') \geq c^{\sqrt{d}} \PP_y (X_T = w). \end{equation}
Let $T'$ be the first exit time of the cube $Q(y,h-1)$.  Since $F$ is a boundary face of this cube, we have $\{X_{T'} \in F\} \subset \{X_T \in F\}$.  Let $z_0$ be the closest point to $y$ in $H$.  Taking $w'=z_0$, we have $w' \in H \cap Q(w,h)$ whenever $w \in F$.  Summing (\ref{harnackincube}) over $w \in F$, we obtain
	\[ (2h+1)^{d-1} c^{-\sqrt{d}} \PP_y(X_T=z_0)  \geq \PP_y(X_T \in F) \geq \PP(X_{T'} \in F) = \frac{1}{2d}. \]
Now for any $z\in F$, taking $w=z_0$ and $w'=z$ in (\ref{harnackincube}), we conclude that
	\[ \PP_y(X_T = z) \geq c^{\sqrt{d}} \PP_y(X_T=z_0) \geq \frac{c^{2\sqrt{d}}}{2d}(2h+1)^{1-d}. \qed \]
\renewcommand{\qedsymbol}{}
\end{proof}

\begin{lemma}
\label{thetrenches}
Let $h,\rho$ be positive integers.  Let $N$ be the number of particles that ever visit the cube $Q(o,\rho)$ during the internal DLA process, if one particle starts at each site $y \in \Z^d - Q(o,\rho+h)$, and $k$ additional particles start at sites $y_1,\ldots, y_k \notin Q(o,\rho+h)$.  If $k \leq \frac14 h^d$, then 
	\[ N \leq \text{\em Binom}(k,p), \]
where $p<1$ is a constant depending only on $d$.
\end{lemma}

\begin{proof}
Let $F_j$ be a half-space defining a face of the cube $Q=Q(o,\rho)$, such that dist$(y_j,F_j) \geq h$.  Let $z_j\in F_j$ be the closest point to $y_j$ in $F_j$, and let $B_j = Q(z_j,h/2) \cap F_j^c$.  Let $Z_j$ be the random set of sites where the particles starting at $y_1, \ldots, y_j$ stop.  Since $\#B_j \geq \frac12 h^d \geq 2k$, there is a hyperplane $H_j$ parallel to $F_j$ and intersecting $B_j$, such that
	\begin{equation} \label{sparseslice} \# B_j \cap H_j \cap Z_{j-1} \leq \frac12 h^{d-1}. \end{equation}
Denote by $A_j$ the event that the particle starting at $y_j$ ever visits $Q$.  On this event, the particle must pass through $H_j$ at a site which was already occupied by an earlier particle.  Thus if $\{X_t\}_{t \geq 0}$ is the random walk performed by the particle starting at $y_j$, then
	\[ A_j \subset \{ X_T \in Z_{j-1} \}, \]
where $T$ is the first hitting time of $H_j$.  By Lemma~\ref{boxshells}, every site $z \in B_j\cap H_j$ satisfies
	\[ \PP (X_T = z) \geq ah^{1-d}. \]
From (\ref{sparseslice}), since $\# B_j \cap H_j \geq h^{d-1}$ we obtain $\PP(X_T \notin Z_{j-1}) \geq a/2$, hence
	\[ \PP(A_j | \mathcal{F}_{j-1}) \leq p \]
where $p=1-a/2$, and $\mathcal{F}_i$ is the $\sigma$-algebra generated by the walks performed by the particles starting at $y_1, \ldots, y_i$.  Thus we can couple the indicators $1_{A_j}$ with i.i.d. indicators $I_j \geq 1_{A_j}$ of mean $p$ to obtain
	\[ N = \sum_{j=1}^k 1_{A_j} \leq \sum_{j=1}^k I_j = \text{Binom}(k,p). \qed \]
\renewcommand{\qedsymbol}{}
\end{proof}

The next lemma shows that if few enough particles start outside a cube $Q(o,3\rho)$, it is highly unlikely that any of them will reach the smaller cube $Q(o,\rho)$.

\begin{lemma}
\label{cubedeath}
Let $\rho,k$ be positive integers with
	\[ k \leq \frac14 \big(1-p^{1/2d}\big)^d \rho^d \]
where $p<1$ is the constant in Lemma~\ref{thetrenches}.  Let $N$ be the number of particles that ever visit the cube $Q(o,\rho)$ during the internal DLA process, if one particle starts at each site $y \in \Z^d - Q(o,3\rho)$, and $k$ additional particles start at sites $y_1,\ldots, y_k \notin Q(o,3\rho)$.  Then
	\[ \PP(N>0) \leq c_0 e^{-c_1\rho} \]
where $c_0,c_1>0$ are constants depending only on $d$.
\end{lemma}

\begin{proof}
Let $N_j$ be the number of particles that ever visit the cube $Q_j = Q(o, \rho_j)$, where
	\[ \rho_j = \big(2+p^{j/2d}\big)\rho. \]
Let $k_j = p^{j/2}k$, and let $A_j$ be the event that $N_j \leq k_j$.  Taking $h = \rho_j - \rho_{j+1}$ in Lemma~\ref{thetrenches}, since
	\[ k_j \leq \frac14 p^{j/2} \big(1-p^{1/2d}\big)^d \rho^d = \frac14 h^d \]
we obtain
	\[ N_{j+1}1_{A_j} \leq \text{Binom}(N_j,p). \]
Hence
	\begin{align} \label{conditionalprobintermsofbinomial} \PP(A_{j+1}|A_j) &\geq \PP \big(\text{Binom}(k_j, p) \leq k_{j+1} \big) \nonumber \\		
		&\geq 1-2e^{-ck_j} \end{align}
where in the second line we have used Lemma~\ref{alonspencer} with $\lambda=\sqrt{p}-p$.

Now let 
	\[ j = \left\lfloor 2 \frac{\log k - \log \rho}{\log(1/p)} \right\rfloor \]
so that $p^{1/2} \rho \leq k_j \leq \rho$.  On the event $A_j$, at most $\rho$ particles visit the cube $Q(o,2\rho)$.  Since the first particle to visit each cube $Q(o,2\rho-i)$ stops there,  at most $\rho-i$ particles visit $Q(o,2\rho-i)$.  Taking $i=\rho$ we obtain $\PP(N=0) \geq \PP(A_j)$.  
 From (\ref{conditionalprobintermsofbinomial}) we conclude that
	\[ \PP(N=0) \geq \PP(A_1) \PP(A_2 | A_1) \cdots \PP(A_j|A_{j-1}) \geq 1 - 2je^{-c\rho}. \]
The right side is at least $1-c_0e^{-c_1\rho}$ for suitable constants $c_0,c_1$.
\end{proof}

\begin{proof}[Proof of Theorem~\ref{IDLAconvergence}] The inner estimate is immediate from Lemma~\ref{stronginnerestimate}.  For the outer estimate, let
	\[ N_n = \# \big\{1 \leq i \leq m_n \big| \nu^i \geq \tilde{\tau}^i\big\} = m_n - \# \tilde{I}_n \]
be the number of particles that leave $\widetilde{D}^\Points$ before aggregating to the cluster.
Let $K_0 = \frac14 \left(\frac{2(1-p^{1/2d})}{3\sqrt{d}}\right)^d$, where $p$ is the constant in Lemma~\ref{thetrenches}.  For $\epsilon>0$ and $n_0 \geq 1$, consider the event 
	\[ F_{n_0} = \big\{ N_n \leq K_0 (\epsilon/\delta_n)^d \text{ for all } n \geq n_0 \big\}. \]
By (\ref{rotorboundedawayfrom1}) and (\ref{rotorclosureofinterior}), the closure of $\widetilde{D}$ contains the support of $\sigma$, so by Proposition~\ref{boundaryregularity},
	\begin{equation} \label{fromconservationofmass} \delta_n^d m_n = \delta_n^d \sum_{x \in \delta_n \Z^d} \sigma_n(x) \rightarrow \int_{\R^d} \sigma(x) dx = \Leb(\widetilde{D}). \end{equation}
Moreover, by Proposition~\ref{boundaryregularity}(i), for sufficiently small $\eta$ we have 
	\[ \Leb\big(\widetilde{D}-\widetilde{D}_{2\eta}\big) \leq \frac12 K_0 \epsilon^d. \]
Taking $n$ large enough so that $\delta_n<\eta$, we obtain
	\[ \delta_n^d \#\widetilde{D}_\eta^\Points = \Leb \big(\widetilde{D}_\eta^{\Points\Box}\big) \geq \Leb\big(\widetilde{D}_{2\eta}\big) \geq \Leb\big(\widetilde{D}\big) - \frac12 K_0 \epsilon^d. \]
Thus for sufficiently large $n$, on the event $\widetilde{D}_\eta^\Points \subset \tilde{I}_n$ we have
	\begin{align*} N_n &\leq m_n - \# \widetilde{D}_\eta^\Points \\
	 	&\leq m_n - \delta_n^{-d} \Leb\big(\widetilde{D}\big) + \frac12 K_0 \epsilon^d \delta_n^{-d} \\
		& \leq K_0 \epsilon^d \delta_n^{-d}, \end{align*}
where in the last line we have used (\ref{fromconservationofmass}).  From Lemma~\ref{stronginnerestimate} we obtain
	\begin{equation} \label{probtendingto1} \PP\big(F_{n_0}\big) \geq \PP\big( \widetilde{D}_\eta^\Points \subseteq \tilde{I}_n \text{ for all } n \geq n_0 \big) \uparrow 1 \end{equation}
as $n_0 \uparrow \infty$.

By compactness, we can find finitely many cubes $Q_1, \ldots, Q_j$ of side length $\ell = 2\epsilon/3\sqrt{d}$ centered at points $x_1, \ldots, x_j \in \partial (\widetilde{D}^\epsilon)$, with $\partial (\widetilde{D}^\epsilon) \subset \bigcup Q_i$.  Taking $k= \floor{K_0(\epsilon/\delta_n)^d}$ and $\rho = \ceil{\ell/\delta_n}$ in Lemma~\ref{cubedeath}, since the cube of side length $3\ell_i$ centered at $x_i$ is disjoint from $\widetilde{D}$, we obtain for $n \geq n_0$
	\[ \PP\big(\{Q_i \cap I_n \neq \emptyset\} \cap F_{n_0}\big) \leq c_0e^{-c_1\ell/\delta_n}. \]
Since $\big\{ I_n \not\subset \widetilde{D}^{\epsilon\Points} \big\} \subset \bigcup_{i=1}^j \big\{ Q_i \cap I_n \neq \emptyset \big\}$, summing over $i$ yields
	\begin{align*} \PP \big(\big\{ I_n \not\subset \widetilde{D}^{\epsilon\Points} \big\} \cap F_{n_0}\big) 
	&\leq c_0 j e^{-c_1\ell/\delta_n}. \end{align*}
By Borel-Cantelli, if $G$ is the event that $I_n \not\subset \widetilde{D}^{\epsilon\Points}$ for infinitely many $n$, then $\PP\big(F_{n_0} \cap G\big) = 0$.  From (\ref{probtendingto1}) we conclude that $\PP(G)=0$.
\end{proof}

\section{Multiple Point Sources}
\label{multiplepointsourcessection}

This section is devoted to proving Theorem~\ref{multiplepointsources}.  We further show in Proposition~\ref{ballquadrature} that the smash sum of balls arising in Theorem~\ref{multiplepointsources} is a classical quadrature domain.  In particular, the boundary of the smash sum of~$k$ disks in~$\R^2$ lies on an algebraic curve of degree at most~$2k$.

\subsection{Associativity and Continuity of the Smash Sum}

In this section we establish a few basic properties of the smash sum (\ref{smashsumdef}).  In Lemma~\ref{associativity} we show that the smash sum is associative.  In Lemma~\ref{volumecontinuity} we show that the smash sum is continuous with respect to the metric on bounded open sets $A,B \subset \R^d$ given by
	\[ d_{\Delta}(A,B) = \Leb (A ~\Delta~ B) \]
where $A ~\Delta~ B$ denotes the symmetric difference $A\cup B - A\cap B$, and $\Leb$ is Lebesgue measure on~$\R^d$.  Lastly, in Lemma~\ref{sumofneighborhoods} we show that the smash sum is also continuous in a variant of the Hausdorff metric.

\begin{lemma}
\label{associativity}
Let $A,B,C \subset \R^d$ be bounded open sets whose boundaries have Lebesgue measure zero.  Then
	\begin{align*} (A \oplus B) \oplus C &= A \oplus (B \oplus C) \\
							        &= A \cup B \cup C \cup D \end{align*}
where $D$ is given by (\ref{thenoncoincidenceset}) with $\sigma = 1_A + 1_B + 1_C$.
\end{lemma}

\begin{proof}
Let
	\[ \gamma(x) = -|x|^2 - G(1_A + 1_B)(x) \]
and
	\[ \hat{\gamma}(x) = -|x|^2 - G(1_{A\oplus B} + 1_C)(x). \]
Let $u = s-\gamma$ and $\hat{u} = \hat{s}-\hat{\gamma}$, where $s,\hat{s}$ are the least superharmonic majorants of $\gamma,\hat{\gamma}$.  Then
	\begin{align} (A \oplus B) \oplus C &= (A \oplus B) \cup C \cup \{\hat{u}>0\} \nonumber \\
							&= A \cup B \cup \{u>0\} \cup C \cup \{\hat{u}>0\} \nonumber \\
\label{sumofodometers}			&= A \cup B \cup C \cup \{u+\hat{u}>0\}. \end{align}
Let $\nu_n$ be the final mass density for the divisible sandpile in $\delta_n \Z^d$ started from initial density $1_{A^\Points} + 1_{B^\Points}$, and let $u_n$ be the corresponding odometer function.  By Theorem~\ref{odomconvergence} we have $u_n \to u$ as $n \to \infty$.  Moreover, by Theorem~\ref{domainconvergence} we have
	\begin{equation} \label{restatementofdomainconvergence} \nu_n^\Box(x) \to 1_{A\oplus B}(x) \end{equation}
for all $x \notin \partial(A \oplus B)$.  Let $\hat{u}_n$ be the odometer function for the divisible sandpile on $\delta_n \Z^d$ started from initial density $\nu_n + 1_{C^\Points}$.  By Proposition~\ref{boundaryregularity}(i) the right side of (\ref{restatementofdomainconvergence}) is continuous almost everywhere, so by Theorem~\ref{odomconvergence} we have $\hat{u}_n \to \hat{u}$ as $n \to \infty$.  On the other hand, by the abelian property, Lemma~\ref{abelianproperty}, the sum $u_n + \hat{u}_n$ is the odometer function for the divisible sandpile on $\delta_n \Z^d$ started from initial density $1_{A^\Points} + 1_{B^\Points} + 1_{C^\Points}$, so by Theorem~\ref{odomconvergence} we have $u_n + \hat{u}_n \to \tilde{u} := \tilde{s} - \tilde{\gamma}$, where
	\[ \tilde{\gamma}(x) = -|x|^2 - G(1_A + 1_B + 1_C)(x) \]
and $\tilde{s}$ is the least superharmonic majorant of $\tilde{\gamma}$.  In particular, 
	\[ u+\hat{u} = \lim_{n \to \infty} (u_n + \hat{u}_n) = \tilde{u}, \] 
so the right side of (\ref{sumofodometers}) 
is equal to $A \cup B \cup C \cup D$.
\end{proof}

\begin{lemma}
\label{volumecontinuity}
Let $A_1,A_2,B$ be bounded open subsets of $\R^d$ whose boundaries have measure zero.  Then
	\[	d_\Delta(A_1\oplus B, A_2 \oplus B) \leq d_\Delta(A_1, A_2). \]
\end{lemma}

\begin{proof}
Let $A=A_1 \cup A_2$, and $A'_i = A-A_i$ for $i=1,2$.  Then by Lemma~\ref{associativity},
	\[ A \oplus B =  (A'_i \oplus A_i) \oplus B =  A'_i \oplus (A_i \oplus B). \]
By Corollary~\ref{volumesadd},
	\[ \Leb(A \oplus B) = \Leb(A'_i) + \Leb(A_i \oplus B) \]
hence
	\[ d_\Delta (A \oplus B,  A_i \oplus B) = \Leb(A'_i). \]
By the triangle inequality, we conclude that
	\[ d_\Delta (A_1 \oplus B, A_2 \oplus B) \leq \Leb (A'_1) + \Leb (A'_2) =  d_\Delta(A_1,A_2). \qed \]
\renewcommand{\qedsymbol}{}
\end{proof}

The following lemma shows that the smash sum of two sets with a small intersection cannot extend very far beyond their union.  As usual, $A^\epsilon$ denotes the outer $\epsilon$-neighborhood (\ref{neighborhoods}) of a set $A \subset \R^d$.

\begin{lemma}
\label{limitedexpansion}
Let $A,B \subset \R^d$ be bounded open sets whose boundaries have measure zero, and let $\rho = (\Leb (A \cap B))^{1/d}$.  There is a constant $c$, independent of $A$ and $B$, such that
	\[ A\oplus B \subset (A \cup B)^{c\rho}. \]
\end{lemma}

\begin{proof}
Let $\gamma,s$ be given by (\ref{theobstacle}) and (\ref{themajorant}) with $\sigma = 1_A+1_B$, and write $u=s-\gamma$.  Fix $x \in (A\oplus B)-(A\cup B)$ and let $r =$ dist$(x,A \cup B)$.  Let $\mathcal{B} = B(x,r/2)$.  By Lemma~\ref{majorantbasicprops}(iii), $s$ is harmonic in $A\oplus B-(A\cup B)$, so by Lemma~\ref{superharmonicpotential} the function
	\begin{align*} w(y) &= u(y) - |x-y|^2 \\  
				&= s(y) + |y|^2 + G\big(1_{A}+1_{B}\big)(y) - |x-y|^2 \end{align*}
is harmonic on the intersection $(A\oplus B) \cap \mathcal{B}$; hence it attains its maximum on the boundary.  Since $w(x)>0$ the maximum cannot be attained on $\partial (A\oplus B)$, so it is attained at some point $y \in \partial \mathcal{B}$, and
	\begin{equation} \label{largeodom} u(y) \geq w(y) + \frac{r^2}{4} > \frac{r^2}{4}. \end{equation}

If $z$ is any point outside $A\oplus B$, then by Lemma~\ref{laplacianofodometer} and Lemma~\ref{continuumatmostquadratic} with $\lambda=4d$, there is a constant $c' \geq 1$ such that $u \leq c' h^2$ on $B(z,h)$ for all $h$.  Taking $h=r/2\sqrt{c'}$, we conclude from (\ref{largeodom}) that $B(y,h) \subset A\oplus B$.
Since $B(y,h)$ is disjoint from $A\cup B$, we have by Corollary~\ref{volumesadd}
	\[ \Leb(A\cup B) + \omega_d h^d \leq \Leb(A\oplus B) = \Leb(A) + \Leb(B), \]
hence $\omega_d h^d \leq \Leb(A \cap B)$.  Taking $c=2 \sqrt{c'}/\omega_d^{1/d}$ yields $r = c\, \omega_d^{1/d} h \leq c\rho$.
\end{proof}

The next lemma, together with Lemma~\ref{monotonicity}, shows that the smash sum is continuous with respect to the metric
	\[ d_{io}(A,B) = \inf \big\{r \geq 0 \,|\, A_r \subset B \subset A^r \mbox{ and } B_r \subset A \subset B^r \big\} \]
on bounded open sets $A,B \subset \R^d$.
Here $A_r$ and $A^r$ denote the inner and outer neighborhoods (\ref{neighborhoods}) of~$A$.  This metric, which we call the \emph{inner-outer Hausdorff metric}, is related to the usual Hausdorff metric $d_H$ by
	\[ d_{io}(A,B) = \max \big(d_H\big(\bar{A},\bar{B}\big), d_H(A^c, B^c)\big). \]

\begin{lemma}
\label{sumofneighborhoods}
Let $A,B \subset \R^d$ be bounded open sets whose boundaries have measure zero.  For any $\epsilon>0$ there exists $\eta>0$ such that
	\begin{equation} \label{threeinclusions} (A \oplus B)_\epsilon \subset A_\eta \oplus B_\eta \subset A^\eta \oplus B^\eta \subset (A \oplus B)^\epsilon. \end{equation}
\end{lemma}

\begin{proof}
By Lemma~\ref{pointset}, since $A\oplus B = A\cup B\cup D$, we have
	\[ (A\oplus B)_\epsilon \subset A_{\epsilon'} \cup B_{\epsilon'} \cup D_{\epsilon'} \]
for some $\epsilon'>0$.  By Lemma~\ref{threesteps}(iii) with $\sigma = 1_A+1_B$ and $\sigma_n = 1_{A_{1/n}} + 1_{B_{1/n}}$, for sufficiently small $\eta$ we have $D_{\epsilon'} \subset A_\eta \oplus B_\eta$.  This proves the first inclusion in (\ref{threeinclusions}).

The second inclusion is immediate from Lemma~\ref{monotonicity}. 

 For the final inclusion, write $A' = A^\eta - A$ and $B' = B^\eta - B$.  Since  $\Leb(A') \downarrow \Leb(\partial A)=0$ and $\Leb(B') \downarrow \Leb(\partial B) = 0$ as $\eta \downarrow 0$, for small enough $\eta$ we have by Lemmas~\ref{associativity} and~\ref{limitedexpansion}
	 \begin{align*} A^\eta \oplus B^\eta &= A \oplus B \oplus A' \oplus B' \\
	 		&\subset ((A \oplus B) \cup A^\eta \cup B^\eta)^{\epsilon-\eta} \\
			&\subset (A \oplus B)^\epsilon. \qed \end{align*}
\renewcommand{\qedsymbol}{}
\end{proof}

%
%

\subsection{Smash Sums of Balls}

In this section we deduce Theorem~\ref{multiplepointsources} from our other results, using the main results of \cite{LBG} and \cite{LP}.  
The following theorem collects the results of \cite[Theorem 1.3]{LP} for the divisible sandpile, \cite[Theorem 1.1]{LP} for the rotor-router model, and \cite[Theorem 1]{LBG} for internal DLA.  

\begin{theorem}
\label{singlepointsource}
Fix $\lambda>0$ and a sequence $\delta_n \downarrow 0$, and let $D_n$ be the domain of fully occupied sites for the divisible sandpile in $\delta_n \Z^d$ started from mass $m=\big\lfloor \lambda \delta_n^{-d}\big\rfloor$ at the origin.  Likewise, let $R_n,I_n$ be the domains of occupied sites in $\delta_n \Z^d$ formed respectively from the rotor-router model (with initial rotors configured arbitrarily) and internal DLA, starting with $m$ particles at the origin.  For any $\epsilon>0$, we have 
	\[ B_\epsilon^\Points \subset D_n, R_n \subset B^{\epsilon\Points} \qquad \text{\em for all sufficiently large $n$,} \] 
where $B$ is the ball of volume $\lambda$ centered at the origin in $\R^d$.  Moreover, if $\delta_n < 1/n$, then with probability one
	\[ B_\epsilon^\Points \subset I_n \subset B^{\epsilon\Points} \qquad \text{\em for all sufficiently large $n$.} \] 
\end{theorem}

\begin{proof}[Proof of Theorem~\ref{multiplepointsources}]
For $i=1,\ldots,k$ let $D_n^i, R_n^i, I_n^i$ be the domain of occupied sites in $\delta_n \Z^d$ 
starting from a single point source of $\big\lfloor \delta_n^{-d} \lambda_i\big\rfloor$ particles at $x_i^\Points$.
By Theorem~\ref{singlepointsource}, for any $\eta>0$ we have with probability one
	\begin{equation} \label{singlepointsources} (B_i)_\eta^\Points \subset D_n^i, R_n^i \subset (B_i)^{\eta\Points} \qquad \text{for all $i$ and all sufficiently large $n$;} \end{equation}
and if $\delta_n < 1/n$, then with probability one
	\begin{equation} \label{singlepointsourcesIDLA} (B_i)_\eta^\Points \subset I_n^i \subset (B_i)^{\eta\Points} \qquad \text{for all $i$ and all sufficiently large $n$.} \end{equation}

Next we argue that the domains $D_n,R_n,I_n$ can be understood as smash sums of $D_n^i, R_n^i, I_n^i$ as $i$ ranges over the integers $1, \ldots, k$.
By the abelian property \cite{DF}, the domain $I_n$ is the Diaconis-Fulton smash sum in $\delta_n \Z^d$ of the domains $I_n^i$.  Likewise, if $r_1$ is an arbitrary rotor configuration on $\delta_n \Z^d$, let $S_n^2$ be the smash sum of $R_n^1$ and $R_n^2$ formed using rotor-router dynamics with initial rotor configuration $r_1$, and let $r_2$ be the resulting final rotor configuration.  For $i \geq 3$ define $S_n^i$ inductively as the smash sum of $S_n^{i-1}$ and $R_n^i$ formed using rotor-router dynamics with initial rotor configuration $r_{i-1}$, and let $r_i$ be the resulting final rotor configuration.  Then $R_n=S_n^{k}$.  Finally, by Lemma~\ref{abelianproperty}, the domain $D_n$ contains the smash sum of domains $D_n^i$ formed using divisible sandpile dynamics, and $D_n$ is contained in the smash sum of the domains $D_n^i \cup \partial D_n^i$.  

Fixing $\epsilon>0$, by Theorem~\ref{DFsum} and Lemma~\ref{monotonicity}, it follows from (\ref{singlepointsources}) and (\ref{singlepointsourcesIDLA}) that for all sufficiently large $n$
	\begin{equation} \label{squeezed} A_{\epsilon/2}^\Points \subset D_n, R_n, I_n \subset \widetilde{A}^{\epsilon/2 \Points} \end{equation}
where $A$ is the smash sum of the balls $(B_i)_\eta$, and $\widetilde{A}$ is the smash sum of the balls $(B_i)^\eta$.  By Lemma~\ref{sumofneighborhoods} we can take $\eta$ sufficiently small so that
	\[ D_{\epsilon/2} \subset A \subset \widetilde{A} \subset D^{\epsilon/2} \]
where $D = B_1 \oplus \ldots \oplus B_k$.  Together with (\ref{squeezed}), this completes the proof.
\end{proof}

\begin{prop}
\label{ballquadrature}
Fix $x_1, \ldots, x_k \in \R^d$ and $\lambda_1, \ldots, \lambda_k>0$.  Let $B_i$ be the ball of volume $\lambda_i$ centered at $x_i$.  Then
	\[ \int_{B_1 \oplus \ldots \oplus B_k} h(x) dx \leq \sum_{i=1}^k \lambda_i h(x_i) \]
for all integrable superharmonic functions $h$ on $B_1 \oplus \ldots \oplus B_k$.
\end{prop}

We could deduce Proposition~\ref{ballquadrature} from Proposition~\ref{boundaryregularitysmooth} by integrating smooth approximations to $h$ against smooth densities $\sigma_n$ converging to the sum of the indicators of the $B_i$.  However, it more convenient to use the following result of M. Sakai, which is proved by a similar type of approximation argument. 

\begin{theorem} \cite[Thm.\ 7.5]{Sakai}
\label{fromsakai}
Let $\Omega \subset \R^d$ be a bounded open set, and let $\sigma$ be a bounded function on $\R^d$ supported on $\bar{\Omega}$ satisfying $\sigma>1$ on $\Omega$.  Let $\gamma, s, D$ be given by (\ref{theobstacle})-(\ref{thenoncoincidenceset}).  
	\[ \int_D h(x) dx \leq \int_D h(x) \sigma(x) dx \]
for all integrable superharmonic functions $h$ on $D$.
\end{theorem}

We remark that the form of the obstacle in \cite{Sakai} is superficially different from ours: taking $w=1$ and $\omega=\sigma$ in section~7 of \cite{Sakai}, the obstacle is given by
	\[ \psi(x) = G1_B - G\sigma \]
for a large ball $B$, rather than our choice (\ref{theobstacle}) of
	\[ \gamma(x) = -|x|^2 - G\sigma. \]
Note, however that $\gamma-\psi$ is constant on $B$ by (\ref{ballpotential}) and (\ref{ballpotentialdim2}).  Thus if $B$ is sufficiently large, the two obstacle problems have the same noncoincidence set $D$ by Lemmas~\ref{majorantonacompactset} and~\ref{occupieddomainisbounded}.

To prove Proposition~\ref{ballquadrature} using Theorem~\ref{fromsakai}, we must produce an appropriate density $\sigma$ on $\R^d$ strictly exceeding $1$ on its support.  We will take $\sigma$ to be twice the sum of indicators of balls of half the volume of the $B_i$.

\begin{proof}[Proof of Proposition~\ref{ballquadrature}]
Let $B'_i$ be the ball of volume $\lambda_i/2$ centered at $x_i$, and consider the sum of indicators
	\[ \sigma = 2 \sum_{i=1}^k 1_{B'_i}. \]
Let $\gamma, s, D$ be given by (\ref{theobstacle})-(\ref{thenoncoincidenceset}).  By Theorem~\ref{fromsakai} and the mean value proprerty for superharmonic functions (\ref{meanvalueproperty}), we have
	\[ \int_D h(x) dx \leq 2 \sum_{i=1}^k \int_{B'_i} h(x) dx \leq \sum_{i=1}^k \lambda_i h(x_i) \]
for all integrable superharmonic functions $h$ on $D$.

It remains to show that $D = B_1 \oplus \ldots \oplus B_k$.  By Lemma~\ref{startingdensitygreaterthan1} we have $B'_i \subset D$ for all $i$, hence by Lemma~\ref{associativity}
	\[ D = B'_1 \oplus B'_1 \oplus \ldots \oplus B'_k \oplus B'_k. \]
By Lemma~\ref{relaxingaball} we have $B_i = B'_i \oplus B'_i$, completing the proof. 
\end{proof}

From Proposition~\ref{ballquadrature} it follows \cite[sec.\ 6]{Gustafsson83} (see also \cite[Lemma 1.1(a)]{Gustafsson88}) that if $x_1, \ldots, x_k$ are distinct points in $\R^2$, and $B_i$ is a disk centered at $x_i$, the boundary of the smash sum $B_1 \oplus \ldots \oplus B_k$ lies on an algebraic curve of degree $2k$.  The methods of \cite{Gustafsson83} rely heavily on complex analysis, which is why they apply only in dimension two.  Indeed, in higher dimensions it is not known whether classical quadrature domains have boundaries given by algebraic surfaces \cite[Ch.\ 2]{Shapiro}.

\section{Potential Theory Proofs}
\label{potentialtheoryproofs}

In this section we collect the proofs of the background results stated in section~\ref{potentialtheorybackground}.  For more background on potential theory in general, we refer the reader to \cite{HFT,Doob}; for the obstacle problem in particular, see \cite{Caffarelli,Friedman}.

\subsection{Least Superharmonic Majorant}

\begin{proof}[Proof of Lemma~\ref{majorantbasicprops}]
(i) Let $f \geq \gamma$ be continuous and superharmonic.  Then $f \geq s$.  By the mean value property (\ref{meanvalueproperty}), we have
	\[ f \geq A_r f \geq A_r s. \]
Taking the infimum over $f$ on the left side, we conclude that $s \geq A_r s$.  

It remains to show that $s$ is lower-semicontinuous.  Let
	\[ \omega(\gamma,r) = \sup_{x,y \in \R^d, |x-y|\leq r} |\gamma(x)-\gamma(y)|. \]
Since
	\[ A_r s \geq A_r \gamma \geq \gamma - \omega(\gamma,r) \]
the function $A_r s + \omega(\gamma,r)$ is continuous, superharmonic, and lies above $\gamma$, so
	\[ A_r s \leq s \leq A_r s + \omega(\gamma,r). \]
Since $\gamma$ is uniformly continuous, we have $\omega(\gamma,r) \downarrow 0$ as $r \downarrow 0$, hence $A_r s \rightarrow s$ as $r \downarrow 0$.  Moreover if $r_0<r_1$, then by Lemma~\ref{basicproperties}(iii)
	\[ A_{r_0} s = \lim_{r \rightarrow 0} A_{r_0} A_r s \geq \lim_{r \rightarrow 0} A_{r_1} A_r s = A_{r_1} s. \]
Thus $s$ is an increasing limit of continuous functions and hence lower-semicontinuous.

(ii) Since $s$ is defined as an infimum of continuous functions, it is also upper-semicontinuous.

(iii) Given $x \in D$, write $\epsilon = s(x)-\gamma(x)$.  Choose $\delta$ small enough so that for all $y \in B = B(x,\delta)$
	\[ |\gamma(x)-\gamma(y)| < \frac{\epsilon}{2} \quad\text{and}\quad |s(x)-s(y)| < \frac{\epsilon}{2}. \]
Let $f$ be the continuous function which is harmonic in $B$ and agrees with $s$ outside $B$.  By Lemma~\ref{basicproperties}(v), $f$ is superharmonic.  By Lemma~\ref{basicproperties}(i), $f$ attains its minimum in $\overline{B}$ at a point $z \in \partial B$, hence for $y\in B$
	\[ f(y) \geq f(z) = s(z) \geq s(x) - \frac{\epsilon}{2} = \gamma(x) + \frac{\epsilon}{2} > \gamma(y). \]
It follows that $f \geq\gamma$ everywhere, hence $f \geq s$.  From Lemma~\ref{basicproperties}(ii) we conclude that $f=s$, and hence $s$ is harmonic at $x$.
\end{proof}

\begin{proof}[Proof of Lemma~\ref{superharmonicpotential}]
Suppose $B(x,r) \subset \Omega$.  Since for any fixed $y$ the function $f(x) = g(x,y)$ is superharmonic in $x$, we have
	\begin{align} G\sigma(x) &= \int_{\R^d} \sigma(y) g(x,y) dy \nonumber \\
			&\geq \int_{\R^d} \sigma(y) A_r f(y) dy \nonumber \\
			&= A_r G\sigma(x). \qed \nonumber \end{align}
\renewcommand{\qedsymbol}{}
\end{proof}

\begin{proof}[Proof of Lemma~\ref{laplacianofobstacle}]
(i) By Lemma~\ref{superharmonicpotential}, since $\sigma$ is nonnegative, the function $\gamma(x) + |x|^2 = -G\sigma(x)$ is subharmonic on $\R^d$.  

(ii) Let $B=B(o,R)$ be a ball containing the support of $\sigma$.  By (\ref{ballpotential}) and (\ref{ballpotentialdim2}), for $x \in B$ we have
	\[ |x|^2 = c_d R^2 - G1_B(x) \]
where $c_2 = 1-2\log R$ and $c_d = \frac{d}{d-2}$ for $d \geq 3$.  Hence for $x \in B$ we have
	\begin{align*} \gamma(x) - (M-1)|x|^2 &= -G\sigma(x) -M|x|^2 \\
								  &= G(M1_B - \sigma)(x) - c_d M R^2. \end{align*}
Since $\sigma \leq M1_B$ on $\Omega$, by Lemma~\ref{superharmonicpotential} the function $\gamma - (M-1)|x|^2$ is superharmonic in $B \cap \Omega$.  Since this holds for all sufficiently large $R$, it follows that $\gamma - (M-1)|x|^2$ is superharmonic on all of $\Omega$.
\end{proof}

\begin{proof}[Proof of Lemma~\ref{laplacianofodometer}]
(i) By Lemmas~\ref{majorantbasicprops}(i) and~\ref{laplacianofobstacle}(i), the function
	\[ u - |x|^2 = s - (\gamma + |x|^2) \]
is the difference of a superharmonic and a subharmonic function, hence superharmonic on $\R^d$.

(ii) With $A_r$ defined by (\ref{meanvalueproperty}), we have
	\begin{align*} A_r |x|^2 &= \frac{1}{\omega_d r^d} \int_{B(o,r)} (|x|^2 + 2x\cdot y + |y|^2) \,dy  \\ 
	&= |x|^2 + \frac{1}{\omega_d r^d} \int_0^r (d\omega_d t^{d-1})t^2 \,dt \\
	&= |x|^2 + \frac{dr^2}{d+2}. \end{align*}
By Lemma~\ref{laplacianofobstacle} we have
	\[ A_r s + \frac{dr^2}{d+2} \geq A_r \gamma + A_r |x|^2 - |x|^2 \geq \gamma. \]
Since $A_r s$ is continuous and superharmonic, it follows that $A_r s + \frac{dr^2}{d+2} \geq s$, hence
	\[ A_r s + A_r |x|^2 = A_r s + |x|^2 + \frac{dr^2}{d+2} \geq s + |x|^2. \]
Thus $s + |x|^2$ is subharmonic on $\R^d$, and hence by Lemma~\ref{laplacianofobstacle}(ii) the function
	\[ u + M|x|^2 = (s+|x|^2) - (\gamma - (M-1)|x|^2) \]
is subharmonic on $\Omega$.
\end{proof}

\begin{proof}[Proof of Lemma~\ref{startingdensitygreaterthan1}]
If $\sigma>1$ in a ball $B=B(x,r)$, by (\ref{ballpotential}) and (\ref{ballpotentialdim2}), for $y\in B$ we have
	\[ \gamma(y) = -|y|^2 - G\sigma(y) = - c_d r^2 - G(\sigma-1_B)(y). \]
By Lemma~\ref{superharmonicpotential} it follows that $\gamma$ is subharmonic in $B$.  In particular, $s>\gamma$ in $B$, so $x \in D$.
\end{proof}

\begin{proof}[Proof of Lemma~\ref{monotonicity}]
Let
	\[ \tilde{s} = s_2 + G(\sigma_2-\sigma_1). \]
Then $\tilde{s}$ is continuous and superharmonic, and since $s_2(x) \geq -|x|^2-G\sigma_2(x)$ we have
	\[ \tilde{s}(x) \geq - |x|^2 - G\sigma_1(x) \]
hence $\tilde{s} \geq s_1$.  Now 
	\[ u_2-u_1 = s_2 - s_1 + G(\sigma_2-\sigma_1)	 = \tilde{s}-s_1 \geq 0. \]
Since $D_i$ is the support of $u_i$, the result follows.
\end{proof}

\begin{proof}[Proof of Lemma~\ref{majorantonacompactset}]
Let $f$ be any continuous function which is superharmonic on $\Omega$ and $\geq \gamma$.  By Lemma~\ref{majorantbasicprops}(iii), $s$ is harmonic on $D$, so $f-s$ is superharmonic on $D$ and attains its minimum in $\overline{D}$ on the boundary.  Hence $f-s$ attains its minimum in $\overline{\Omega}$ at a point $x$ where $s(x)=\gamma(x)$.  Since $f \geq \gamma$ we conclude that $f \geq s$ on $\Omega$ and hence everywhere.  Thus $s$ is at most the infimum in (\ref{majorantinadomain}).
Since the infimum in (\ref{majorantinadomain}) is taken over a strictly larger set than that in (\ref{themajorant}), the reverse inequality is trivial.
\end{proof}

\subsection{Boundary Regularity}

For the proof below, we follow the sketch in Caffarelli \cite[Theorem 2]{Caffarelli}.

\begin{proof}[Proof of Lemma~\ref{smoothdensity}]
Fix $x_0 \in \partial D$, and define
 	\[ L(x) = \gamma(x_0) + \langle \nabla \gamma(x_0), x-x_0 \rangle. \]
Since $\sigma$ is $C^1$, we have that $\gamma$ is $C^2$ by Lemma~\ref{smoothnessofpotential}(ii).
Let $A$ be the maximum second partial of $\gamma$ in the ball $B  = B(x_0, 4\epsilon)$.  By the mean value theorem and Cauchy-Schwarz, for $x \in B$ we have
	\begin{align} |L(x) - \gamma(x)| &= |\langle \nabla \gamma(x_0) - \nabla \gamma(x_*), x-x_0 \rangle| \nonumber \\
	&\leq A \sqrt{d} |x_0 - x_*| |x-x_0| \leq c \epsilon^2, \label{secondordererror} \end{align}
where $c = 16 A \sqrt{d}$.  Hence
	\[ s(x)  \geq \gamma(x) \geq L(x) - c\epsilon^2, \qquad x \in B. \]
Thus the function
	\[ w = s - L + c\epsilon^2 \]
is nonnegative and superharmonic in $B$.  Write $w = w_0+w_1$, where $w_0$ is harmonic and equal to $w$ on $\partial B$.  Then since $s(x_0) = \gamma(x_0)$, we have
	\[ w_0(x_0) \leq w(x_0) = s(x_0) - L(x_0) + c\epsilon^2 = c \epsilon^2. \]
By the Harnack inequality, it follows that $w_0 \leq c' \epsilon^2$ on the ball $B' = B(x_0, 2\epsilon)$, for a suitable constant $c'$. 

Since $w_1$ is nonnegative and vanishes on $\partial B$, it attains its maximum in $\bar{B}$ at a point $x_1$ in the support of its Laplacian.  Since $\Delta w_1 = \Delta s$, by Lemma~\ref{majorantbasicprops}(iii) we have $s(x_1) = \gamma(x_1)$, hence
	\[ w_1(x_1) \leq w(x_1) = s(x_1) - L(x_1) + c\epsilon^2 \leq 2 c \epsilon^2, \]
where in the last step we have used (\ref{secondordererror}).  We conclude that $0 \leq w \leq (2c+c')\epsilon^2$ on $B'$ and hence $|s-L| \leq (c+c') \epsilon^2$ on $B'$.  Thus on $B'$ we have
	\[ |u| = |s-\gamma| \leq |s - L| + |\gamma-L| \leq (2c+c')\epsilon^2. \]
In particular, $u$ is differentiable at $x_0$, and $\nabla u(x_0)=0$.  Since $s$ is harmonic in $D$ and equal to $\gamma$ outside $D$, it follows that $u$ is differentiable everywhere, and $C^1$ off $\partial D$.

Given $y \in \partial D_\epsilon$, let $x_0$ be the closest point in $\partial D$ to $y$.  Since $B(y,\epsilon) \subset D$, by Lemma~\ref{majorantbasicprops}(iii) the function $w$ is harmonic on $B(y,\epsilon)$, so by the Cauchy estimate \cite[Theorem 2.4]{HFT} we have
	\begin{equation} \label{fromcauchyest} |\nabla s(y) - \nabla \gamma(x_0)| = |\nabla w(y) | \leq \frac{C}{\epsilon} \sup_{z\in B(y,\epsilon)} w(z). \end{equation}
Since $B(y,\epsilon) \subset B'$, the right side of (\ref{fromcauchyest}) is at most $C(2c+c')\epsilon$, hence
	\[ |\nabla u(y)| \leq |\nabla s(y) - \nabla \gamma(x_0)| +  |\nabla \gamma(y) - \nabla \gamma(x_0)| \leq C_0 \epsilon \]
where $C_0 = C(c+c') + A \sqrt{d}$.  Thus $u$ is $C^1$ on $\partial D$ as well.
\end{proof}

We now turn to the proof of Proposition~\ref{boundaryregularitysmooth}.  The first part of the proof follows Friedman \cite[Ch.\ 2, Theorem~3.5]{Friedman}.  It uses the Lebesgue density theorem, as stated in the next lemma.

\begin{lemma}
\label{density1}
Let $A \subset \R^d$ be a Lebesgue measurable set, and let
	\[ A' = \big\{x \in A \,|\, \liminf_{\epsilon \rightarrow 0} \epsilon^{-d} \Leb\big(B(x,\epsilon) \cap A^c\big) > 0 \big\}. \]
Then $\Leb(A')=0$.
\end{lemma}

To prove the first part of Proposition~\ref{boundaryregularitysmooth}, given a boundary point $x \in \partial D$, the idea is first to find a point $y$ in the ball $B(x,\epsilon)$ where $u=s-\gamma$ is relatively large, and then to argue using Lemma~\ref{smoothdensity} that a ball $B(y,c\epsilon)$ must be entirely contained in $D$.  Taking $A=\partial D$ in the Lebesgue density theorem, we obtain that $x \in A'$.

The proof of the second part of Proposition~\ref{boundaryregularitysmooth} uses Green's theorem in the form
	\begin{equation} \label{greensidentity} \int_{D'} (u\Delta h -  \Delta u h) dx = \int_{\partial D'} \left(u \frac{\partial h}{\partial\normal} - \frac{\partial u}{\partial \normal} h \right) dr. \end{equation}
Here $D'$ is the union of boxes $x^\Box$ that are contained in $D$, and $dx$ is the volume measure in $D'$, while $dr$ is the $(d-1)$-dimensional surface measure on $\partial D'$.  The partial derivatives on the right side of (\ref{greensidentity}) are in the outward normal direction from $\partial D'$.  

\begin{proof}[Proof of Propostion~\ref{boundaryregularitysmooth}]
(i) Fix $0<\lambda<1$.  For $x \in \partial D$ with $\sigma(x) \leq \lambda$, for small enough $\epsilon$ we have $\sigma \leq \frac{1+\lambda}{2}$ on $B(x,\epsilon)$.  By Lemma~\ref{laplacianofobstacle}(ii) the function
	\[ f(y) = \gamma(y) + \frac{1-\lambda}{2} |y|^2 \]
is superharmonic on $B(x,\epsilon) \cap D$, so by Lemma~\ref{majorantbasicprops}(iii) the function
	\begin{align*} w(y) &= u(y) - \frac{1-\lambda}{2} |x-y|^2 \\
					&=  s(y) - f(y) + (1-\lambda) \langle x,y \rangle - \frac{1-\lambda}{2} |x|^2 \end{align*}
is subharmonic on $B(x,\epsilon) \cap D$.  Since $w(x)=0$, its maximum is not attained on $\partial D$, so it must be attained on $\partial B(x,\epsilon)$, so there is a point $y$ with $|x-y|=\epsilon$ and
	\[ u(y) \geq \frac{1-\lambda}{2} \epsilon^2. \]

By Lemma~\ref{smoothdensity} we have $|\nabla u| \leq (1+c) C_0 \epsilon$ in the ball $B(y,c\epsilon)$.  Taking $c = \frac{1-\lambda}{4C_0}$,
we obtain for $z \in B(y,c\epsilon)$
	\begin{align*} u(z) &= u(y) + \langle \nabla u (y_*), z-y \rangle \\
				&\geq \frac{1-\lambda}{2} \epsilon^2 - c (1+c) C_0 \epsilon^2 > 0. \end{align*}
Thus for any $x \in \partial D \cap \{\sigma \leq \lambda\}$
	\[ \Leb(B(x,(1+c)\epsilon) \cap (\partial D)^c) \geq \omega_d (c \epsilon)^d. \]
By the Lebesgue density theorem it follows that 
	\begin{equation} \label{fromthedensitytheorem} \Leb(\partial D \cap \{\sigma \leq \lambda\}) = 0. \end{equation} 
By Lemma~\ref{startingdensitygreaterthan1} we have $\sigma \leq 1$ on $\partial D$.  Taking $\lambda \uparrow 1$ in (\ref{fromthedensitytheorem}), we obtain 
	\[ \Leb(\partial D) \leq \Leb(\partial D \cap \{\sigma < 1\}) + \Leb(\sigma^{-1}(1)) = 0. \]

(ii) Fix $\delta>0$ and let
	\[ D' = \bigcup_{x\in \delta \Z^d \,:\, x^\Box \subset D} x^\Box, \]
where $x^\Box = x + [-\frac{\delta}{2}, \frac{\delta}{2}]^d$.  By Lemmas~\ref{majorantbasicprops}(iii) and~\ref{smoothnessofpotential}(ii), in $D'$ we have
	\[ \Delta u = \Delta s - \Delta \gamma = 2d(1-\sigma). \]
Now by Green's theorem (\ref{greensidentity}), since $u$ is nonnegative and $h$ is superharmonic,
	\begin{equation} \label{greensthm} \int_{D'} (1-\sigma) h \,dx = \frac{1}{2d} \int_{D'} \Delta u h \,dx \leq \frac{1}{2d} \int_{\partial D'} \left(u \frac{\partial h}{\partial \normal} - \frac{\partial u}{\partial\normal}h \right) dr \end{equation}
where $\normal$ denotes the unit outward normal vector to $\partial D'$.
By Lemma~\ref{smoothdensity}, the integral on the right side is bounded by
	\begin{equation} \label{surfaceintegral} \int_{\partial D'} \left| u \frac{\partial h}{\partial \normal} \right| + \left| \frac{\partial u}{\partial\normal}h \right| dr \leq C_0 \sqrt{d} (\delta^2 ||\nabla h||_{\infty} + \delta ||h||_{\infty}) \Leb_{d-1}(\partial D'), \end{equation}
where $\Leb_{d-1}$ denotes $(d-1)$-dimensional Lebesgue measure.  Let
	\[ S = \bigcup_{x \in \delta \Z^d \,:\, x^\Box \cap \partial D \neq \emptyset} x^\Box. \]
Since $\Leb(\partial D)=0$, given $\epsilon>0$ we can choose $\delta$ small enough so that $\Leb(S)<\epsilon$.  Since $\partial D' \subset \partial S$, we have
	\[ \Leb_{d-1}(\partial D') \leq \Leb_{d-1}(\partial S) \leq \frac{2d\Leb(S)}{\delta}. \]
Since $D$ is open, $\Leb(D') \uparrow \Leb(D)$ as $\delta \downarrow 0$. Taking $\delta<\epsilon$ smaller if necessary so that $\Leb(D') \geq \Leb(D) - \epsilon$, we obtain from (\ref{greensthm}) and (\ref{surfaceintegral})
	\[ \int_{D} (1-\sigma)h\,dx \leq (M+1)\epsilon + C_0 \sqrt{d} (||h||_{\infty} + ||\nabla h||_{\infty}) \epsilon \]
where $M$ is the maximum of $|\sigma|$.  Since this holds for any $\epsilon>0$, we conclude that 	$\int_D h(x) dx \leq \int_D h(x) \sigma(x) dx$.
\end{proof}

\begin{proof}[Proof of Proposition~\ref{boundaryregularity}]
(i) Fix $\epsilon>0$.  Since $\sigma$ is continuous almost everywhere, there exist $C^1$ functions $\sigma_0 \leq \sigma \leq \sigma_1$ with $\int_{\R^d} (\sigma_1 - \sigma_0) dx < \epsilon$.  Scaling by a factor of $1+\delta$ for sufficiently small $\delta$, we can ensure that $\Leb(\sigma_i^{-1}(1))=0$, so that $\sigma_0$ and $\sigma_1$ satisfy the hypotheses of Proposition~\ref{boundaryregularitysmooth}.  For $i=0,1$ let 
	\[ \gamma_i(x) = -|x|^2 - G\sigma_i(x), \]
and let $s_i$ be the least superharmonic majorant of $\gamma_i$.  Choose $\alpha$ with $\lambda<\alpha<1$ such that $\Leb(\sigma_i^{-1}(\alpha))=0$ for $i=0,1$, and write 
	\begin{align*} &D_i = \{s_i>\gamma_i\}; \\ 
		&S_i = D_i \cup \{\sigma_i \geq \alpha\}. \end{align*}	
By (\ref{boundedawayfrom1again}) we have $\{\sigma_0 \geq \alpha \} \subset \{\sigma \geq 1\}^o$, hence by Lemma~\ref{monotonicity}
	\begin{align} \label{leftsqueeze} S_0 = D_0 \cup \{\sigma_0 \geq \alpha\} \subset D \cup \{\sigma \geq 1\}^o &= \widetilde{D} \\  
	\label{rightsqueeze} &\subset \overline{\widetilde{D}} \subset \overline{D_1 \cup \{\sigma \geq 1\}} = \overline{S_1}. \end{align}
For $i=0,1$ write
	\[ \sigma_i^\circ (x)  =  \begin{cases} 1, & x \in D_i \\ \sigma_i(x), & x \not\in D_i. \end{cases} \]
By Proposition~\ref{boundaryregularitysmooth}(ii) with $h=1$, we have
	\begin{equation} \label{conservation} \int_{\R^d} \sigma_i^\circ = \Leb(D_i) + \int_{D_i^c} \sigma_i = \int_{\R^d} \sigma_i. \end{equation}
For $0< \alpha_0 < \alpha$, we have
	\begin{equation} \label{thisissmall} \Leb(S_1 - S_0) \leq \Leb \big( S_1 \cap \{ \sigma_0^\circ \leq \alpha_0\} \big) + \Leb \big( \{ \alpha_0 < \sigma_0^\circ < \alpha \} \big). \end{equation}
Since $\sigma_1^\circ \geq \alpha$ on $S_1$, the first term is bounded by
	\[ \Leb \big( S_1 \cap \{ \sigma_0^\circ \leq \alpha_0\} \big) \leq \frac{|| \sigma_1^\circ - \sigma_0^\circ ||_1}{\alpha-\alpha_0} = \frac{||\sigma_1 - \sigma_0||_1}{\alpha-\alpha_0} < \frac{\epsilon}{\alpha-\alpha_0}, \]
where the equality in the middle step follows from (\ref{conservation}).

By Proposition~\ref{boundaryregularitysmooth}(i) we have $\Leb(\partial D_1) = 0$, hence $\Leb(\partial S_1) = 0$.  Taking $\alpha_0 = \alpha - \sqrt{\epsilon}$ we obtain from (\ref{leftsqueeze}), (\ref{rightsqueeze}), and (\ref{thisissmall})
	 \[ \Leb \big(\partial \widetilde{D} \big) \leq \Leb \big(\overline{S_1} - S_0\big) = \Leb(S_1 - S_0) < \sqrt{\epsilon} + \Leb \big( \{ \alpha -\sqrt{\epsilon} < \sigma_0 < \alpha \} \big). \]
Since this holds for any $\epsilon>0$, the result follows.

(ii) Write
	\[ \sigma^\circ (x)  =  \begin{cases} 1, & x \in D \\ \sigma(x), & x \not\in D. \end{cases} \]
From (\ref{conservation}) we have
	\[ \int_{\R^d} \sigma_0 = \int_{\R^d} \sigma_0^\circ \leq \int_{\R^d} \sigma^\circ \leq \int_{\R^d} \sigma_1^\circ = \int_{\R^d} \sigma_1. \]
The left and right side differ by at most $\epsilon$.  Since $\int_{\R^d} \sigma_0 \leq \int_{\R^d} \sigma \leq \int_{\R^d} \sigma_1$, and $\epsilon>0$ is arbitrary, it follows that $\int_{\R^d} \sigma = \int_{\R^d}\sigma^\circ$, hence
	\begin{equation} \label{conservationinD}  \int_D \sigma = \int_{\R^d} \sigma^\circ - \int_{D^c} \sigma = \Leb(D). \end{equation}
	
(iii) By Lemma~\ref{startingdensitygreaterthan1}, if $\sigma(x)>1$ and $x \notin D$, then $\sigma$ is discontinuous at~$x$.  Thus $\sigma=1$ almost everywhere on $\widetilde{D}-D$, and we obtain from (\ref{conservationinD})
	\[ \int_{\widetilde{D}} \sigma = \Leb(D) + \int_{\widetilde{D}-D} \sigma = \Leb(\widetilde{D}). \qed \]
\renewcommand{\qedsymbol}{}
\end{proof}

\subsection{Convergence of Obstacles, Majorants and Domains}

\begin{proof}[Proof of Lemma~\ref{threesteps}]
(i) Since $\sigma,\sigma_n$ are supported on $B$, we have for $x \in \R^d$
	\begin{equation} \label{triangleineq} | G\sigma_n(x) - G\sigma(x) | \leq \int_B |g(x,y)| |\sigma_n(y)-\sigma(y)| dy. \end{equation}
Fix $0<\epsilon<1$.  Since $0 \leq \sigma,\sigma_n \leq M$, we have by (\ref{ballpotential}) and (\ref{ballpotentialdim2})
	\begin{equation} \label{tinyball} \int_{B(x,\epsilon)} |g(x,y)| |\sigma_n(y) - \sigma(y)| dy \leq 
	\begin{cases} 
	2M\epsilon^2 \left(2 \log \frac{1}{\epsilon} + 1\right), &d=2 \\
	\frac{2d}{d-2} M\epsilon^2, & d\geq 3. 
	\end{cases} 
\end{equation}
Now let $B_0 = B-B(x,\epsilon)$.  Since $\sigma_n \to \sigma$ in $L^1$, we can take~$n$ large enough so that $\int_B |\sigma_n-\sigma| < \epsilon^{d-1}$, obtaining
	\begin{align*} \int_{B_0} |g(x,y)| |\sigma_n(y) - \sigma(y)| &< \epsilon^{d-1} \omega_d R^d \sup_{y \in B_0} |g(x,y)| \\
	&\leq
	\begin{cases} 
	2R^2 \epsilon \max(\log \frac{1}{\epsilon}, \log(R+|x|)), &d=2; \\
	\frac{2}{d-2} R^d \epsilon, & d\geq 3. 
	\end{cases} 
	\end{align*}
Thus the right side of (\ref{triangleineq}) can be made arbitrarily small for sufficiently large $n$.

(ii) By Lemmas~\ref{monotonicity} and~\ref{relaxingaball}, the noncoincidence sets $D,D_{(n)}$ are contained in the ball $B_1 = M^{1/d} B$.  Given $\epsilon>0$ and a compact set $K$ containing $B_1$, choose $N$ large enough so that $|\gamma_n-\gamma|<\epsilon$ on $K$ for all $n\geq N$.  Then
	\[ s+\epsilon \geq \gamma + \epsilon > \gamma_n \]
on $K$, so the function $f_n= $ max$(s+\epsilon,\gamma_n)$ is superharmonic on $K$.  By Lemma~\ref{majorantonacompactset} we have $f_n \geq s_n$, and hence $s + \epsilon \geq s_n$ on $K$.  Likewise
	\[ s_n + \epsilon \geq \gamma_n + \epsilon > \gamma \]
on $K$, so the function $\tilde{f}_n =$ max$(s_n+\epsilon,\gamma)$ is superharmonic on $K$.  By Lemma~\ref{majorantonacompactset} we have $\tilde{f}_n \geq s$ and hence $s_n + \epsilon \geq s$ on $K$.  Thus $s_n \rightarrow s$ uniformly on $K$.

(iii) Let $\beta>0$ be the minimum value of $s-\gamma$ on $\overline{D_\epsilon}$, and choose $N$ large enough so that $|\gamma-\gamma_n|, |s-s_n| < \beta/2$ on $\overline{D}$ for all $n \geq N$.  Then
	\[ s_n - \gamma_n > s- \gamma - \beta > 0 \]
on $D_\epsilon$, hence $D_\epsilon \subset D_{(n)}$ for all $n \geq N$.
\end{proof}

\begin{proof}[Proof of Lemma~\ref{atmostquadratic}]
Consider the functions
	\begin{equation} \label{quadraticcorrection} f_\pm(y) = f(y) \pm \lambda |y|^2. \end{equation}
By hypothesis, $f_+$ is subharmonic and $f_-$ is superharmonic on $B(o,R)$.  Given $x \in B(o,\beta R)$, let $r =$ min$(R,2|x|)$, and let $h_\pm$ be the harmonic function on the ball $B=B(o,r-1)$ which agrees with $f_\pm$ on $\partial B$.  Then 
	 \begin{equation} \label{f_+upperbound} f_+ \leq h_+ \leq h_- + 2 \lambda r^2 \leq f_- + 2\lambda r^2. \end{equation}
By the Harnack inequality \cite[Theorem 1.7.2]{Lawler} there is a constant $\tilde{c}_\beta$ such that
	\[ h_+(x) \leq \tilde{c}_\beta h_+(o). \]
Since $f_-(o)=f(o)=0$ we have from (\ref{f_+upperbound})
	\[ h_+(o) \leq f_-(o) + 2 \lambda r^2 \leq 8 \lambda |x|^2. \]
hence 
	\[ f(x) \leq h_+(x) - \lambda |x|^2 \leq c_\beta \lambda |x|^2 \] 
with $c_\beta = 8 \tilde{c}_\beta -1$.
\end{proof}

Lemma~\ref{continuumatmostquadratic} is proved in the same way as Lemma~\ref{atmostquadratic}, using the continuous Harnack inequality in place of the discrete one, and replacing $|y|^2$ by $|y|^2/2d$ in (\ref{quadraticcorrection}).

\begin{proof}[Proof of Lemma~\ref{pointset}]
Let $K$ be the closure of $(A \cup B)_\epsilon$, and let $Y_n = A_{1/n} \cup B_{1/n}$.  Since $K$ is contained in $\bigcup Y_n = A\cup B$, the sets $K\cap Y_n$ form an open cover of $K$, which has a finite subcover, i.e.\ $K\subset Y_n$ for some $n$.
\end{proof}

\subsection{Discrete Potential Theory}

\begin{proof}[Proof of Lemma~\ref{thirdderiv}]
By Taylor's theorem with remainder
	\[ f(x+\delta_n e_i) - 2f(x) + f(x-\delta_n e_i) = \frac{\partial^2f}{\partial x_i^2}\delta_n^2 + \frac16 \frac{\partial^3 f}{\partial x_i^3}(x+te_i) \delta_n^3 - \frac16 \frac{\partial^3 f}{\partial x_i^3}(x-ue_i) \delta_n^3 \]
for some $0\leq t,u\leq 1$.  Summing over $i=1,\ldots,d$ and dividing by $\delta_n^2$ gives the result.
\end{proof}

\begin{proof}[Proof of Lemma~\ref{greensfunctionconvergence}]
In dimensions $d\geq 3$ we have from (\ref{definitionofg_n}) and the standard estimate for the discrete Green's function \cite{Uchiyama} (see also \cite[Theorem 1.5.4]{Lawler})
	\begin{align*} g_n(x,y) &= \delta_n^{2-d} \left( a_d \left(\frac{|x-y|}{\delta_n}\right)^{2-d} + O\left(\frac{|x-y|}{\delta_n}\right)^{-d}\right) \\
		&= a_d |x-y|^{2-d} + O(\delta_n^2 |x-y|^{-d}). \end{align*}
Likewise, in dimension two, using the standard estimate for the potential kernel \cite{FU} (see also \cite[Theorem 1.6.2]{Lawler}) in (\ref{definitionofg_ndimension2}) gives
	\[ g_n(x,y) =  -\frac{2}{\pi} \log \frac{|x-y|}{\delta_n} + \frac{2}{\pi} \log \delta_n + O\left(\frac{|x-y|}{\delta_n}\right)^{-2}. \qed \]
\renewcommand{\qedsymbol}{}
\end{proof}


\begin{proof}[Proof of Lemma~\ref{greensintegralconvergencecont}]
Let $K \subset \R^d$ be compact.  By the triangle inequality
	\[ |(G_n\sigma_n)^\Box - G\sigma| \leq |(G_n \sigma_n)^\Box - G \sigma_n^\Box| + |G\sigma_n^\Box - G\sigma|. \]
By Lemma~\ref{threesteps}(i) the second term on the right side is $<\epsilon/2$ on $K$ for sufficiently large $n$.  The first term is at most
	\begin{align} |(G_n \sigma_n)^\Box(x) - G \sigma_n^\Box(x)| 
	&\leq \sum_{y \in \delta_n\Z^d} \sigma_n(y) \left| \delta_n^d g_n(x^\Points,y) - \int_{y^\Box} g(x,z) dz \right| \nonumber \\
	&\leq M \sum_{y \in B^\Points} \int_{y^\Box} |g_n(x^\Points,y) - g(x,z)| dz \nonumber \\
	&= M \int_{B^{\Points\Box}} |g_n(x^\Points,z^\Points) - g(x,z)| dz. \label{gotitwherewewantit} \end{align}
By Lemma~\ref{greensfunctionconvergence}, we have
	\begin{align*} |g_n(x^\Points,z^\Points) - g(x,z)| &\leq 
		|g_n(x^\Points,z^\Points)-g(x^\Points,z^\Points)| + |g(x^\Points,z^\Points)-g(x,z)| \\
			&\leq C\delta_n|x-z|^{1-d} \end{align*}
for a constant $C$ depending only on $d$.  Integrating (\ref{gotitwherewewantit}) in spherical shells about $x$, we obtain
		\[ |(G_n \sigma_n)^\Box - G \sigma_n^\Box| \leq Cd\omega_dMR\delta_n, \]
where $R$ is the radius of $B$.  Taking $n$ large enough so that the right side is $<\epsilon/2$, the proof is complete.
\end{proof}

\begin{proof}[Proof of Lemma~\ref{discretemajorantonacompactset}]
Let $f$ be any function which is superharmonic on $\Omega$ and $\geq \gamma_n$.  Since $s_n$ is harmonic on $D_n$, the difference $f-s_n$ is superharmonic on $D_n$ and attains its minimum in $D_n \cup \partial D_n$ on the boundary.  Hence $f-s_n$ attains its minimum in $\Omega \cup \partial \Omega$ at a point $x$ where $s_n(x)=\gamma_n(x)$.  Since $f \geq \gamma_n$ we conclude that $f \geq s_n$ on $\Omega$ and hence everywhere.  Thus $s_n$ is at most the infimum in (\ref{discretemajorantinadomain}).
Since the infimum in (\ref{discretemajorantinadomain}) is taken over a strictly larger set than that in (\ref{thediscretemajorant}), the reverse inequality is trivial.
\end{proof}

\section{Open Problems}

We conclude by mentioning two simple variants of aggregation models which appear to have interesting limit shapes, but about which we cannot prove anything at present.  In the first variant, particles start at the origin and perform rotor-router walks in $\Z^2$ until either reaching an unoccupied site or hitting the positive $x$-axis.  The resulting set of occupied sites is shown on the left in Figure~\ref{cardioidandspiral}.   In the second variant, shown on the right, particles again start at the origin and perform rotor-router walks until reaching an unoccupied site, except that the rotors on the positive $x$-axis do not rotate but instead always point downward.


\begin{figure}
\centering
\includegraphics[scale=.185]{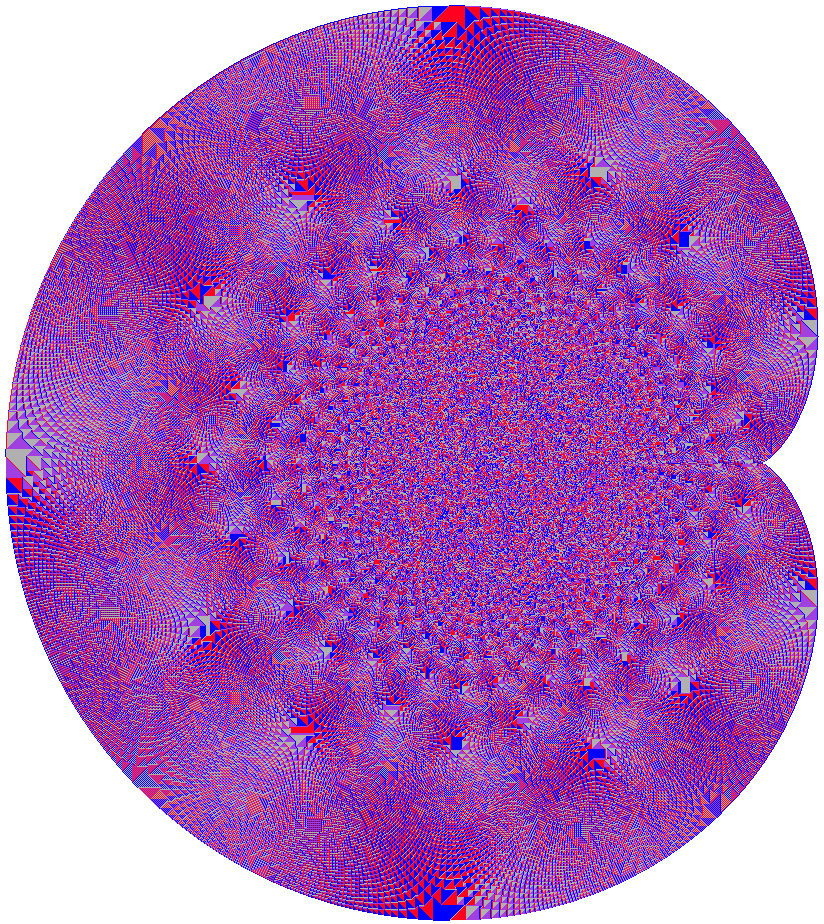} ~
\includegraphics[scale=.37]{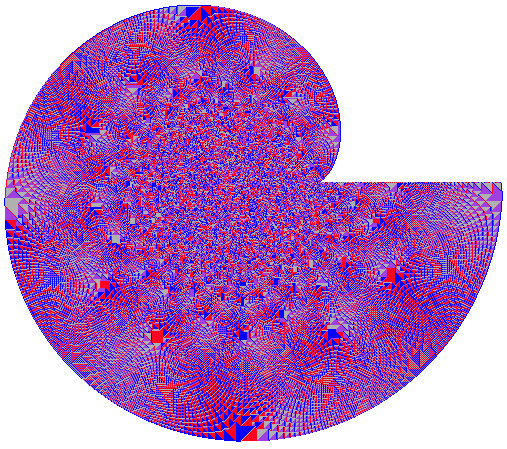}
\caption{Modified rotor-router aggregation started from a point source at the origin in $\Z^2$.  In the left figure, particles are killed on hitting the positive $x$-axis.  In the right figure, sites on the positive $x$-axis send all particles in the downward direction.}
\label{cardioidandspiral}
\end{figure}

Our simulations indicate that limiting shapes exist for both of these variants.  Pavel Etingof (personal communication) has found exact solutions to the analogous Hele-Shaw flow problems in the continuum.  The challenge is to extend our methods to prove that the lattice models produce shapes converging to these Hele-Shaw solutions.

\section*{Acknowledgments}

Jim Propp was the first to suggest that the Diaconis-Fulton smash sum should have a well-defined deterministic scaling limit.  Thanks also to Jim for bringing the rotor-router model to our attention, and for many fruitful discussions.  Matt Cook's work on the rotor-router model inspired our approach using the odometer function.  In addition, we thank Misha Sodin for suggesting that we use the least superharmonic majorant, Craig Evans for pointing us to helpful references on the obstacle problem, and Darren Crowdy for a helpful discussion on quadrature domains.  We also had useful discussions with Scott Armstrong, Oded Schramm, and Scott Sheffield.  Itamar Landau helped create several of the figures.  We thank David Jerison for pointing out an error in the proof of Lemma~\ref{coreoftheargument} in an earlier draft.








\end{document}